\documentclass{article}%
\usepackage{authblk}
\usepackage{amsmath}%
\usepackage{amssymb}%
\usepackage{stix}
\usepackage{graphicx}
\usepackage{hyperref}
\usepackage[usenames,dvipsnames]{color}
\providecommand{\U}[1]{\protect\rule{.1in}{.1in}}
\newtheorem{theorem}{Theorem}

\newtheorem{condition}[theorem]{Condition}

\newtheorem{corollary}[theorem]{Corollary}

\newtheorem{definition}[theorem]{Definition}

\newtheorem{lemma}[theorem]{Lemma}

\newtheorem{proposition}[theorem]{Proposition}
\newtheorem{remark}[theorem]{Remark}

\newenvironment{proof}[1][Proof]{\noindent\textbf{#1.} }{\ \rule{0.5em}{0.5em}}

\newcommand{\B}{\mathbb{B}}

\newcommand{\E}{\mathbb{E}}

\newcommand{\N}{\mathbb{N}}

	\renewcommand{\P}{\mathbb{P}}

\newcommand{\R}{\mathbb{R}}

\newcommand{\wh}{\widehat}
\newcommand{\wt}{\widetilde}
\newcommand{\ovl}{\overline}
\newcommand{\ep}{\epsilon}
\newcommand{\cA}{\mathcal{A}}

\newcommand{\cD}{\mathcal{D}}

\newcommand{\cF}{\mathcal{F}}
\newcommand{\cG}{\mathcal{G}}

\newcommand{\cL}{\mathcal{L}}
\newcommand{\cM}{\mathcal{M}}
\newcommand{\cN}{\mathcal{N}}

\newcommand{\bx}{\mathbf{x}}

\newcommand{\bp}{\mathbf{p}}
\newcommand{\bP}{\mathbf{P}}
\newcommand{\bE}{\mathbf{E}}

\begin{document}
\title{Coagulation dynamics under environmental noise: \\
Scaling limit to SPDE
}
\author{Franco Flandoli and Ruojun Huang}
\maketitle

\begin{abstract}
We prove that a system of locally interacting diffusions carrying discrete masses, subject to an environmental noise and undergoing mass coagulation, converges to a system of Stochastic Partial Differential Equations (SPDEs) with Smoluchowski-type nonlinearity. Existence, uniqueness and regularity of the SPDEs are also proven.
\end{abstract}

\section{Introduction}
Environmental noise in deterministic or stochastic interacting particle systems is a space-dependent noise acting on all particles, opposite to the more common independent noise for each particle. For particles in a fluid, the environmental noise may be an idealized description of a turbulent fluid. An example which motivates the model studied in the present work are the small rain droplets of which clouds are made. Droplets move in the cloud due to the force exerted on them by the surrounding turbulent air, and coagulate when they become sufficiently close to each other. We introduce a particle system modeling these two phenomena and investigate its scaling limit to a continuous density model, which is a stochastic Smoluchowski system. Opposite to independent noise for each particle which becomes a Laplacian in the scaling limit, the environmental noise yields a stochastic transport term in the continuous limit.
Raindrop formation in turbulent fluids has been studied in the Physics literature \cite{ST, FFS, Bode, PW}; our paper provides foundational results on the particle system viewpoint and its continuum limit to a Stochastic Partial Differential Equation (SPDE).

We model the individual rain droplets as diffusions on $\mathbb{R}^{d}$,
$d\geq1$, with a small molecular diffusivity, and subject to a
\textit{{common}} Stratonovich transport-type noise. Any pair of particles has
a propensity to coagulate into one, combining their masses, when their
positions get locally close to each other. Without the common noise, this is
in the spirit of classical Smoluchowski coagulation model which leads to his
famous PDEs \cite{Smo, Smo2}, mathematically derived before in the kinetic limit from interacting particle systems in \cite{LN, HR, HR3}. We prove in this work that as the total number of particles in our
system tends to infinity, the empirical measures as indexed by mass parameter,
converge to a system of SPDEs with the same Stratonovich transport-type noise.
From Stratonovich-to-It\^{o} correction, we obtain an extra second-order
divergence-form operator, the \textquotedblleft eddy
diffusion\textquotedblright. This opens the door to the investigation of
diffusion and coagulation enhancement, along the lines of \cite{Fla11,
Flandoli LNM, FGP, DFV, Ga, FL18, FGL} and references therein (also for other
independent works on mixing or diffusion enhancement). We stress that the
philosophy underlying the emergence of enhanced diffusion starting from a
transport-type noise and through a suitable scaling limit was first discovered
by Galeati \cite{Ga}. Numerical results are recently obtained for a closely
related coagulation model in \cite{AP}.

Given a filtered probability space $(\Omega,\cF,\{\cF_t\},\P)$, for each $N\in\N$ consider an interacting particle system in $\R^d$, $d\ge 1$, consisting initially of $N(0)=N$ particles. Each particle $i$ has a position $x_i^N(t)\in\R^d$ and carries an integer mass $m_i^N(t)$, which takes values in
\begin{align*}
\mathbb S:=\{1,2,...,M,\emptyset\}, 
\end{align*}
where $M\in\N$ is the largest possible mass, and $\emptyset$ is a fictitious element. The particles are subject to pairwise coagulation, at which time some particles may cease to be active in the system. So long as particle $i$ is still active (how the index set changes when a coagulation event happens will be explained below), its position  $x_i^N(t)$ obeys the SDE:
\begin{align}\label{SDE}
dx_i^N(t)=\sum_{k\in K}\sigma_k\left(x_i^N(t)\right)\,\circ\,dW^k_t+ \lambda \, d\beta_i(t),\quad i\in\cN(t)
\end{align}
where $\circ$ denotes Stratonovich integration, scalar $\lambda>0$, $\cN(t)\subset\{1,...,N\}$ denotes the set of indices of active particles at time $t$, whose cardinality $|\cN(t)|=N(t)\le N$, $K$ is a {\it{finite}} set, $\{W_t^k\}_{k\in K}$ is a given finite collection of independent standard Brownian motions in $\R$, and $\{\beta_i(t)\}_{i=1}^\infty$ are given independent standard Brownian motions in $\R^d$, $\{\sigma_k(x)\}_{k\in K}$ are given divergence-free vector field of class $C_b^\infty(\R^d;\R^d)$, where $C_b^\infty(\cdot)$ denotes the space of smooth functions with derivatives of all orders uniformly bounded. Following \cite{CF}, we call 
\begin{align*}
\frac{d\mathscr W}{dt}(t,x):=\sum_{k\in K}\sigma_k(x)\frac{dW^k_t}{dt}
\end{align*}
the environmental noise or common noise acting simultaneously on all particles. 
We denote the $d\times d$ spatial covariance matrix of $\mathscr W$ by
\begin{align*}
Q(x,y):=\sum_{k\in K}\sigma_k(x)\otimes\sigma_k(y).
\end{align*} 
As the set $K$ is finite, we do not assume $Q$ is spatially homogeneous as in some prior studies using such transport noise.

From Stratonovich to It\^o, \eqref{SDE} can be equivalently written as 
\begin{align*}
dx_i^N(t)=\sum_{k\in K}\sigma_k\left(x_i^N(t)\right)dW^k_t+\frac{1}{2}\sum_{k\in K}\left(\nabla\sigma_k\cdot\sigma_k\right)\left(x_i^N(t)\right)dt
+ \lambda\,  d\beta_i(t),\quad i\in\cN(t)
\end{align*}
where component-wise, 
\begin{align}\label{strat-corr}
\left(\nabla\sigma_k\cdot\sigma_k\right)^\alpha(x):=\sum_{\beta=1}^d\sigma_k^\beta(x)\partial_\beta\sigma_k^\alpha(x), \quad \alpha=1,...,d.
\end{align}

The initial masses $\{m_i(0)\}_{i=1}^N$ are chosen i.i.d. from $\{1,...,M\}$ with probability $r_m$ to be mass $m$, so that $\sum_{m=1}^M r_m=1$; and once the masses are determined, the distributions of $x_i(0), i=1,2,..$ are independent with density $p_{m_i(0)}(x), i=1,2,..$, for some given probability density functions $p_m(x): \R^d\to\R_+$, $m=1,...,M$, all of which satisfy
\begin{condition}\label{cond:init}
\begin{enumerate}
\item
They are compactly supported in the Euclidean ball $\B(0,R)$ of finite radius $R$ centered at the origin; 
\item
They are uniformly bounded above by some finite constant $\Gamma$, i.e. $\|p_m\|_\infty\le \Gamma$;
\item
There exists some integer $n>d/4$ such that
\[
p_m(x)\in W^{2n,2}(\R^d),
\]
where $W^{k,p}(\R^d)$ are standard Sobolev spaces.
\end{enumerate}
\end{condition}
Note that the initial conditions for different $N\in\N$ are naturally coupled together, hence we do not stress the dependence of $x_i(0), m_i(0)$ on $N$. 
The probability space is endowed with the canonical filtration
\[
\cF_t:=\sigma\left\{\{m_i(0)\}_{i=1}^\infty, \{x_i(0)\}_{i=1}^\infty,\{\beta_i(s)\}_{i=1}^\infty, \{W^k_s\}_{k\in K}: s\in[0,t]\right\}, \quad t\ge0.
\]

The interaction between particles is by means of coagulation of masses, heuristically described below. It was studied, without the environmental noise, by Hammond and Rezakhanlou in \cite{HR, HR3} (however the limit there is a system of PDEs instead of SPDEs, among other differences). Let $\theta:\R^d\to\R_+$ be nonnegative, of class $C^\alpha(\R^d)$ (the space of H\"older continuous functions on $\R^d$) for some $\alpha\in(0,1)$, compactly supported in $\B(0,C_0)$ for some finite constant $C_0$, with $\theta(0)=0$ and  $\int_{\R^d} \theta =1$, where we denote by $\B(x,r)$ the open Euclidean ball of radius $r$ centered at $x\in\R^d$. For every $\ep>0$, we denote 
\begin{align*}
\theta^\ep(x):=\ep^{-d}\theta(\ep^{-1}x). 
\end{align*} 
Throughout we impose the following relation between $N$ and $\ep$:
\begin{equation}\label{local}
\begin{aligned}
\ep=\ep(N)\to0 \quad &\text{as}\quad  N\to\infty, \\
\limsup_{N\to\infty}\frac{\ep^{1-d}}{N}<\infty, \quad d\ge 2; \quad &
\limsup_{N\to\infty}\frac{|\log\ep|}{N}<\infty, \quad d=1.
\end{aligned}
\end{equation}
This regime includes local interaction (when $\ep^{-d}\asymp N$), see discussion below, which is the scaling of interest for modelling raindrop formations. A configuration of the stochastic system is a finite string of positions and masses  (whose length is at most $2N$)
\begin{align*}
\eta=(x_1,x_2,...; m_1,m_2,..)\in (\R^d)^N\times \mathbb S^N.
\end{align*}
On top of individual diffusions \eqref{SDE}, we impose the following rule of mass coagulation. Suppose the current configuration is $\eta$. For each pair of particles $(i,j)$ active in the system (i.e. $m_i,m_j\neq \emptyset$), with positions $(x_i,x_j)$ and masses $(m_i,m_j)$, where $i\neq j$ each ranges over the index set, with a rate 
\begin{align}\label{range}
\frac{1}{N}\theta^\ep(x_i-x_j)
\end{align}
we remove both particles from the system, and then add a new particle with mass $(m_i+m_j)$ at either $x_i$ or $x_j$ chosen with probability $\frac{m_i}{m_i+m_j}$ for the former (and $\frac{m_j}{m_i+m_j}$ for the latter). However, we do this only if $m_i+m_j\le M$, otherwise after the pair is removed, no new particle is added to the system, or equivalently we add a fictitious particle with mass $\emptyset$ at the origin. We denote the new configuration obtained from $\eta$ this way by $S_{i,j}^1\eta$, and respectively $S_{i,j}^2\eta$. For labelling purposes, in the new configuration $S_{i,j}^1\eta$, we call particle $i$ the new mass-combined particle found at position $x_i$, whereas there is no longer a particle $j$ in the system; and analogously for $S_{i,j}^2\eta$. We will only be tracking the empirical measures of particles with mass $\le M$. The interpretation is that, in terms of raindrop formations, when the mass of a raindrop exceeds a certain threshold, it falls. 

\medskip
For $N$ diffusion particles evolving in a unit-order volume in $\R^d$, the typical inter-particle distance is on the order of $N^{-1/d}$. With $\ep$ representing the typical length scale over which pairs of particles can interact, see \eqref{range}, the relation \eqref{local} includes the most relevant case $\ep\asymp N^{-1/d}$ for our modelling of raindrops formation, namely when the interaction is so-called ``local''. In such a regime, each particle typically interacts with a bounded number of others at any given time, which is an analog of the nearest-neighbor, or finite-range, interactions widely studied in the discrete setting, e.g. on lattices \cite{DMP, KL}. If $\ep^{-d}\ll N$, then the interaction is no longer local, but more spread-out, in particular when $\ep$ is independent of $N$ it is the ``mean-field'' case, and the regime between mean-field and local we can term ``moderate'' following Oelschläger. In those regimes, each particle interacts with a diverging (in $N$) number of others, see \cite{FLO, FLR, FH} for more discussion on local versus moderate or mean-field interactions for diffusion systems. In the latter cases techniques are much more developed, see e.g. \cite{Oe85, Oe, Sz, MC, JM, FLO}. Also in those regimes, particle systems subject to environmental noise, with even ``singular'' interactions, have been studied, see \cite{CF, FL, GL} among others, where the interaction can occur in the drift of the SDEs and not just in auxiliary variables (like mass). Localizing the range of interaction for diffusions is a non-trivial task, in particular, in this work (as well as in \cite{FH}) we have to utilize some techiniques developed in \cite{HR, HR3} in the spirit of the classical It\^o-Tanaka trick, to handle the convergence of the nonlinear terms, see Section \ref{ito-tanaka}. In those papers \cite{HR, HR3}, a system of Smochulowski PDEs is rigorously derived from interacting particles in the so-called ``mean-free path'' regime, which corresponds to diluted gases (not covered by our result). Some of the other classical works on diffusions with local interactions (that occur in the drift) are \cite{Var, OV, OVY, Uch}.

Now we can formally give the infinitesimal generator of the system 
\begin{align*}
\cL^NF(\eta):=\cL^N_DF(\eta)+\cL^N_JF(\eta)
\end{align*}
where the first, diffusion part of the generator is
\begin{equation}\label{diff-gen}
\begin{aligned}
&\cL^N_DF(\eta)\\
&:=\frac{\lambda^2}{2}\sum_{i\in\cN(\eta)}\Delta_{x_i} F(\eta)+\frac{1}{2}\sum_{i,j\in\cN(\eta)}\sum_{\alpha,\beta=1}^d\frac{\partial^2F}{\partial x_i^\alpha\partial x_j^\beta}(\eta)\sum_{k\in K}\sigma_k^\alpha(x_i)\sigma_k^\beta(x_j)+\frac{1}{2}\sum_{i\in\cN(\eta)}\sum_{\alpha,\beta=1}^d\partial_{x_i^\alpha}F(\eta)\sum_{k\in K}\sigma_k^\beta(x_i)\partial_\beta\sigma_k^\alpha(x_i)\\
&=\frac{\lambda^2}{2}\sum_{i\in\cN(\eta)}\Delta_{x_i} F(\eta)+\frac{1}{2}\sum_{i,j\in\cN(\eta)}\sum_{\alpha,\beta=1}^d\frac{\partial^2F}{\partial x_i^\alpha\partial x_j^\beta}(\eta)Q^{\alpha\beta}(x_i,x_j)+\frac{1}{2}\sum_{i\in\cN(\eta)}\sum_{\alpha,\beta=1}^d\partial_{x_i^\alpha}F(\eta)\partial_{x_i^\beta}\left(Q^{\alpha\beta}(x_i,x_i)\right)
\end{aligned}
\end{equation}
using $\operatorname{div}\sigma_k(x)=0$ in the last line, 
and $\cN(\eta)$ denotes the set of indices of active particles in configuration $\eta$; and the second, jump (or coagulation) part of the generator is
\begin{align}\label{jump-gen}
\cL^N_JF(\eta):=\frac{1}{N}\sum_{i\neq j\in\cN(\eta)}\theta^\ep(x_i-x_j)\left[\frac{m_i}{m_i+m_j}F(S_{i,j}^1\eta)+\frac{m_j}{m_i+m_j}F(S_{i,j}^2\eta)-F(\eta)\right].
\end{align}
We henceforth denote by $\eta(t)$ the random configuration at time $t$. Note that if at time $t$, a coagulation event happens for some pair of particles, the cardinality of active particles $N(t)$ decreases either by one or by two.

For each $N\in\N$ and $1\le m\le M$, we denote the empirical measure of mass-$m$ particles
\begin{align}\label{empirical}
\mu^{N,m}_t(dx):=\frac{1}{N}\sum_{i\in\cN(t)}\delta_{x_i^N(t)}(dx)1_{\{m_i^N(t)=m\}}, \quad t\ge 0.
\end{align}
For each $t$, these are nonnegative random measures on $\R^d$ with total mass bounded above by $1$. Denote by $\cM_{+,1}(\R^d)$ the space of subprobability measures on $\R^d$ endowed with the weak topology, and by $\cD([0,T];\cM_{+,1}(\R^d))\equiv \cD_T(\cM_{+,1})$ the space of c\`adl\`ag functions taking values in $\cM_{+,1}(\R^d)$, endowed with the Skorohod topology. Then, $\{\mu^{N,m}_t:t\in[0,T]\}_{m\le M}$ are $\cD_T(\cM_{+,1})^M$-valued random variables.

Let us also introduce a finite system of SPDEs
\begin{align}\label{spde}
\begin{cases}
d u_m(t,x)&=\frac{1}{2}\lambda^2\Delta u_m(t,x)dt+\sum_{n=1}^{m-1}u_n(t,x)u_{m-n}(t,x)dt-2\sum_{n=1}^Mu_n(t,x)u_m(t,x)dt\\\\
&+\frac{1}{2}\operatorname{div}\left(Q(x,x)\nabla u_m(t,x)\right)dt
+\sum_{k\in K}\sigma_k(x)\cdot \nabla u_m(t,x) \; dW_t^k,\quad (t,x)\in[0,T]\times\R^d, \\\\
u_m(0,x)&=r_mp_m(x), \quad m=1,...,M
\end{cases}
\end{align}
where by convention, when $m=1$, set $\sum_{n=1}^{m-1}[...]=0$. For our notion of (analytically) weak solutions of the system \eqref{spde}, see Definition \ref{def:spde-w}. The SPDE can be equivalently written as 
\begin{align*}
\begin{cases}
d u_m(t,x)&=\frac{1}{2}\lambda^2\Delta u_m(t,x)dt+\sum_{n=1}^{m-1}u_n(t,x)u_{m-n}(t,x)dt-2\sum_{n=1}^Mu_n(t,x)u_m(t,x)dt\\\\
&+\sum_{k\in K}\sigma_k(x)\cdot \nabla u_m(t,x) \,\circ\, dW_t^k,\quad (t,x)\in[0,T]\times\R^d, \\\\
u_m(0,x)&=r_mp_m(x), \quad m=1,...,M
\end{cases}
\end{align*}
since by $\operatorname{div}\sigma_k(x)=0$, the Stratonovich-to-It\^o correction takes the form
\begin{align*}
&\frac{1}{2}\sum_{k\in K}\sigma_k(x)\cdot\nabla\left(\sigma_k(x)\cdot\nabla u_m(t,x)\right)dt=\frac{1}{2}\sum_{k\in K}\sum_{\alpha,\beta=1}^d\sigma_k^\alpha(x)\partial_\alpha\left(\sigma_k^\beta(x)\partial_\beta u_m(t,x)\right)dt\\
&=\frac{1}{2}\sum_{k\in K}\sum_{\alpha,\beta=1}^d\partial_\alpha\left(\sigma_k^\alpha(x)\sigma_k^\beta(x)\partial_\beta u_m(t,x)\right)dt
=\frac{1}{2}\operatorname{div}\left(Q(x,x)\nabla u_m(t,x)\right)dt.
\end{align*}

The main result of the article is the following.
\begin{theorem}
Assume that $N$ and $\ep$ satisfy \eqref{local}, and Condition \ref{cond:init} holds. Then, for every finite $T$ and $d\ge 1$, the empirical measure of the particle system $\{\mu^{N,m}_t(dx):t\in[0,T]\}_{m\le M}$ defined in \eqref{empirical} converges in probability as $N\to\infty$, in the space $\cD_T(\cM_{+,1})^M$ to a limit $\{\ovl \mu_t^m(dx):t\in[0,T]\}_{m\le M}$. The latter random measure is absolutely continuous with respect to Lebesgue measure, with a uniformly bounded density $\{u_m(t,x): t\in[0,T], x\in\R^d\}_{m\le M}$ that is the pathwise unique weak solution to the SPDE \eqref{spde}.
\end{theorem}

Besides pathwise uniqueness (Corollary \ref{Coroll uniqueness}), we also obtain regularity results for solutions of such nonlinear SPDEs \eqref{spde}, see Proposition \ref{Proposition regularity nonlinear}, that may be of independent interest.

\medskip
Regarding some of our assumptions, we have the following remarks.
\begin{remark}
(a). We have assumed $\sigma_{k}(x)\in C_{b}^{\infty}(\mathbb{R}%
^{d};\mathbb{R}^{d})$ for $k\in K$ to facilitate the nontrivial SPDE arguments
in Section \ref{sec:free}. While we did not try to optimize the regularity and
perhaps there exists a different proof that imposes less regularity (see some
remarks on this issue in Section \ref{sec:free}), we believe that our
arguments herein are of independent interest.

(b). The reason we chose finitely many mass levels (instead of infinite, i.e. $M=\infty$), is in order to have a finite system of SPDEs \eqref{spde} in the limit, whose pathwise uniqueness of solution can be proved, see Section \ref{unique}. Note that $M=\infty$ is assumed in \cite{HR, HR3}, with uniqueness for their PDE system proved in \cite{HR2} under appropriate assumptions.

(c). The assumption $\theta(0)=0$ seems to be necessary in order to rewrite certain quantity by means of the empirical measure, see \eqref{theta-redef}. A potential criticism is that it is unrealistic: $\theta(0)=0$ and smooth implies $\theta(\cdot)$ is very small near the origin, hence two particles which are ``too close'' to each other have also little rate to coagulate. However, ``physics'' is saved since: (i) the region, call it $\B(0,\iota\ep)$, where $\theta^\ep(\cdot)$ is very small near the origin can be extremely small compared to $\B(0,C_0\ep)$, and (ii) before getting $(\iota\ep)$-close, two particles have to be $(C_0\ep)$-close for a while, hence the loss of rate in $\B(0,\iota\ep)$ is practically not influential.

(d). Our dynamics \eqref{SDE} is less general than the setting of \cite{HR, HR3} in that the coefficients $\sigma_k$ and $\lambda$ do not depend on the mass parameter $m_i^N$. This is for technical reasons, with the need to couple with an auxiliary free system in Sections \ref{sec:bounds} and \ref{exist-density}, notably Proposition \ref{bdd-density}. How to incorporate the additional mass dependence in \eqref{SDE} is an open problem.
\end{remark}

The strategy of proof is summarised as follows. In Section \ref{tightness}, we show that the sequence of probability laws $\{\mathscr P^N\}_N$ induced by the $\cD_T(\cM_{+,1})^M$-valued random variables
\begin{align*}
\big\{\mu^{N,m}_t(dx): t\in[0,T]\big\}_{m\le M}, \quad N\in\N
\end{align*}
is tight hence weakly relatively compact. Fix any weak subsequential limit $\ovl{\mathscr P}$ of $\{\mathscr P^{N_j}\}_{j\in\N}$ along a subsequence $N_j$. By Skorohod's representation theorem, we can construct on an auxiliary probability space $(\wh\Omega, \wh\cF, \bP)$ (that depends on the subsequence) random variables $\{\wh\mu^{N_j,m} \}_{m\le M}$, $j\ge 1$, and $\{\ovl \mu ^m\}_{m\le M}$, having the laws $\mathscr P^{N_j}$, $j\ge 1$, and $\ovl{\mathscr P}$, respectively, such that  $\bP$-a.s.
\[
\{\wh\mu^{N_j,m} \}_{m\le M} \to \{\ovl \mu ^m\}_{m\le M}, \quad j\to\infty.
\]
In Sections \ref{exist-density}-\ref{sec:free}, we show that any subsequential limit measure $\{\ovl \mu ^m_t(dx): t\in[0,T]\}_{m\le M}$ is supported on the subset of measures that are absolutely continuous with respect to Lebesgue measure, with density
\begin{align*}
\left\{u_m(t,x),t\in[0,T]\right\}_{m\le M}  
\end{align*}
uniformly bounded by the deterministic constant $\Gamma$ in Condition \ref{cond:init}. This is achieved by considering an auxiliary (``free'') particle system without mass-coagulation, that dominates our true system, and studying the regularity and boundedness of solutions of its associated SPDE \eqref{equation for measures}, see in particular Theorem \ref{thm:reg-free} and Lemma \ref{lemma uniform upper bound}. In Sections \ref{sec:emp-mea}-\ref{sec:bounds}, we show that on $(\wh\Omega, \wh\cF, \bP)$, there exist a finite collection of independent Brownian motions $(\ovl W_t^k)_{k\in K}$ such that $\left(\{u_m\}_{m\le M}, (\ovl W^k)_{k\in K}\right)$ is a weak solution to the finite system of SPDEs \eqref{spde}. Here we need to deal with the difficulty posed by local interaction.

In Section \ref{unique} we show that (analytically) weak solutions to \eqref{spde} are pathwise unique. By a well-known theorem of Gyöngy-Krylov \cite[Lemma 1.1, Theorem 2.4]{Gyon Krylov}, see also \cite[Chapter 2]{Flandoli LNM}, \cite[Appendix C]{FGP}, having pathwise uniqueness, one can strengthen the convergence in law along subsequences to convergence in probability along the full sequence, of $\{\mu^{N,m}\}_{m\in M}$, $N\in\N$, on the original probability space $(\Omega, \cF, \P)$. The details involve representing two copies of the empirical measures by Skorohod's representation, and showing that any subsequential limit pair is concentrated on the diagonal. As this is similar to the Skorohod argument we already elaborate in Section \ref{sec:emp-mea} as well as those in the above references, we omit the proof of the Gyöngy-Krylov step.

\section{Identity involving the empirical measure}\label{sec:emp-mea}
We start by deriving an identity involving the empirical measure $\{\mu_t^m(dx)\}_{m\le M}$. Taking any $\phi\in C_c^\infty(\R^d)$, we consider the functional 
\begin{align*}
F_1(\eta):=\frac{1}{N}\sum_{i\in\cN(\eta)}\phi(x_i)1_{\{m_i=m\}}, 
\end{align*}
for any fixed $1\le m\le M$. Applying It\^o formula to the process 
\begin{align*}
F_1(\eta(t))=\frac{1}{N}\sum_{i\in\cN(t)}\phi(x_i^N(t))1_{\{m_i^N(t)=m\}}=\left\langle \phi(x), \mu^{N,m}_t(dx)\right\rangle, \quad t\ge0
\end{align*}
where the notation $\langle f, \nu\rangle$ denotes integrating a function $f$ against a measure $\nu$,
in view of \eqref{diff-gen}-\eqref{jump-gen}, we get that for every finite $T$,
\begin{equation*}
\begin{aligned}
\left\langle \phi(x), \mu^{N,m}_T(dx)\right\rangle 
&= \left\langle \phi(x), \mu^{N,m}_0(dx)\right\rangle
+\int_0^Tdt\frac{\lambda^2}{2N}\sum_{i\in\cN(t)}\Delta \phi(x_i^N(t))1_{\{m_i^N(t)=m\}}\\
&+\int_0^Tdt\frac{1}{2N}\sum_{i\in\cN(t)}\sum_{\alpha,\beta=1}^dQ^{\alpha\beta}\left(x_i^N(t),x_i^N(t)\right)\left(\partial^2_{\alpha\beta}\phi\right)(x_i^N(t))1_{\{m_i^N(t)=m\}}\\
&+\int_0^Tdt\frac{1}{2N}\sum_{i\in\cN(t)}\sum_{\alpha,\beta=1}^d\partial_\alpha\phi\left(x_i^N(t)\right)\partial_\beta\left(Q^{\alpha\beta}\left(x_i^N(t),x_i^N(t)\right)\right)1_{\{m_i^N(t)=m\}}\\
&+\int_0^Tdt \frac{1}{N}\sum_{i\neq j\in\cN(t)}\frac{1}{N}\theta^\ep(x_i^N(t)-x_j^N(t))\Big[\frac{m_i^N(t)}{m_i^N(t)+m_j^N(t)}\phi(x_i^N(t))1_{\{m_i^N(t)+m_j^N(t)=m\}}\\
&\qquad +\frac{m_j^N(t)}{m_i^N(t)+m_j^N(t)}\phi(x_j^N(t))1_{\{m_i^N(t)+m_j^N(t)=m\}} -\phi(x_i^N(t))1_{\{m_i^N(t)=m\}}-\phi(x_j^N(t))1_{\{m_j^N(t)=m\}}\Big]\\
& +M^{1,D,\phi}_T + M^{2,D,\phi}_T+ M_T^{J,\phi}.
\end{aligned}
\end{equation*}
Since for every $\alpha,\beta=1,...,d$,
\[
Q^{\alpha\beta}\left(x,x\right)\left(\partial^2_{\alpha\beta}\phi\right)(x)+\partial_\alpha\phi\left(x\right)\partial_\beta\left(Q^{\alpha\beta}\left(x,x\right)\right)=\partial_\beta\left(Q^{\alpha\beta}\left(x,x\right)\partial_\alpha\phi(x)\right),
\]
the above identity can be written as
\begin{equation}\label{key-identity}
\begin{aligned}
\left\langle \phi(x), \mu^{N,m}_T(dx)\right\rangle 
= & \left\langle \phi(x), \mu^{N,m}_0(dx)\right\rangle + 
\int_0^T\left\langle \mu_t^{N,m}(dx), \frac{\lambda^2}{2}\Delta \phi(x)\right\rangle dt\\
&+\int_0^T\left\langle\mu^{N,m}_t(dx),\frac{1}{2}\operatorname{div}\left(Q(x,x)\nabla \phi(x)\right)\right\rangle dt\\
&+\int_0^Tdt \sum_{n=1}^{m-1}\frac{1}{N^2}\sum_{i\neq j\in\cN(t)}\theta^\ep(x_i^N(t)-x_j^N(t))\frac{n}{m}\phi(x_i^N(t))1_{\{m_i^N(t)=n\}}1_{\{m_j^N(t)=m-n\}}\\
&+\int_0^Tdt \sum_{n=1}^{m-1}\frac{1}{N^2}\sum_{i\neq j\in\cN(t)}\theta^\ep(x_i^N(t)-x_j^N(t))\frac{m-n}{m}\phi(x_j^N(t))1_{\{m_i^N(t)=n\}}1_{\{m_j^N(t)=m-n\}}\\
&-\int_0^Tdt \sum_{n=1}^{M}\frac{1}{N^2}\sum_{i\neq j\in\cN(t)}\theta^\ep(x_i^N(t)-x_j^N(t))\phi(x_i^N(t))1_{\{m_i^N(t)=m\}}1_{\{m_j^N(t)=n\}}\\
&-\int_0^Tdt \sum_{n=1}^{M}\frac{1}{N^2}\sum_{i\neq j\in\cN(t)}\theta^\ep(x_i^N(t)-x_j^N(t))\phi(x_j^N(t))1_{\{m_j^N(t)=m\}}1_{\{m_i^N(t)=n\}}\\
&+M_T^{1,D,\phi}+M_T^{2,D,\phi}+M_T^{J,\phi}
\end{aligned}
\end{equation}
where $\{M^{J,\phi}_t\}_{t\ge0}$ denotes the martingale associated with jumps (which we do not write out explicitly), and there are two martingales associated with diffusions
\begin{align*}
M_T^{1,D,\phi}&:=\int_0^T\frac{\lambda}{N}\sum_{i\in\cN(t)}\nabla \phi(x_i^N(t))1_{\{m_i^N(t)=m\}}\cdot d\beta_i(t)\\
M_T^{2,D,\phi}&:=\int_0^T\sum_{k\in K}\sigma_k(x_i^N(t))\cdot \frac{1}{N}\sum_{i\in\cN(t)}\nabla \phi(x_i^N(t))1_{\{m_i^N(t)=m\}}dW_t^k  \\
&=\int_0^T \sum_{k\in K}\left\langle \mu_t^{N,m}(dx),\sigma_k(x)\cdot\nabla\phi(x)\right\rangle dW_t^k.
\end{align*}
Firstly, by It\^o isometry and the independence among $\beta_i(t)$, we have that 
\begin{align}\label{mg-van-1}
\E\left[\sup_{t\in[0,T]}\left|M_t^{1,D,\phi}\right|^2\right]&\le4 \E\left[\int_0^Tdt\frac{\lambda^2}{N^2}\sum_{i\in\cN(t)} |\nabla\phi(x_i^N(t))|^21_{\{m_i^N(t)=m\}}dt\right]\le \frac{C\left(\|\phi\|_{C^1}, T\right)}{N}.
\end{align}
The second diffusion martingale $M^{2,D,\phi}$ is not negligible, and is responsible for the stochastic term in the limit SPDE.
Secondly, we bound the jump martingale by (cf. \cite[Proposition 8.7]{DN})
\begin{align}
&\E\left[\sup_{t\in[0,T]}\left|M_t^{J,\phi}\right|^2\right]  \nonumber\\
&\le 4\E \int_0^Tdt \frac{1}{N^2}\sum_{i\neq j\in\cN(t)}\frac{1}{N}\theta^\ep(x_i^N(t)-x_j^N(t))\Big[\frac{m_i^N(t)}{m}\phi(x_i^N(t))1_{\{m_i^N(t)+m_j^N(t)=m\}}+\frac{m_j^N(t)}{m}\phi(x_j^N(t))1_{\{m_i^N(t)+m_j^N(t)=m\}}  \nonumber\\
&\quad -\phi(x_i^N(t))1_{\{m_i^N(t)=m\}}- \phi(x_j^N(t))1_{\{m_j^N(t)=m\}}\Big]^2 \nonumber\\
&\le 64\|\phi\|^2_\infty\E\Big[\int_0^Tdt\frac{1}{N^3}\sum_{i\neq j\in\cN(t)}\theta^\ep(x_i^N(t)-x_j^N(t))\Big]\le \frac{C(\|\phi\|_\infty)}{N}, \label{mg-van-2}
\end{align}
by Lemma \ref{cardinality} below. 
\begin{lemma}\label{cardinality}
We have that 
\begin{align*}
\E\int_0^Tdt\sum_{i\neq j\in\cN(t)}\theta^\ep(x_i^N(t)-x_j^N(t))\le N^2.
\end{align*}
\end{lemma}
\begin{proof}
We apply the generator to the process of cardinality $N(t)$ of active particles. The diffusion part of the generator does not affect $N(t)$, whereas the coagulation part decreases cardinality by either one or two (depending on if the combined mass exceeds $M$ or not). Hence, by It\^o formula and taking expectation, we have that 
\begin{align*}
\E N(T) \le \E N(0)-\E\int_0^Tdt\frac{1}{N}\sum_{i\neq j\in\cN(t)}\theta^\ep(x_i^N(t)-x_j^N(t)).
\end{align*}
Since $N(0)=N$, this completes the proof.
\end{proof}

\medskip
We observe that the middle four nonlinear terms of \eqref{key-identity} can also be written by means of empirical measure, since for every $m,n,t$,
\begin{equation}
\begin{aligned}\label{theta-redef}
&\frac{1}{N^2}\sum_{i\neq j\in\cN(t)}\theta^\ep(x_i^N(t)-x_j^N(t))\phi(x_i^N(t))\psi(x_j^N(t))1_{\{m_i^N(t)=m\}}1_{\{m_j^N(t)=n\}}\\
&=\left\langle\theta^\ep(x-y)\phi(x)\psi(y), \mu^{N,m}_t(dx)\mu^{N,n}_t(dy)\right\rangle, 
\end{aligned}
\end{equation}
where due to $\theta^\ep(0)=0$ we can include the terms with repeated indices $i=j$ to the LHS of \eqref{theta-redef}.

By \eqref{key-identity},\eqref{mg-van-1}, \eqref{mg-van-2},\eqref{theta-redef} for every $\phi\in C_c^\infty(\R^d)$ and $1\le m\le M$, we have that
 
\begin{equation}\label{key-refrased}
\begin{aligned}
\limsup_{N\to\infty}\E\Big|\left\langle \phi(x), \mu^{N,m}_T(dx)\right\rangle&-  \left\langle \phi(x), \mu^{N,m}_0(dx)\right\rangle - 
\int_0^Tdt\left\langle \mu_t^{N,m}(dx), \frac{\lambda^2}{2}\Delta \phi(x)\right\rangle\\
&-\int_0^Tdt\left\langle\mu^{N,m}_t(dx),\frac{1}{2}\operatorname{div}\left(Q(x,x)\nabla \phi(x)\right)\right\rangle\\
&-\int_0^Tdt \sum_{n=1}^{m-1}\frac{n}{m}\left\langle\theta^\ep(x-y)\phi(x), \mu^{N,m-n}_t(dy)\mu^{N,n}_t(dx)\right\rangle\\
&-\int_0^Tdt \sum_{n=1}^{m-1}\frac{m-n}{m}\left\langle\theta^\ep(x-y)\phi(y), \mu^{N,m-n}_t(dy)\mu^{N,n}_t(dx)\right\rangle\\
&+\int_0^Tdt \sum_{n=1}^{M}\left\langle\theta^\ep(x-y)\phi(x), \mu^{N,n}_t(dy)\mu^{N,m}_t(dx)\right\rangle\\
&+\int_0^Tdt \sum_{n=1}^M\left\langle\theta^\ep(x-y)\phi(y), \mu^{N,m}_t(dy)\mu^{N,n}_t(dx)\right\rangle\\
&-\int_0^T\sum_{k\in K}\left\langle \mu^{N,m}_t(dx),\sigma_k(x)\cdot\nabla \phi(x)\right\rangle dW_t^k\Big|=0.
\end{aligned}
\end{equation}
In Section \ref{tightness}, we show that the laws $\{\mathscr P_*^{N}\}_N$ of the sequence of $\cD_T(\cM_{+,1})^M\times C([0,T];\R)^{|K|}$-valued random variables
\[
\big\{\big\{\mu^{N,m}_t:t\in[0,T]\big\}_{m\le M},\big\{W_t^k:t\in[0,T]\big\}_{k\in K}\big\}, \quad N\in \N
\] 
are tight hence relatively compact (cf. Remark \ref{joint-tight}), where the space $C([0,T];\R)$ is endowed with the uniform topology. Fix any weak subsequential limit $\ovl{\mathscr P}_*$ of $\{\mathscr P_*^{N_j}\}_{j\in\N}$ along a subsequence $N_j$. By Skorohod's representation theorem, we can construct on an auxiliary probability space $(\wh\Omega, \wh\cF, \bP)$ (that depends on the subsequence) random variables $\big\{\{\wh\mu^{N_j,m} \}_{m\le M},\{\wh{W}^{N_j,k}\}_{k\in K}\big\}$, $j\ge 1$, and $\big\{\{\ovl \mu ^m\}_{m\le M},\{\ovl{W}^k\}_{k\in K}\big\}$, having the laws $\mathscr P_*^{N_j}$, $j\ge 1$, and $\ovl{\mathscr P}_*$, respectively, such that $\bP$-a.s. 
\[
\big\{\{\wh\mu^{N_j,m} \}_{m\le M},\{\wh{W}^{N_j,k}\}_{k\in K}\big\} \to \big\{\{\ovl \mu ^m\}_{m\le M},\{\ovl{W}^k\}_{k\in K}\big\}, \quad j\to\infty.
\]
Further, the limit measure has a uniformly bounded density $\{u_m(t,x):t\in[0,T]\}_{m\le M}$, i.e. for every $m$
\[
\ovl{\mu}^m_t(dx)=u_m(t,x)dx, \quad \|u_m\|_\infty\le\Gamma,
\]
as shown in Sections \ref{exist-density} and \ref{sec:free}, and we also have that $u_m(0,x)=r_mp_m(x)$.

By \eqref{key-refrased} and the representation, on $(\wh\Omega, \wh\cF, \bP)$ we have that for every $\phi\in C_c^\infty(\R^d)$ and $1\le m\le M$, 

\begin{equation}\label{key-on-new}
\begin{aligned}
\lim_{j\to\infty}\bE\Big|\left\langle \phi(x), \wh{\mu}^{N_j,m}_T(dx)\right\rangle&-  \left\langle \phi(x), \wh\mu^{N_j,m}_0(dx)\right\rangle - 
\int_0^Tdt\left\langle \wh\mu_t^{N_j,m}(dx), \frac{\lambda^2}{2}\Delta \phi(x)\right\rangle\\
&-\int_0^Tdt\left\langle\wh\mu^{N_j,m}_t(dx),\frac{1}{2}\operatorname{div}\left(Q(x,x)\nabla \phi(x)\right)\right\rangle\\
&-\int_0^Tdt \sum_{n=1}^{m-1}\frac{n}{m}\left\langle\theta^\ep(x-y)\phi(x),\wh \mu^{N_j,m-n}_t(dy)\wh\mu^{N_j,n}_t(dx)\right\rangle\\
&-\int_0^Tdt \sum_{n=1}^{m-1}\frac{m-n}{m}\left\langle\theta^\ep(x-y)\phi(y),\wh \mu^{N_j,m-n}_t(dy)\wh\mu^{N_j,n}_t(dx)\right\rangle\\
&+\int_0^Tdt \sum_{n=1}^{M}\left\langle\theta^\ep(x-y)\phi(x), \wh\mu^{N_j,n}_t(dy)\wh\mu^{N_j,m}_t(dx)\right\rangle\\
&+\int_0^Tdt \sum_{n=1}^M\left\langle\theta^\ep(x-y)\phi(y), \wh\mu^{N_j,m}_t(dy)\wh\mu^{N_j,n}_t(dx)\right\rangle\\
&-\int_0^T\sum_{k\in K}\left\langle \wh{\mu}^{N_j,m}_t(dx),\sigma_k(x)\cdot\nabla \phi(x)\right\rangle d\wh{W}_t^{N_j,k}\Big|=0.
\end{aligned}
\end{equation}
Since $\wh\mu^{N_j,m} \to \ovl \mu ^m$ in $\cD_T(\cM_{+,1})$ under Skorohod topology, $\bP$-a.s.  for every $1\le m\le M$, we have that $\bP$-a.s. for every $\phi\in C_c^\infty(\R^d)$ 
\begin{equation}\label{linear-terms}
\begin{aligned}
 &\sup_{t\in[0,T]}\left\langle \wh \mu^{N_j,m}_t(dx)- \ovl\mu^{m}_t(dx), \phi(x) \right\rangle \to 0 \\ 
&\sup_{t\in[0,T]}\left\langle \wh\mu_t^{N_j,m}(dx)- \ovl\mu_t^{m}(dx), \frac{\lambda^2}{2}\Delta \phi(x)\right\rangle \to 0\\
&\sup_{t\in[0,T]}\left\langle \wh\mu_t^{N_j,m}(dx)- \ovl\mu_t^{m}(dx), \frac{1}{2}\operatorname{div}\left(Q(x,x)\nabla \phi(x)\right)\right\rangle \to 0
\end{aligned}
\end{equation}
(cf. \cite[Ch. 3, Proposition 5.3]{EK}). 
The convergences also hold in $L^1(\bP)$ by dominated convergence, since the variables in \eqref{linear-terms} are all uniformly bounded. The middle four nonlinear terms in \eqref{key-on-new} also converge in $L^1(\bP)$ by Lemma \ref{lem:nonlin} below.

We now argue that the last martingale term in \eqref{key-on-new} also converges in $L^1(\bP)$, i.e. for every $1\le m\le M$,
\begin{align}\label{env-mg-conv}
\int_0^T\sum_{k\in K}\left\langle \wh{\mu}^{N_j,m}_t(dx),\sigma_k(x)\cdot\nabla \phi(x)\right\rangle d\wh{W}_t^{N_j,k}- \int_0^T\sum_{k\in K}\left\langle \ovl{\mu}^{m}_t(dx),\sigma_k(x)\cdot\nabla \phi(x)\right\rangle d\ovl{W}_t^{k}\to0, \quad j\to\infty.
\end{align}
Indeed, by Burkholder-Davis-Gundy inequality, we first have that for every $j\in\N$,
\begin{align}\label{BDG}
&\bE \left|\int_0^T\sum_{k\in K}\left\langle\wh \mu_t^{N_j,m}(dx)-\ovl \mu_t^m(dx),\nabla\phi(x)\cdot\sigma_k(x)\right\rangle d\wh{W}_t^{N_j,k}\right| \nonumber\\
&\le C_1\bE\left[\left(\int_0^T\sum_{k\in K}\left|\left\langle \wh\mu_t^{N_j,m}(dx)-\ovl \mu_t^m(dx),\nabla\phi(x)\cdot\sigma_k(x)\right\rangle\right|^2dt\right)^{1/2}\right]    \nonumber\\
&\le C_1\sqrt{T}\bE\sum_{k\in K}\sup_{t\in[0,T]}\left|\left\langle \wh\mu_t^{N_j,m}(dx)-\ovl \mu_t^m(dx),\nabla\phi(x)\cdot\sigma_k(x)\right\rangle\right|.
\end{align}
Since the variable inside the expectation in the last line is uniformly bounded (by $2\sum_{k\in K}\|\phi\|_{C^1}\|\sigma_k\|_\infty$), and converges to zero $\bP$-a.s. as $j\to\infty$, we have that \eqref{BDG} converges to zero. 

Secondly, since $\wh{W}_t^{N_j,k}-\ovl{W}_t^{k}$ is a $\bP$-martingale on $t\in[0,T]$ for every $k\in K$, we denote by $\big[\wh{W}^{N_j,k}-\ovl{W}^{k}\big]_t$ its quadratic variation. By Burkholder-Davis-Gundy inequality, 
\begin{align}\label{BDG-2}
&\bE \left|\int_0^T\sum_{k\in K}\left\langle \ovl \mu_t^m(dx),\nabla\phi(x)\cdot\sigma_k(x)\right\rangle d\left(\wh{W}_t^{N_j,k}-\ovl{W}_t^{k}\right)\right| \nonumber\\
&\le\sum_{k\in K}\bE \left|\int_0^T\left\langle \ovl \mu_t^m(dx),\nabla\phi(x)\cdot\sigma_k(x)\right\rangle d\left(\wh{W}_t^{N_j,k}-\ovl{W}_t^{k}\right)\right| \nonumber\\
&\le C_1'\sum_{k\in K}\bE\left[\left(\int_0^T\left|\left\langle \ovl{\mu}^m_t,\nabla\phi(x)\cdot\sigma_k(x)\right\rangle\right|^2d\big[\wh{W}^{N_j,k}-\ovl{W}^{k}\big]_t\right)^{1/2}\right]  \nonumber\\
&\le C'_1\|\phi\|_{C^1}\sum_{k\in K}\|\sigma_k\|_\infty\bE\left[\big[\wh{W}^{N_j,k}-\ovl{W}^{k}\big]_T^{1/2}\right].
\end{align}
Now by the definition of quadratic variation,
\begin{align}\label{quad-var}
\bE \big[\wh{W}^{N_j,k}-\ovl{W}^{k}\big]_T = \bE \left[\big|\wh{W}^{N_j,k}_T-\ovl{W}^{k}_T\big|^2\right].
\end{align}
Since $\wh{W}^{N_j,k}-\ovl{W}^{k}\to 0$ in the uniform topoplogy of $C([0,T];\R)$, $\bP$-a.s. as $j\to\infty$, we have that $\wh{W}^{N_j,k}_T\to\ovl{W}^{k}_T$, $\bP$-a.s.; and besides, for all $j\in\N$ and $p>2$,
\[
\bE\big[|\wh{W}_T^{N_j,k}|^p\big]=\bE\big[|\ovl{W}^{k}_T|^p\big]<\infty.
\]
By Vitali convergence theorem, RHS of \eqref{quad-var} converges to zero as $j\to\infty$, and as a consequence \eqref{BDG-2} also converges to zero. 

Combining \eqref{BDG}, \eqref{BDG-2} yields our claim \eqref{env-mg-conv} by the triangle inquality. Then, combining \eqref{key-on-new}, \eqref{linear-terms}, \eqref{env-mg-conv} and Lemma \ref{lem:nonlin}, we conclude that for every $\phi\in C_c^\infty(\R^d)$ and $1\le m\le M$, it holds that $\bP$-a.s.
\begin{align*}
\left\langle \phi(x), u_m(T,x)\right\rangle&-  \left\langle \phi(x), r_mp_m(x)\right\rangle - 
\int_0^Tdt\left\langle u_m(t,x), \frac{\lambda^2}{2}\Delta \phi(x)\right\rangle\\
&-\int_0^Tdt\left\langle u_m(t,x),\frac{1}{2}\operatorname{div}\left(Q(x,x)\nabla \phi(x)\right)\right\rangle\\
&-\int_0^Tdt \sum_{n=1}^{m-1}\left\langle u_{ m-n}(t,x)u_{ n}(t,x), \phi(x)\right\rangle\\
&+2\int_0^Tdt \sum_{n=1}^{M}\left\langle u_{ n}(t,x)u_m(t,x), \phi(x)\right\rangle\\
&-\int_0^T\sum_{k\in K}\left\langle \ovl{\mu}^{m}_t(dx),\sigma_k(x)\cdot\nabla \phi(x)\right\rangle d\ovl{W}_t^{k}=0.
\end{align*}
(We used $n/m+(m-n)/m=1$ for every $1\le n\le m-1$ in \eqref{key-on-new}.)
By the separability of the space $C_c^\infty(\R^d)$, we can combine countably many null sets such that $\bP$-a.s. for all $\phi\in C_c^\infty(\R^d)$, the preceding identity holds, which means that $\big(\{u_m\}_{m\le M}, \{\ovl{W}^k\}_{k\in K}\big)$ is a (both analytically and probabilistically) weak solution to the SPDE \eqref{spde} on $(\wh\Omega,\wh\cF,\bP)$ endowed with the filtration
\begin{align*}
\cG_t:=\sigma\big\{\{u_m(s,\cdot)\}_{m\le M}, \{\ovl{W}^k_s\}_{k\in K}: s\in[0,t]\big\}, \quad t\ge0.
\end{align*}

To treat the nonlinear terms in \eqref{key-on-new}, in Sections \ref{ito-tanaka}-\ref{sec:bounds} we use It\^o-Tanaka trick to show that
\begin{proposition}\label{ppn:tanaka}
On the original probability space $(\Omega,\cF,\P)$, for every $T$ finite, $1\le m,n\le M$ and $\phi,\psi\in C_c^\infty(\R^d)$, we have that 
\begin{equation}\label{key}
\begin{aligned}
\lim_{|z|\to0}\limsup_{N\to\infty}\E&\sup_{t\in[0,T]}\Big|\int_0^tds\frac{1}{N^2}\sum_{i\neq j\in\cN(s)}\phi\left(x_i^N(s)\right)\psi\left(x_j^N(s)\right)1_{\{m_i^N(s)=m,m_j^N(s)=n\}}\\
&\cdot\left[\theta^\ep\left(x_i^N(s)-x_j^N(s)+z\right)-\theta^\ep\left(x_i^N(s)-x_j^N(s)\right)\right]\Big|=0.
\end{aligned}
\end{equation}
\end{proposition}
Given \eqref{key}, we can show the convergence in $L^1(\bP)$ of each of the nonlinear terms in \eqref{key-on-new} as in the following lemma.
\begin{lemma}\label{lem:nonlin}
Granted Proposition \ref{ppn:tanaka}. On $(\wh\Omega,\wh\cF,\bP)$, for every $T$ finite, $1\le m,n\le M$ and $\phi,\psi\in C_c^\infty(\R^d)$, we have that 

\begin{align}\label{conv-nonlin}
\lim_{j\to\infty}\bE\sup_{t\in[0,T]}\Big|\int_0^t \left\langle\theta^\ep(x-y)\phi(x)\psi(y),\wh\mu_s^{N_j,m}(dx)\wh\mu^{N_j,n}_s(dy)\right\rangle ds- \int_0^tds\int_{\R^d}dw\phi(w)\psi(w)u_m(s,w)u_n(s,w)\Big|=0.
\end{align}
\end{lemma} 

\begin{proof}
First notice that the quantity inside the absolute value on the LHS of \eqref{key} is a function of the empirical measure (using $\theta^\ep(0)=0$)
\[
\int_0^t \left\langle\big[\theta^\ep(x-y+z)-\theta^\ep(x-y)\big]\phi(x)\psi(y),\mu_s^{N,m}(dx)\mu^{N,n}_s(dy)\right\rangle ds.
\]
Thus,  under Skorohod's representation, the same limit \eqref{key} holds on the auxiliary probability space $(\wh\Omega,\wh\cF,\bP)$ along the subsequence $N_j$, namely
\begin{align}\label{key-new}
\lim_{|z|\to0}\limsup_{j\to\infty}\bE\sup_{t\in[0,T]}\Big|\int_0^t \left\langle\big[\theta^\ep(x-y+z)-\theta^\ep(x-y)\big]\phi(x)\psi(y),\wh\mu_s^{N_j,m}(dx) \wh\mu^{N_j,n}_s(dy)\right\rangle ds\Big|=0.
\end{align}
The subsequent argument is similar to \cite[pages 42-43]{HR}. We introduce an auxiliary mollifier $\chi^\delta(x)=\delta^{-d}\chi(\delta^{-1}x)$ for some $C^\infty_b(\R^d)$ function $\chi:\R^d\to\R_+$, nonnegative, compactly supported in $\B(0,1)$,  with $\int\chi=1$. By \eqref{key-new}, we can write 
\begin{equation}\label{new-moll}
\begin{aligned}
&\sup_{t\in[0,T]}\Big|\int_0^t \left\langle\theta^\ep(x-y)\phi(x)\psi(y),\wh\mu_s^{N_j,m}(dx)\wh\mu^{N_j,n}_s(dy)\right\rangle ds\\
&- \iint_{\R^{2d}}dz_1dz_2\chi^\delta(z_1)\chi^\delta(z_2)\int_0^t \left\langle\theta^\ep(x-y+z_2-z_1)\phi(x)\psi(y),\wh\mu_s^{N_j,m}(dx) \wh\mu^{N_j,n}_s(dy)\right\rangle ds\Big|\\
&=:\text{error}(N,\delta)\\
&\quad \text{where }\quad \lim_{\delta\to0}\limsup_{N\to\infty}\bE|\text{error}(N,\delta)|=0.
\end{aligned}
\end{equation}
Shifting the arguments of $\phi$ and $\psi$ by $z_1$ and $z_2$ respectively,  with $|z_1|,|z_2|\le \delta$, we have that 
\begin{equation}
\begin{aligned}\label{shift}
&\Big|\iint_{\R^{2d}}dz_1dz_2\chi^\delta(z_1)\chi^\delta(z_2)\int_0^t \left\langle\theta^\ep(x-y+z_2-z_1)\phi(x)\psi(y),\wh\mu_s^{N_j,m}(dx)\wh\mu^{N_j,n}_s(dy)\right\rangle ds\\
- &\iint_{\R^{2d}}dz_1dz_2\chi^\delta(z_1)\chi^\delta(z_2)\int_0^t\left\langle\theta^\ep(x-y+z_2-z_1)\phi(x-z_1)\psi(y-z_2),\wh\mu_s^{N_j,m}(dx)\wh\mu^{N_j,n}_s(dy)\right\rangle ds\Big|\\
&\le C(\phi,\psi)\delta\iint_{\R^{2d}}dz_1dz_2\chi^\delta(z_1)\chi^\delta(z_2)\int_0^t  \left\langle\theta^\ep(x-y+z_2-z_1),\wh\mu_s^{N_j,m}(dx) \wh\mu^{N_j,n}_s(dy)\right\rangle ds.
\end{aligned}
\end{equation}
By a change of variables, the second term in \eqref{shift}
\begin{align*}
&\iint_{\R^{2d}}dz_1dz_2\chi^\delta(z_1)\chi^\delta(z_2)\int_0^t\left\langle \theta^\ep(x-y+z_2-z_1)\phi(x-z_1)\psi(y-z_2),\wh\mu_s^{N_j,m}(dx)\wh\mu^{N_j,n}_s(dy)\right\rangle ds\\
&= \iint_{\R^{2d}}dw_1dw_2\theta^\ep(w_1-w_2)\phi(w_1)\psi(w_2)\int_0^t\left\langle\chi^\delta(x-w_1),\wh\mu_s^{N_j,m}(dx)\right\rangle \left\langle \chi^\delta(y-w_2)\wh\mu^{N_j,n}_s(dy)\right\rangle ds.
\end{align*}
We now shift $w_2$ to $w_1$ in some arguments, using that $|w_1-w_2|\le 2C_0\ep$ being in the support of $\theta^\ep$, we get for every $t\in[0,T]$,
\begin{align*}
&\Big|\int_0^tds\iint_{\R^{2d}}dw_1dw_2\theta^\ep\left(w_1-w_2\right)\phi(w_1)\psi(w_2)\left\langle \chi^\delta(\cdot-w_1), \wh\mu_s^{N_j,m}\right\rangle\left\langle \chi^\delta(\cdot-w_2), \wh\mu_s^{N_j,n}\right\rangle\\
&- \int_0^tds\iint_{\R^{2d}}dw_1dw_2\theta^\ep\left(w_1-w_2\right)\phi(w_1)\psi(w_1)\left\langle \chi^\delta(\cdot-w_1), \wh\mu_s^{N_j,m}\right\rangle\left\langle \chi^\delta(\cdot-w_1), \wh\mu_s^{N_j,n}\right\rangle\Big|\\
&\le C(\phi,\psi, C_0,T)\ep\delta^{-2d-1}.
\end{align*}
Since $\int \theta^\ep(w_1-w_2)dw_2 =1$, the second term above equals
\begin{align*}
&\int_0^tds\iint_{\R^{2d}}dw_1dw_2\theta^\ep\left(w_1-w_2\right)\phi(w_1)\psi(w_1)\left\langle \chi^\delta(\cdot-w_1),\wh \mu_s^{N_j,m}\right\rangle\left\langle \chi^\delta(\cdot-w_1), \wh\mu_s^{N_j,n}\right\rangle\\
&=\int_0^tds\int_{\R^d}dw_1\phi(w_1)\psi(w_1)\left\langle \chi^\delta(\cdot-w_1), \wh\mu_s^{N_j,m}\right\rangle\left\langle \chi^\delta(\cdot-w_1), \wh\mu_s^{N_j,n}\right\rangle.
\end{align*}
In Section \ref{exist-density}, it is shown that $\{\wh\mu_t^{N_j,m}(dx):t\in[0,T]\}_{m\le M}\to \{u_m(t,x)dx: t\in[0,T]\}_{m\le M}$, $\bP$-a.s., we get that for fixed $\delta>0$, as $N_j\to\infty$ (hence $\ep=\ep(N_j)\to0$),  $\bP$-a.s. 
\begin{align*}
&\sup_{t\in[0,T]}\Big|\int_0^tds\int_{\R^d}dw_1\phi(w_1)\psi(w_1)\left\langle \chi^\delta(\cdot-w_1), \wh\mu_s^{N_j,m}\right\rangle\left\langle \chi^\delta(\cdot-w_1), \wh\mu_s^{N_j,n}\right\rangle\\
&-\int_0^tds\int_{\R^d}dw_1\phi(w_1)\psi(w_1)\left\langle \chi^\delta(\cdot-w_1), u_m(s,\cdot)\right\rangle\left\langle \chi^\delta(\cdot-w_1), u_n(s,\cdot)\right\rangle\Big|\to0.
\end{align*}
The convergence also holds in $L^1(\bP)$ by dominated convergence (note that at this step $\delta$ is fixed).
Then, since $\int\chi^\delta =1$ for any $\delta>0$ and $\{u_m\}_{m\le M}$ is bounded above uniformly by $\Gamma$, we have that 
\[
\left\langle \chi^\delta(\cdot-w_1), u_m(s,\cdot)\right\rangle\le \Gamma.
\]
By dominated convergence theorem, as $\delta\to0$ we get that $\bP$-a.s. 
\begin{align*}
&\sup_{t\in[0,T]}\Big|\int_0^tds\int_{\R^d}dw_1\phi(w_1)\psi(w_1)\left\langle \chi^\delta(\cdot-w_1), u_m(s,\cdot)\right\rangle\left\langle \chi^\delta(\cdot-w_1), u_n(s,\cdot)\right\rangle\\
&-\int_0^tds\int_{\R^d}dw_1\phi(w_1)\psi(w_1)u_m(s,w_1)u_n(s,w_1)\Big|\to0.
\end{align*}
The convergence also holds in $L^1(\bP)$. There is also the minor term on the RHS of \eqref{shift}
\[
C(\phi,\psi)\delta\sup_{t\in[0,T]}  \Big|\int_0^t\iint_{\R^{2d}}dz_1dz_2\chi^\delta(z_1)\chi^\delta(z_2)\left\langle\theta^\ep(x-y+z_2-z_1),\wh\mu_s^{N_j,m}(dx)\wh\mu^{N_j,n}_s(dy)\right\rangle ds\Big|
\]
that can be shown in a similar way to vanish in $L^1(\bP)$ as $j\to\infty$ followed by $\delta\to0$. 

By \eqref{new-moll}, \eqref{shift} and the previous chain of limits, we get \eqref{conv-nonlin}.
\end{proof}

\section{It\^o-Tanaka procedure}\label{ito-tanaka}
Our goal in this and  next sections is to prove Proposition \ref{ppn:tanaka}, which we argue on the original probability space $(\Omega,\cF,\P)$. 

The It\^o-Tanaka trick, well-known in the setting of SDEs, is a way to substitute a less regular function (here $\theta^\ep(\cdot)$) by a more regular one, via the application of It\^o formula. Fix $1\le m,n\le N$ and consider the (time-dependent) functional 
\begin{align*}
F_2(t,\eta):=\frac{1}{N^2}\sum_{i\neq j\in\cN(\eta)}v^{\ep,z}\left(t,x_i, x_j\right)\phi(x_i)\psi(x_j)1_{\{m_i=m,m_j=n\}}
\end{align*}
where $v^{\ep,z}(t,x,y):[0,T]\times\R^{d}\times\R^d\to\R$ also depends on $z\in\R^d$. In fact, it is of the form of a difference 
\begin{align*}
v^{\ep,z}(t,x,y)= r^{\ep,z}(t,x,y)-r^{\ep,0}(t,x,y)
\end{align*}
where $r^{\ep,z}(t,x,y)$ is a family (indexed by $z$) of nonnegative functions in the parabolic H\"older space $C^{1+\alpha/2, 2+\alpha}([0,T]\times\R^{2d})$ (cf. \cite{Kry}), for some $\alpha\in(0,1)$, defined in \eqref{aux-pde}. Applying It\^o formula to the process
\begin{align*}
F_2(t,\eta(t)):=\frac{1}{N^2}\sum_{i\neq j\in\cN(t)}v^{\ep,z}\left(t,x_i^N(t),x_j^N(t)\right)\phi(x_i^N(t))\psi(x_j^N(t))1_{\{m_i^N(t)=m,m_j^N(t)=n\}}, \quad t\ge0
\end{align*}
and integrating on $[0,T']$, for any $T'\le T$, we get the following terms  from the action of the diffusion generator $\cL^N_D$ \eqref{diff-gen}:
\begin{align}\label{H1}
H_{1}:=\int_0^{T'}dt \frac{1}{N^2}\sum_{i\neq j\in\cN(t)}\left(\left(\partial_t+\frac{\lambda^2}{2} \Delta_x+\frac{\lambda^2}{2}\Delta_y\right) v^{\ep,z}\right)\left(t,x_i^N(t),x_j^N(t)\right)\phi(x_i^N(t))\psi(x_j^N(t))1_{\{m_i^N(t)=m,m_j^N(t)=n\}}\;;
\end{align}
\begin{equation}
\begin{aligned}
H_{2}&:=\int_0^{T'}dt\frac{\lambda^2}{N^2}\sum_{i\neq j\in\cN(t)}\Big[\frac{1}{2}v^{\ep,z}\left(t,x_i^N(t),x_j^N(t)\right)\left(\Delta \phi(x_i^N(t))\psi(x_j^N(t))+\phi(x_i^N(t))\Delta \psi(x_j^N(t))\right)\\
&+\left(\nabla_x v^{\ep,z}\right)\left(t,x_i^N(t),x_j^N(t)\right)\cdot \nabla \phi(x_i^N(t))\psi(x_j^N(t))\\
&+\left(\nabla_y v^{\ep,z}\right)\left(t,x_i^N(t),x_j^N(t)\right)\cdot \nabla \psi(x_j^N(t))\phi(x_i^N(t))\Big]
1_{\{m_i^N(t)=m,m_j^N(t)=n\}} \;;
\end{aligned}
\end{equation}
where $\Delta_x$ denotes Laplacian with respect to the first $d$ spatial coordinates, and $\Delta_y$ with respect to the last $d$ spatial coordinates; the same interpretation applies for gradients $\nabla_x, \nabla_y$;
\begin{equation}\label{H3}
\begin{aligned}
H_{3}:=&\int_0^{T'}dt \frac{1}{N^2}\sum_{i\neq j\in\cN(t)}
\frac{1}{2}\sum_{\alpha,\beta=1}^d\Big[Q^{\alpha\beta}(x_i^N(t),x_i^N(t))\left(\partial^2_{x^\alpha x^\beta}v^{\ep,z}\right)\left(t,x_i^N(t),x_j^N(t)\right)\\
&+Q^{\alpha\beta}(x_j^N(t),x_j^N(t))\left(\partial^2_{y^\alpha y^\beta}v^{\ep,z}\right)\left(t,x_i^N(t),x_j^N(t)\right)\\
&+Q^{\alpha\beta}(x_i^N(t),x_j^N(t))\left(\partial^2_{x^\alpha y^\beta}v^{\ep,z}\right)\left(t,x_i^N(t),x_j^N(t)\right)\\
&+Q^{\alpha\beta}(x_j^N(t),x_i^N(t))\left(\partial^2_{y^\alpha x^\beta}v^{\ep,z}\right)\left(t,x_i^N(t),x_j^N(t)\right)\Big]
\phi(x_i^N(t))\psi(x_j^N(t))1_{\{m_i^N(t)=m,m_j^N(t)=n\}} \; ;
\end{aligned}
\end{equation}
\begin{equation}
\begin{aligned}
H_{4}:=&\int_0^{T'}dt \frac{1}{N^2}\sum_{i\neq j\in\cN(t)}
\frac{1}{2}\sum_{\alpha,\beta=1}^d  \Big[Q^{\alpha\beta}(x_i^N(t),x_i^N(t))v^{\ep,z}\left(t,x_i^N(t),x_j^N(t)\right)\partial^2_{\alpha \beta} \phi(x_i^N(t))\psi(x_j^N(t))\\
& +Q^{\alpha\beta}(x_j^N(t),x_j^N(t))v^{\ep,z}\left(t,x_i^N(t),x_j^N(t)\right)\phi(x_i^N(t))\partial^2_{\alpha \beta} \psi(x_j^N(t))\\
&+2Q^{\alpha\beta}(x_i^N(t),x_i^N(t)) \left(\partial_{x^\alpha}v^{\ep,z}\right) \left(t,x_i^N(t),x_j^N(t)\right)\partial_{\beta} \phi(x_i^N(t))\psi(x_j^N(t))\\
&+ 2Q^{\alpha\beta}(x_j^N(t),x_j^N(t)) \left(\partial_{y^\alpha}v^{\ep,z}\right)\left(t,x_i^N(t),x_j^N(t)\right)\partial_{\beta} \psi(x_j^N(t))\phi(x_i^N(t))\\
&+Q^{\alpha\beta}(x_i^N(t),x_j^N(t))v^{\ep,z}\left(t,x_i^N(t),x_j^N(t)\right)\partial_{\alpha}\phi(x_i^N(t))\partial_{\beta} \psi(x_j^N(t))\\
&+Q^{\alpha\beta}(x_i^N(t),x_j^N(t))\left(\partial_{y^\beta} v^{\ep,z}\right)\left(t,x_i^N(t),x_j^N(t)\right)\partial_{\alpha}\phi(x_i^N(t))\psi(x_j^N(t))\\
&+Q^{\alpha\beta}(x_i^N(t),x_j^N(t))\left(\partial_{x^\alpha}v^{\ep,z}\right) \left(t,x_i^N(t),x_j^N(t)\right)\phi(x_i^N(t))\partial_{\beta} \psi(x_j^N(t))\\
&+Q^{\alpha\beta}(x_j^N(t),x_i^N(t))v^{\ep,z}\left(t,x_i^N(t),x_j^N(t)\right)\partial_{\beta}\phi(x_i^N(t))\partial_{\alpha} \psi(x_j^N(t))\\
&+Q^{\alpha\beta}(x_j^N(t),x_i^N(t))\left(\partial_{x^\beta}v^{\ep,z}\right) \left(t,x_i^N(t),x_j^N(t)\right)\phi(x_i^N(t))\partial_{\alpha} \psi(x_j^N(t))\\
&+Q^{\alpha\beta}(x_j^N(t),x_i^N(t))\left(\partial_{y^\alpha}v^{\ep,z}\right) \left(t,x_i^N(t),x_j^N(t)\right)\partial_{\beta} \phi(x_i^N(t))\psi(x_j^N(t))\Big]
1_{\{m_i^N(t)=m,m_j^N(t)=n\}} \;;
\end{aligned}
\end{equation}
\begin{equation}\label{H5}
\begin{aligned}
&H_{5}:=\\
&\int_0^{T'}dt \frac{1}{2N^2}\sum_{i\neq j\in\cN(t)}\sum_{\alpha,\beta=1}^d\partial_\alpha\left(v^{\ep,z}\left(t,\cdot,x_j^N(t)\right)\phi(\cdot)\right) \left(x_i^N(t)\right)\partial_\beta\left(Q^{\alpha\beta}\left(x_i^N(t),x_i^N(t)\right)\right)\psi(x_j^N(t))1_{\{m_i^N(t)=m,m_j^N(t)=n\}}\\
&+\int_0^{T'}dt \frac{1}{2N^2}\sum_{i\neq j\in\cN(t)}\sum_{\alpha,\beta=1}^d\partial_\alpha\left(v^{\ep,z}\left(t,x_i^N(t),\cdot\right)\psi(\cdot)\right) \left(x_j^N(t)\right)\partial_\beta\left(Q^{\alpha\beta}\left(x_j^N(t),x_j^N(t)\right)\right)\phi(x_i^N(t))1_{\{m_i^N(t)=m,m_j^N(t)=n\}} \; ;
\end{aligned}
\end{equation}
we note here that only $H_1$ and $H_3$ involve second partial derivatives of $v^{\ep,z}$.
 
From the action of the coagulation generator $\cL^N_J$ \eqref{jump-gen} we get:
\begin{equation}\label{H6}
\begin{aligned}
H_{6}:=&\int_0^{T'}dt\frac{1}{N^2}\sum_{i\in\cN(t)}\sum_{k\in\cN(t), k\neq i}\frac{1}{N}\theta^\ep(x_i^N(t)-x_k^N(t))\\
&\cdot\Big[\sum_{j\in\cN(t), j\neq i,k}\frac{m_i^N(t)}{m_i^N(t)+m_k^N(t)}v^{\ep,z}(t,x_i^N(t),x_j^N(t))\phi(x_i^N(t))1_{\{m_i^N(t)+m_k^N(t)=m\}}\psi(x_j^N(t))1_{\{m_j^N(t)=n\}}\\
& + \sum_{j\in\cN(t), j\neq i,k}\frac{m_k^N(t)}{m_i^N(t)+m_k^N(t)}v^{\ep,z}(t,x_k^N(t),x_j^N(t))\phi(x_k^N(t))1_{\{m_i^N(t)+m_k^N(t)=m\}}\psi(x_j^N(t))1_{\{m_j^N(t)=n\}}\\
&-\sum_{ j\in\cN(t), j\neq i,k}v^{\ep,z}(t,x_i^N(t),x_j^N(t))\phi(x_i^N(t))1_{\{m_i^N(t)=m\}}\psi(x_j^N(t))1_{\{m_j^N(t)=n\}}\\
&-\sum_{j\in\cN(t), j\neq i,k}v^{\ep,z}(t,x_k^N(t),x_j^N(t))\phi(x_k^N(t))1_{\{m_i^N(t)=m\}}\psi(x_j^N(t))1_{\{m_j^N(t)=n\}}\Big]
\end{aligned}
\end{equation}
\begin{equation*}
\begin{aligned}
+&\int_0^{T'}dt\frac{1}{N^2}\sum_{i\in\cN(t)}\sum_{k\in\cN(t),k\neq i}\frac{1}{N}\theta^\ep(x_i^N(t)-x_k^N(t))\\
&\cdot\Big[\sum_{j\in\cN(t),j\neq i,k}\frac{m_i^N(t)}{m_i^N(t)+m_k^N(t)}v^{\ep,z}(t,x_j^N(t),x_i^N(t))\phi(x_j^N(t))1_{\{m_j^N(t)=m\}}\psi(x_i^N(t))1_{\{m_i^N(t)+m_k^N(t)=n\}}\\
& + \sum_{j\in\cN(t),j\neq i,k}\frac{m_k^N(t)}{m_i^N(t)+m_k^N(t)}v^{\ep,z}(t,x_j^N(t),x_k^N(t))\phi(x_j^N(t))1_{\{m_j^N(t)=m\}}\psi(x_k^N(t))1_{\{m_i^N(t)+m_k^N(t)=n\}}\\
&-\sum_{j\in\cN(t),j\neq i,k}v^{\ep,z}(t,x_j^N(t),x_i^N(t))\phi(x_j^N(t))1_{\{m_j^N(t)=m\}}\psi(x_i^N(t))1_{\{m_i^N(t)=n\}}\\
&-\sum_{j\in\cN(t),j\neq i,k}v^{\ep,z}(t,x_j^N(t),x_k^N(t))\phi(x_j^N(t))1_{\{m_i^N(t)=m\}}\psi(x_k^N(t))1_{\{m_k^N(t)=n\}}\Big] \; ;
\end{aligned}
\end{equation*}

\begin{align}\label{H7}
H_7:=&-\int_0^{T'}dt\frac{1}{N^3}\sum_{i\in\cN(t)}\sum_{k\in\cN(t),k\neq i}\theta^\ep(x_i^N(t)-x_k^N(t))v^{\ep,z}\left(t,x_i^N(t),x_k^N(t)\right)\phi(x_i^N(t))\psi(x_k^N(t))1_{\{m_i^N(t)=m,m_k^N(t)=n\}}.
\end{align}

\begin{remark}
This negative term $H_7$ arises because of the specific coagulation rule, namely if particles $(i,k)$ coagulate, then they are both removed from the system. There may be a new particle added, but it is of a different mass hence has to be reconsidered (in $H_6$). Unlike \cite{HR, HR3}, under our scaling \eqref{local} $H_7$ turns out to be negligible, see Lemma \ref{lem:H6}.
\end{remark}

Regarding the martingale terms, of which $M_1,M_2$ come from diffusion
\begin{align*}
M_1:=&\int_0^{T'}\frac{\lambda}{N^2}\sum_{i\in\cN(t)}\sum_{j\in\cN(t),j\neq i}\nabla\left[v^{\ep,z}\left(t,\cdot,x_j^N(t)\right)\phi(\cdot)\right]\left(x_i^N(t)\right)\psi(x_j^N(t))1_{\{m_i^N(t)=m,m_j^N(t)=n\}}\cdot d\beta_i(t)\\
&+\int_0^{T'}\frac{\lambda}{N^2}\sum_{j\in\cN(t)}\sum_{i\in\cN(t),i\neq j}\nabla\left[v^{\ep,z}\left(t,x_i^N(t),\cdot\right)\psi(\cdot)\right]\left(x_j^N(t)\right) \phi(x_i^N(t))1_{\{m_i^N(t)=m,m_j^N(t)=n\}}\cdot d\beta_j(t)
\end{align*}
whose quadratic variation is

\begin{equation}\label{B1}
\begin{aligned}
B_1&=\int_0^{T'}\frac{\lambda^2}{N^4}\sum_{i\in\cN(t)}\left(\sum_{j\in\cN(t),j\neq i}\nabla\left[v^{\ep,z}\left(t,\cdot,x_j^N(t)\right)\phi(\cdot)\right]\left(x_i^N(t)\right) \psi(x_j^N(t))1_{\{m_i^N(t)=m,m_j^N(t)=n\}}\right)^2dt\\
&+\int_0^{T'}\frac{\lambda^2}{N^4}\sum_{j=1}^{N(t)}\left(\sum_{i\in\cN(t),i\neq j}\nabla\left[v^{\ep,z}\left(t,x_i^N(t),\cdot\right)\psi(\cdot)\right]\left(x_j^N(t)\right)\phi(x_i^N(t))1_{\{m_i^N(t)=m,m_j^N(t)=n\}}\right)^2dt \;;
\end{aligned}
\end{equation}
and 
\begin{align*}
M_2&=\int_0^{T'}\sum_{k\in K}\sigma_k(x_i^N(t))\cdot\frac{1}{N^2}\sum_{i\in\cN(t)}\sum_{j\in\cN(t),j\neq i}\nabla\left[v^{\ep,z}\left(t,\cdot,x_j^N(t)\right)\phi(\cdot)\right]\left(x_i^N(t)\right) \psi(x_j^N(t))1_{\{m_i^N(t)=m,m_j^N(t)=n\}}dW_k(t)\\
&+\int_0^{T'}\sum_{k\in K}\sigma_k(x_j^N(t))\cdot\frac{1}{N^2}\sum_{j\in\cN(t)}\sum_{i\in\cN(t),i\neq j}\nabla\left[v^{\ep,z}\left(t,x_i^N(t),\cdot\right)\psi(\cdot)\right]\left(x_j^N(t)\right) \phi(x_i^N(t))1_{\{m_i^N(t)=m,m_j^N(t)=n\}}dW_k(t)
\end{align*}
whose quadratic variation is
\begin{equation}\label{B2}
\begin{aligned}
B_2&:=\int_0^{T'}\frac{1}{N^4}\sum_{k\in K}\left(\sigma_k(x_i^N(t))\cdot\sum_{i\neq j\in\cN(t)}\nabla\left[v^{\ep,z}\left(t,\cdot,x_j^N(t)\right)\phi(\cdot)\right]\left(x_i^N(t)\right) \psi(x_j^N(t))1_{\{m_i^N(t)=m,m_j^N(t)=n\}}\right)^2dt\\
&+\int_0^{T'}\frac{1}{N^4}\sum_{k\in K}\left(\sigma_k(x_j^N(t))\cdot\sum_{i\neq j\in\cN(t)}\nabla\left[v^{\ep,z}\left(t,x_i^N(t),\cdot\right)\psi(\cdot)\right]\left(x_j^N(t)\right)  \phi(x_i^N(t))1_{\{m_i^N(t)=m,m_j^N(t)=n\}}\right)^2dt \; ;
\end{aligned}
\end{equation}
and the jump part of the martingale $M_3$ (which we do not write explicitly) has its second moment bounded above by
\begin{equation}\label{B3}
\begin{aligned}
B_3:=&\int_0^{T'}dt\frac{4}{N^4}\sum_{i\in\cN(t)}\sum_{k\in\cN(t),k\neq i}\frac{1}{N}\theta^\ep(x_i^N(t)-x_k^N(t))\\
&\cdot\Big[\sum_{j\neq i,k}\frac{m_i^N(t)}{m_i^N(t)+m_k^N(t)}v^{\ep,z}(t,x_i^N(t),x_j^N(t))\phi(x_i^N(t))1_{\{m_i^N(t)+m_k^N(t)=m\}}\psi(x_j^N(t))1_{\{m_j^N(t)=n\}}\\
& + \sum_{j\neq i,k}\frac{m_k^N(t)}{m_i^N(t)+m_k^N(t)}v^{\ep,z}(t,x_k^N(t),x_j^N(t))\phi(x_k^N(t))1_{\{m_i^N(t)+m_k^N(t)=m\}}\psi(x_j^N(t))1_{\{m_j^N(t)=n\}}\\
&-\sum_{j\neq i,k}v^{\ep,z}(t,x_i^N(t),x_j^N(t))\phi(x_i^N(t))1_{\{m_i^N(t)=m\}}\psi(x_j^N(t))1_{\{m_j^N(t)=n\}}\\
&-\sum_{j\neq i,k}v^{\ep,z}(t,x_k^N(t),x_j^N(t))\phi(x_k^N(t))1_{\{m_i^N(t)=m\}}\psi(x_j^N(t))1_{\{m_j^N(t)=n\}}\\
&-v^{\ep,z}\left(t,x_i^N(t),x_k^N(t)\right)\phi(x_i^N(t))\psi(x_k^N(t))1_{\{m_i^N(t)=m,m_k^N(t)=n\}}\Big]^2
\end{aligned}
\end{equation}
\begin{equation*}
\begin{aligned}
&+\int_0^{T'}\frac{4}{N^4}\sum_{i\in\cN(t)}\sum_{k\in\cN(t),k\neq i}\frac{1}{N}\theta^\ep(x_i^N(t)-x_k^N(t))\\
&\cdot\Big[\sum_{j\neq i,k}\frac{m_i^N(t)}{m_i^N(t)+m_k^N(t)}v^{\ep,z}(t,x_j^N(t),x_i^N(t))\phi(x_j^N(t))1_{\{m_i^N(t)=m\}}\psi(x_i^N(t))1_{\{m_i^N(t)+m_k^N(t)=n\}}\\
& + \sum_{j\neq i,k}\frac{m_k^N(t)}{m_i^N(t)+m_k^N(t)}v^{\ep,z}(t,x_j^N(t),x_k^N(t))\phi(x_j^N(t))1_{\{m_j^N(t)=m\}}\psi(x_k^N(t))1_{\{m_i^N(t)+m_k^N(t)=n\}}\\
&-\sum_{j\neq i,k}v^{\ep,z}(t,x_j^N(t),x_i^N(t))\phi(x_j^N(t))1_{\{m_j^N(t)=m\}}\psi(x_i^N(t))1_{\{m_i^N(t)=n\}}\\
&-\sum_{j\neq i,k}v^{\ep,z}(t,x_j^N(t),x_k^N(t))\phi(x_j^N(t))1_{\{m_i^N(t)=m\}}\psi(x_k^N(t))1_{\{m_k^N(t)=n\}}\\
&-v^{\ep,z}\left(t,x_i^N(t),x_k^N(t)\right)\phi(x_i^N(t))\psi(x_k^N(t))1_{\{m_i^N(t)=m,m_k^N(t)=n\}}\Big]^2dt.
\end{aligned}
\end{equation*}
We also have the initial and terminal conditions
\begin{align*}
H_8:=&\frac{1}{N^2}\sum_{i\neq j\in\cN(0)}v^{\ep,z}\left(0,x_i(0),x_j(0)\right)\phi(x_i(0))\psi(x_j(0))1_{\{m_i(0)=m,m_j(0)=n\}};\\
H_9:=&\frac{1}{N^2}\sum_{i\neq j\in\cN(T')}v^{\ep,z}\left(T',x^N_i(T'),x^N_j(T')\right)\phi(x^N_i(T'))\psi(x^N_j(T'))1_{\{m^N_i(T')=m,m^N_j(T')=n\}}.
\end{align*}
We set up the following family (indexed by $z\in\R^d$) of auxiliary PDE terminal value problems, whose unique nonnegative solution is called $r^{\ep,z}(t,x,y):[0,T]\times\R^{d}\times\R^d\to\R_+$:
\begin{align}
\begin{cases}
\left[\partial_t+ \frac{\lambda^2}{2}\big(\Delta_x+\Delta_y\big)+\frac{1}{2}\sum_{\alpha,\beta=1}^d \big(Q^{\alpha\beta}(x,x)\partial^2_{x^\alpha x^\beta}+Q^{\alpha\beta}(y,y)\partial^2_{y^\alpha y^\beta}+Q^{\alpha\beta}(x,y)\partial^2_{x^\alpha y^\beta}+Q^{\alpha\beta}(y,x)\partial^2_{y^\alpha x^\beta}\big)\right]r^{\ep,z}\left(t,x,y\right)\\
=-\theta^\ep(x-y+z),
\quad  (t,x,y)\in[0,T]\times\R^{2d}\\\\
r^{\ep,z}(T,x,y)=0.\label{aux-pde}
\end{cases}
\end{align}
To be more transparant, if we denote $\bx:=(x,y)\in\R^{2d}$, $\Delta_\bx:=\Delta_x+\Delta_y$, 
\[
D^2_\bx:=\left(
\begin{matrix}
D^2_{xx}, D^2_{xy}\\
D^2_{yx}, D^2_{yy}
\end{matrix}\right)
\]
and the $(2d)\times(2d)$ non-negative definite matrix
\begin{align*}
\wh Q(\bx):=\left(
\begin{matrix}
Q(x,x), Q(x,y)\\
Q(y,x), Q(y,y)
\end{matrix}\right), \quad\bx=(x,y)
\end{align*}
then, \eqref{aux-pde} can be rewritten as 
\begin{align}\label{aux-pde-simp}
\begin{cases}
\left[\partial_t+\frac{\lambda^2}{2}\Delta_\bx+\frac{1}{2}\text{tr}\big(\wh Q(\bx)D^2_{\bx}\big)\right]r^{\ep,z}(t,\bx)=-\theta^\ep(x-y+z)\\
r^{\ep,z}(T,\bx)=0, \quad \bx=(x,y), t\in[0,T].
\end{cases}
\end{align}
To see that $\wh Q(\bx)$ is non-negative, take any $\xi=(\xi_1,\xi_2)\in\R^{2d}$, we have that 
\[
\xi\wh Q(\bx)\xi^T=\sum_{k\in K}|\xi_1\cdot\sigma_k(x)+\xi_2\cdot\sigma_k(y)|^2\ge0.
\]
Since $\theta^\ep\in C^\alpha(\R^d)$ for some $\alpha\in(0,1)$ and $\wh Q$ is smooth of class $C_b^\infty$, the solution $r^{\ep,z}(t,x,y)\in C^{1+\alpha/2, 2+\alpha}([0,T]\times\R^{2d})$, for every $\ep\in(0,1)$ and $z\in\R^d$ (cf. \cite[Theorem 8.10.1]{Kry}).
Let us denote the non-divergence form operator with $C^\infty_b(\R^{2d})$ coefficients
\begin{align}\label{UE}
\cA_\bx:= \frac{\lambda^2}{2}\Delta_\bx+\frac{1}{2}\text{tr}\left(\wh Q(\bx)D^2_\bx\right).
\end{align}
By the parabolic Maximum Principle, since $\theta^\ep\ge 0$, the unique solution $r^{\ep,z}(t,x,y)$ of \eqref{aux-pde-simp} is nonnegative. Since
\begin{align*}
v^{\ep,z}(t,x,y)=r^{\ep,z}(t,x,y)-r^{\ep,0}(t,x,y),
\end{align*}
by linearity it also follows that 
\begin{align}\label{eq:u}
\begin{cases}
\left[\partial_t+\frac{\lambda^2}{2}\Delta_\bx+\frac{1}{2}\text{tr}\big(\wh Q(\bx)D^2_\bx\big)\right]v^{\ep,z}(t,\bx)=-\theta^\ep(x-y+z)+\theta^\ep(x-y)\\
v^{\ep,z}(T,\bx)=0, \quad \bx=(x,y), t\in[0,T].
\end{cases}
\end{align}
From \eqref{H1}, \eqref{H3},  \eqref{eq:u}, we have that (recall $T'\le T$)
\begin{equation}
\begin{aligned}
H_{1}+H_{3}&=\int_0^{T'}dt \frac{1}{N^2}\sum_{i\neq j\in\cN(t)}\phi(x_i^N(t))\psi(x_j^N(t))1_{\{m_i^N(t)=m,m_j^N(t)=n\}}\\
&\qquad \cdot\left[\partial_t+ \frac{\lambda^2}{2}\Delta_\bx+\frac{1}{2}\text{tr}\big(\wh Q\big(x_i^N(t),x_j^N(t)\big)D^2_\bx\big)\right] v^{\ep,z}\left(t,x_i^N(t),x_j^N(t)\right)\\
&=- \int_0^{T'}dt\frac{1}{N^2}\sum_{i\neq j\in\cN(t)}\phi(x_i^N(t))\psi(x_j^N(t))1_{\{m_i^N(t)=m,m_j^N(t)=n\}} \\
&\qquad \cdot\left[\theta^\ep\left(x_i^N(t)-x_j^N(t)+z\right)-\theta^\ep\left(x_i^N(t)-x_j^N(t)\right)\right].
\end{aligned}
\end{equation}
In view of the identity from It\^o formula
\begin{align}\label{ito-terms}
H_1+...+H_9+M_1+M_2+M_3=0,
\end{align}
we can accomplish \eqref{key} if we can show that the rest of the terms in \eqref{ito-terms}, namely those apart from $H_1,H_3$, are all negligible in the sense that 
\begin{equation}\label{negligible}
\begin{aligned}
&\lim_{|z|\to0}\limsup_{N\to\infty}\E\sup_{T'\in[0,T]}|H_i|=0, \quad i\in\{2,4,5,6,7,8,9\}\\
&\lim_{|z|\to0}\limsup_{N\to\infty}\E\sup_{T'\in[0,T]}|M_i|=0, \quad i\in\{1,2,3\}.
\end{aligned}
\end{equation}
These terms only contain up to first partial derivatives of $v^{\ep,z}$, hence their regularity is strictly better than $\theta^\ep$. 

\medskip
The following key proposition provides uniform bounds on $r^\ep$ and its gradient. The proof is similar to \cite[Proposition 5]{FH}, due to which we omit some details. Recall that $C_0$ is the maximal radius of the compact support of $\theta$. 
\begin{proposition}\label{bounds}
Let $r^{\ep,z}(t,x,y)$ be the unique solution of the PDE \eqref{aux-pde-simp}. Then, there exists some finite constant $C(d, T, C_0, \\
\{\sigma_k\}_{k\in K}, \lambda)$ such that 
for any $x,y,z\in\R^d$, $t\in[0,T]$ and $\ep\le\ep_0$ for some small $\ep_0(C_0)$, we have that 
\begin{align}\label{unif-bd}
&r^{\ep,z}(t,x,y)\le 
\begin{cases}
Ce^{-C|y-x-z|^2}1_{\{|y-x-z|\ge 4\}}+C\left(|y-x-z|\vee \ep\right)^{2-d}1_{\{|y-x-z|<4\}}  ,    \quad d\ge3 \\\\
Ce^{-C|y-x-z|^2}1_{\{|y-x-z|\ge 4\}}+C\left|\log\left(|y-x-z|\vee \ep\right)\right|1_{\{|y-x-z|<4\}}, \quad d=2\\\\
Ce^{-C|y-x-z|^2}1_{\{|y-x-z|\ge 4\}}+C1_{\{|y-x-z|<4\}},   \quad d=1.
\end{cases}
\end{align}
\begin{equation}\label{unif-bd-grad}
\begin{aligned}
\left|\nabla_x r^{\ep,z}(t,x,y)\right|&\le
\begin{cases}
 Ce^{-C|y-x-z|^2}1_{\{|y-x-z|\ge 4\}}+C\left(|y-x-z|\vee \ep\right)^{1-d}1_{\{|y-x-z|<4\}},
\quad d\ge 2\\\\
 Ce^{-C|y-x-z|^2}1_{\{|y-x-z|\ge 4\}}+C\left|\log(|y-x-z|\vee \ep)\right|1_{\{|y-x-z|<4\}},
\quad d=1.
\end{cases}
\end{aligned}
\end{equation}
The same also holds for $\nabla_y$ in place of $\nabla_x$.
\end{proposition}

\begin{remark}
In particular, we have the following useful crude bounds: there exists $C=C(d,T,C_0,\{\sigma_k\}_{k\in K}, \lambda)$, such that for any $x,y,z\in\R^d$, $t\in[0,T]$  and $\ep\in(0,1)$ small
\begin{align}
r^{\ep,z}(t,x,y)&\le 
\begin{cases}\label{unif-bd-simp}
C\ep^{2-d}, \quad d\ge 3\\
C|\log \ep|, \quad d=2\\
C, \quad d=1
\end{cases}\\
|\nabla_x r^{\ep,z}(t,x,y)|&\le 
\begin{cases}\label{unif-grad-simp}
C\ep^{1-d}, \quad d\ge 2\\
C|\log\ep|, \quad d=1.
\end{cases}
\end{align}
\end{remark}

\begin{remark}
Let us explain the first inequality in \eqref{unif-bd-simp} by a simple heuristic argument.
First, notice that $\theta_{\epsilon}\left(  x\right)  \leq c\epsilon
^{-d}1_{\B\left(  0,\epsilon C_0\right)  }\left(  x\right)  $ for some constant
$c>0$; take $C_0=1$ below for notational simplicity. Due to the Gaussian upper
bound described at the beginning of the proof, the problem reduces to prove
that%
\[
\epsilon^{-d}\int_{0}^{t}\P\left(  B_{s}\in S\left(  y-x-z,\epsilon\right)
\right)  ds\leq C\epsilon^{-d+2}%
\]
for some constant $C>0$, where $\left(  B_{t}\right)  $ is a standard Brownian motion
in $\mathbb{R}^{2d}$ and $S\left(  u,\epsilon\right)  $ is the strip in
$\mathbb{R}^{2d}$ defined as%
\[
S\left(  u,\epsilon\right) : =\left\{  \left(  x^{\prime},y^{\prime}\right)
\in\mathbb{R}^{2d}:\left\vert y^{\prime}-x^{\prime}-u\right\vert \leq
\epsilon\right\}, \quad u\in\R^d.
\]
When $\left\vert y-x-z\right\vert $ is not infinitesimal in $\epsilon$, say
$\left\vert y-x-z\right\vert \geq4$ as in the proof above, then the strip
$S\left(  y-x-z,\epsilon\right)  $ is far from the origin (where $\left(
B_{t}\right)  $ starts deterministically), hence, thanks also to the
exponential decay of the Brownian density, the problem reduces to estimate
(where $Leb$ denotes Lebesgue measure in $\mathbb{R}^{2d}$)
\[
\epsilon^{-d}\int_{0}^{t}Leb\left(  S\left(  y-x-z,\epsilon\right)  \cap
\B\left(  0,1\right)  \right)  ds.
\]
It is easily seen that $Leb\left(  S\left(  y-x-z,\epsilon\right)  \cap
\B\left(  0,1\right)  \right)  \leq\epsilon^{d}$ (up to constants) and thus the
bound above holds (see also the first inequality in \eqref{unif-bd}). If $\left\vert
y-x-z\right\vert $ is very small, on the contrary, let us examine the worst
case,
\[
\epsilon^{-d}\int_{0}^{t}\P\left(  B_{s}\in S\left(  0,\epsilon\right)
\right)  ds.
\]
The problem is that $\P\left(  B_{s}\in S\left(  0,\epsilon\right)  \right)  $
is not infinitesimal with $\epsilon$ for $s=0$. The very crude but convincing
argument is that the Brownian motion $B_{s}$ remains in the strip $S\left(
0,\epsilon\right)  $ for a time of order $\epsilon^{2}$, contributing to the
integral with the term
\[
\int_{0}^{\epsilon^{2}}\P\left(  B_{s}\in S\left(  0,\epsilon\right)  \right)
ds\sim\epsilon^{2}.
\]
After this time the estimate is better. Thus there is an addend $\epsilon
^{-d}\epsilon^{2}$ in the estimate. 
\end{remark}

\begin{proof}
We first show \eqref{unif-bd}, starting with the representation of solution of \eqref{aux-pde-simp}
\begin{align*}
r^{\ep,z}(T-t,x,y)=&\int_{[0,t]\times\R^{2d}}q_{t-s}(x,y;x',y')\theta^\ep(x'-y'+z)dsdx'dy'
\end{align*}
where $q_t(x,y; x',y')$ is the $(2d)$-dimensional heat kernel (i.e. fundamental solution) associated with the operator $\cA_\bx$ \eqref{UE}. Since the latter operator is uniformly elliptic, we have Gaussian upper (and lower) bounds for the heat kernel (cf. \cite[Theorem 1]{IKO})
\begin{align*}
&q_t(x,y;x',y')\le \frac{C}{t^{\frac{2d}{2}}}e^{-\frac{|x-x'|^2+|y-y'|^2}{Ct}}
\end{align*}
for any $t>0, x,y,x',y'\in\R^d$, where $C=C\left(d, \{\sigma_k\}_{k\in K}, \lambda\right)$ finite. One can show, for details cf. the proof of \cite[Proposition 5, pp. 615]{FH}, that for some $C=C(d,T)$ and any $t\in[0,T]$, $x,y,x',y'\in\R^d$ and $d>1$ (hence $2d>2$), we have that
\begin{align*}
&\int_0^tq_{s}(x,y;x',y')ds\le \int_0^t\frac{C}{s^{\frac{2d}{2}}}e^{-\frac{|x-x'|^2+|y-y'|^2}{Cs}}ds\\
&\le C\left(|x-x'|^2+|y-y'|^2\right)^{\frac{2-2d}{2}}e^{-\frac{|x-x'|^2+|y-y'|^2}{C}}+C1_{\{|x-x'|^2+|y-y'|^2\le 1\}}.
\end{align*}
Now let us integrate in $x',y'$ of the preceding expression against $\theta^\ep(x'-y'+z)$, up to constant $C$ we have that 
\begin{align}
&\iint_{\R^{2d}}\left(|x-x'|^2+|y-y'|^2\right)^{\frac{2-2d}{2}}e^{-\frac{|x-x'|^2+|y-y'|^2}{C}}\theta^\ep(x'-y'+z)dx'dy' \nonumber\\
&\quad +\iint_{\R^{2d}}1_{\{|x-x'|^2+|y-y'|^2\le 1\}}\theta^\ep(x'-y'+z)dx'dy'  \nonumber\\
&\overset{\zeta=x'-y'+z,\; \gamma=x-x'}{=}\iint_{\R^{2d}}\left(|\gamma|^2+|y-x-z+\gamma+\zeta|^2\right)^{\frac{2-2d}{2}}e^{-\frac{|\gamma|^2+|y-x-z+\gamma+\zeta|^2}{C}}\theta^\ep(\zeta)d\gamma d\zeta   \label{1st-term}\\
&\quad +\iint_{\R^{2d}}1_{\{|\gamma|^2+|y-x-z+\gamma+\zeta|^2\le 1\}}\theta^\ep(\zeta)d\gamma d\zeta.\label{2nd-term}
\end{align}
We first argue about the exponential decay when $|y-x-z|\ge 4$. In this case, if we look at \eqref{2nd-term}, since the support of $\zeta$ is in $|\zeta|\le C_0\ep$, and the support of $\gamma$ is in $|\gamma|\le 1$ (otherwise the indicator in \eqref{2nd-term} is $0$), we see that when $|y-x-z|\ge 4$, then necessarily $|y-x-z+\gamma+\zeta|>1$ for $\ep$ small enough, rendering the indicator again $0$. Thus, we can focus solely on the first term \eqref{1st-term} and we argue exponential decay in $|y-x-z|$ when $|y-x-z|\ge 4$.

We separate the integral in $\gamma$ of \eqref{1st-term} according to $|\gamma|\le |y-x-z|/2$ and $|\gamma|>|y-x-z|/2$. In the former case, $|y-x-z+\gamma+\zeta|\ge |y-x-z|/4\ge 1$ for $\ep$ small enough,  thus we can bound part of the integral by (noting $2-2d<0$, and $\int\theta^\ep(\zeta)d\zeta=1$)
\begin{align*}
C\int_{|\gamma|\le |y-x-z|/2}e^{-\frac{|y-x-z|^2}{C}}d\gamma\le C|y-x-z|^de^{-\frac{|y-x-z|^2}{C}}\le C'e^{-\frac{|y-x-z|^2}{C'}}
\end{align*}
for some $C'>C$. In the latter case that $|\gamma|>|y-x-z|/2\ge 2$, we can bound the other part of the integral by
\begin{align*}
C\int_{|\gamma|> |y-x-z|/2}e^{-\frac{|\gamma|^2}{C}}d\gamma\le C|y-x-z|^{d-2}e^{-\frac{|y-x-z|^2}{C}}\le C'e^{-\frac{|y-x-z|^2}{C'}}.
\end{align*}
That is, we have the claimed exponential decay when $|y-x-z|\ge 4$. 

Now we turn to the case when $|y-x-z|<4$. Here the term \eqref{2nd-term} is merely bounded (due to $\int\theta^\ep=1$ and that the indicator is over a compact ball), so the total bound cannot be smaller than a constant bound. We now focus on the first integral \eqref{1st-term}. We separate two cases: $|y-x-z|\ge 4C_0\ep$ or $|y-x-z|< 4C_0\ep$.

If  $|y-x-z|\ge 4C_0\ep$ and $|\gamma|\le |y-x-z|/4$, then $|y-x-z+\gamma+\zeta|\ge |y-x-z|/4$ since $|\zeta|\le C_0\ep$, we can bound part of the integral \eqref{1st-term} by (ignore the exponential)
\begin{align*}
C\int_{|\gamma|\le|y-x-z|/4}|y-x-z|^{2-2d}d\gamma\le C|y-x-z|^{2-d}.
\end{align*}
If  $|y-x-z|\ge 4C_0\ep$ and $|\gamma|>|y-x-z|/4$, then we bound the other part of the integral by 
\begin{align*}
&C\int_{|\gamma|>|y-x-z|/4}|\gamma|^{2-2d}e^{-|\gamma|^2/C}d\gamma\le C\int_{1\ge |\gamma|>|y-x-z|/4}|\gamma|^{2-2d}d\gamma+\int_{|\gamma|>1}e^{-|\gamma|^2/C}d\gamma\\
&\le 
\begin{cases}
C|y-x-z|^{2-d}+C, \quad d\ge 3\\
C|\log|y-x-z||+C, \quad d=2
\end{cases}
\le C
\begin{cases}
|y-x-z|^{2-d}, \quad d\ge 3\\
|\log|y-x-z||, \quad d=2.
\end{cases}
\end{align*}
If on the other hand $|y-x-z|< 4C_0\ep$, then we have that $|y-x-z+\zeta|\le 5C_0\ep$. We note that for any $\gamma\in\R^d$, 
\[
|\gamma|\vee|y-x-z+\gamma+\zeta|\ge \frac{|x+z-y-\zeta|}{2},
\]
hence
\[
|\gamma|^2+|y-x-z+\gamma+\zeta|^2\ge \frac{|x+z-y-\zeta|^2}{4}.
\]
We separate the integral \eqref{1st-term} according to $\gamma\in\B\left(\frac{x+z-y-\zeta}{2}, |x+z-y-\zeta|\right)$ or otherwise. In the former case, we can bound part of the integral
\begin{align*}
&C\int_{\R^d}\int_{\B\left(\frac{x+z-y-\zeta}{2}, |x+z-y-\zeta|\right)} |x+z-y-\zeta|^{2-2d}\theta^\ep(\zeta)d\gamma d\zeta \le C\int |x+z-y-\zeta|^{2-d}\theta^\ep(\zeta)d\zeta\\
&\le C\ep^{-d}\int_{|\zeta|\le C_0\ep}|x+z-y-\zeta|^{2-d}d\zeta \le C\ep^{-d}\int_{|w|\le 5C_0\ep} |w|^{2-d}dw\le C\ep^{2-d}
\end{align*}
where we have bounded $\|\theta^\ep\|_\infty\le \ep^{-d}\|\theta\|_\infty$ and used $|y-x-z+\zeta|\le 5C_0\ep$. The second part of \eqref{1st-term} we bound by
\begin{align*}
&C\int_{\R^d}\int_{\R^d\backslash\B\left(\frac{x+z-y-\zeta}{2}, |x+z-y-\zeta|\right)} |\gamma|^{2-2d}e^{-|\gamma|^2/C}\theta^\ep(\zeta)d\gamma d\zeta
\le C\int\int_{|\gamma|>\frac{x+z-y-\zeta}{2}}|\gamma|^{2-2d}e^{-|\gamma|^2/C}\theta^\ep(\zeta)d\gamma d\zeta\\
&\le C
\begin{cases}
\int_{|\zeta|\le C_0\ep}|x+z-y-\zeta|^{2-d}\theta^\ep(\zeta)d\zeta, \quad d\ge 3\\
\int_{|\zeta|\le C_0\ep}|\log|x+z-y-\zeta||\theta^\ep(\zeta)d\zeta, \quad d=2
\end{cases}\\
&\le C\ep^{-d}
\begin{cases}
\int_{|w|\le 5C_0\ep}|w|^{2-d}d\zeta, \quad d\ge 3\\
\int_{|w|\le 5C_0\ep}|\log|w||d\zeta, \quad d=2
\end{cases}
\le C
\begin{cases}
\ep^{2-d}, \quad d\ge 3\\
|\log\ep|, \quad d=2.
\end{cases}
\end{align*}

To summarise, we showed here that when $|y-x-z|<4$, then a bound of the form  $C|y-x-z|^{2-d}\wedge C\ep^{2-d}$ holds for $d\ge 3$ and $C|\log|y-x-z||\wedge C|\log\ep|$ for $d=2$. This, together with the exponential decay when $|y-x-z|\ge 4$, completes the proof of the first two items of \eqref{unif-bd}. 

The $d=1$ case of \eqref{unif-bd} requires some changes. Since $2d=2$, we have that (cf. \cite[Proposition 5, pp. 615]{FH}) 
\begin{align*}
\int_0^tq_{s}(x,y;x',y')ds\le Ce^{-\frac{|x-x'|^2+|y-y'|^2}{C}}-C\log\left(|x-x'|^2+|y-y'|^2\right)1_{\{|x-x'|^2+|y-y'|^2\le 1\}}
\end{align*} 
and thus, 
\begin{align*}
\iint_{\R^{2}}\int_0^tq_{s}(x,y;x',y')dsdx'dy'&\le
C\iint_{\R^{2}}e^{-\frac{|\gamma|^2+|y-x-z+\gamma+\zeta|^2}{C}}\theta^\ep(\zeta)d\gamma d\zeta  \\
&\quad -C\iint_{\R^{2d}}\log\left(|\gamma|^2+|y-x-z+\gamma+\zeta|^2\right)1_{\{|\gamma|^2+|y-x-z+\gamma+\zeta|^2\le 1\}}\theta^\ep(\zeta)d\gamma d\zeta.
\end{align*}
Proceeding similarly as the $d\ge 2$ case yields the thesis, where the constant bound is essentially due to the integrability of $-\log r$ function near $r=0$. The gradient bounds \eqref{unif-bd-grad} can also be proved analogously; we only comment that we  start with
\begin{align*}
\nabla_x r^{\ep,z}(T-t,x,y)=&\int_{[0,t]\times\R^{2d}}\nabla_x q_{t-s}(x,y;x',y')\theta^\ep(x'-y'+z)dsdx'dy'
\end{align*}
and the fact that due to $\cA_\bx$ uniformly elliptic, the gradient of its $(2d)$-dimensional heat kernel satisfies 
\begin{align*}
|\nabla_x q_t(x,y;x',y')|\le \frac{C}{t^{\frac{1+2d}{2}}}e^{-\frac{|x-x'|^2+|y-y'|^2}{Ct}}
\end{align*}
for any $t>0$, $x,y,x',y'\in\R^d$ 
(cf. \cite[Theorem 1]{IKO}). Correspondingly, we have that for any $d\ge 1$ (cf. \cite[Proposition 5, pp. 615]{FH}),
\begin{align*}
&\int_0^t|\nabla_xq_{s}(x,y;x',y')|ds\le \int_0^t\frac{C}{s^{\frac{1+2d}{2}}}e^{-\frac{|x-x'|^2+|y-y'|^2}{Cs}}ds\\
&\le C\left(|x-x'|^2+|y-y'|^2\right)^{\frac{1-2d}{2}}e^{-\frac{|x-x'|^2+|y-y'|^2}{C}}+C1_{\{|x-x'|^2+|y-y'|^2\le 1\}}.
\end{align*}
The rest of the proof is analogous to that for the uniform bounds given above.
\end{proof}

\section{Bounding various terms}\label{sec:bounds}
Building on the It\^o-Tanaka procedure of Section \ref{ito-tanaka}, still working on the original probability space $(\Omega,\cF,\P)$ we complete the proof of Proposition \ref{ppn:tanaka} by showing the negligibility of all the minor terms in \eqref{ito-terms}. 

To prepare, for each $N\in\N$ let us denote by $[0,\tau_i^N)$ the lifespan of particle $i=1,...,N$ in our system, namely at the stopping time $\tau_i^N\in(0,\infty]$ particle $i$ is removed from the system due to the coagulation rule. Let us also consider on the same probability space $(\Omega,\cF,\P)$ an auxiliary ``free'' particle system $(x_i^f(t))_{i=1}^\infty$, that obeys the same dynamics as \eqref{SDE}, with $x_i^f(0)=x_i(0)$, $i\in\N$, and the {\it{same}} driving Brownian motions $\{W_t^k\}_{k\in K}$, $\{\beta_i(t)\}_{i=1}^\infty$, but without coagulation of masses. In other words, 
\begin{align}\label{SDE-free}
dx_i^f(t)=\sum_{k\in K}\sigma_k\left(x_i^f(t)\right)\,\circ\,dW^k_t+ \lambda \, d\beta_i(t),\quad i\in\N
\end{align}
with lifespan $[0,\infty)$. This auxiliary particle system dominates our true system in the following sense: $\P$-a.s.
\begin{align}\label{lifespan}
x_i^f(t)=x_i^N(t), \quad t\in[0,\tau_i^N), \; i=1,...,N
\end{align}
due to the fact that the coefficients of the SDE \eqref{SDE} do not depend on the mass parameter.

We start with a lemma proving that any fixed number of particles in the auxiliary system have a joint density that is uniformly bounded and decays exponentially at infinity. 
Recall that $\Gamma$ is the uniform upper bound on the initial density of an individual particle, and $R$ the maximal radius of its compact support.
\begin{proposition}\label{bdd-density}
Let $\ell\in\N, d\ge 1$ be fixed, and denote
\begin{align*}
\ovl{X}_t^f& := \left(x_1^f(t),...,x_\ell^f(t)\right)\in\R^{\ell d}\\
 \ovl x&:=\left(x_1,...,x_\ell\right)\in\R^{\ell d}
\end{align*}
where $x_i^f(t)$ are defined above. Then $\ovl X_t^f$ has a probability density $\bp_t(\ovl x)$ on $\R^{\ell d}$ that satisfies for any $t\in[0,T]$,
\begin{align}\label{joint-HK}
\bp_t(\ovl x)\le  c^{-1}e^{-c|\ovl x|}
\end{align}
where $c=c(d,\lambda,\ell, T, R, \Gamma, \{\sigma_k\}_{k\in K}, )$ is a positive constant.
\end{proposition}
\begin{proof}
Let us denote 
\begin{align*}
\ovl B_t&:=(\beta_1(t),...,\beta_\ell(t))\in\R^{\ell d}\\
\Sigma_k(\ovl x)&:=(\sigma_k(x_1),...,\sigma_k(x_\ell))\in\R^{\ell d}\\
\left(\nabla\Sigma_k\cdot\Sigma_k\right)(\ovl x)&:=\left(\left(\nabla\sigma_k\cdot\sigma_k\right)(x_1), ..., \left(\nabla\sigma_k\cdot\sigma_k\right)(x_\ell)\right)\in\R^{\ell d}
\end{align*}
where $\nabla\sigma_k\cdot\sigma_k$ is defined in \eqref{strat-corr}. 
Then, by \eqref{SDE-free} the $\ell$-tuple $\ovl X_t^f$ satisfies the SDE in $\R^{\ell d}$:

\begin{align*}
d\ovl X_t^f &= \sum_{k\in K}\Sigma_k(\ovl X_t^f)\,\circ\,d W^k_t+\lambda \, d\ovl B_t\\
&=\sum_{k\in K}\Sigma_k(\ovl X_t^f)\,d W^k_t+\frac{1}{2}\sum_{k\in K}\left(\nabla\Sigma_k\cdot\Sigma_k\right)(\ovl X_t^f)dt
+\lambda \,d\ovl B_t.
\end{align*}
This is an It\^o diffusion associated with the operator with coefficients of class $C_b^\infty(\R^{\ell d})$ that acts on functions on $\R^{\ell d}$,
\begin{align*}
\mathscr L=\frac{\lambda^2}{2}\Delta+\frac{1}{2}\sum_{k\in K}\text{tr}\left(\Sigma_k^T\Sigma_k(\ovl x)D^2\right)+\frac{1}{2}\sum_{k\in K}\left(\nabla\Sigma_k\cdot\Sigma_k\right)(\ovl x)\cdot D
\end{align*} 
thus admits a transition density (i.e. heat kernel) $q^{(\ell)}_t(\ovl x, \ovl y)$ that is bounded above (and below) by Gaussian kernel (cf. \cite[Theorem 1]{IKO}), i.e.
\[
q^{(\ell)}_t(\ovl x, \ovl y)\le CQ_t(\ovl x-\ovl y), \quad \forall t>0,\;  \ovl x, \ovl y\in\R^{\ell d},
\]
where $Q_t(\cdot)$ is the density of a centered Gaussian vector $\ovl Y^0_t$ in $\R^{\ell d}$ with covariance matrix $CtI$, for some $C=C\left(d, \ell, \{\sigma_k\}_{k\in K}, \lambda\right)$ finite. Further, the initial condition $\ovl X_0^f$  has a density $\bp_0(\ovl x)$ bounded above by $\Gamma^\ell$, and compactly supported in the ball $\B(0,\sqrt{\ell}R)$, due to Condition \ref{cond:init}.
Then, for any $t\in[0,T]$, the probability density of $\ovl X_t^f$
\begin{equation*}
\begin{aligned}
&\bp_t(\ovl x)=\left(q_t^{(\ell)}* \bp_0\right)(\ovl x)\le \int Q_t^{(\ell)}(\ovl x-\ovl y) \bp_0(\ovl y)d\ovl y\\
&=\E\left[\bp_0\left(\ovl x+\ovl Y_t^0\right)\right]\le \Gamma^\ell\P\left(\ovl x+\ovl Y_t^0\in\B\left(0, \sqrt{\ell}R\right)\right)\\
&\le \Gamma^\ell\P\left(|\ovl Y_t^0|\ge|\ovl x|-\sqrt{\ell}R\right)\le c^{-1}e^{-c\left(|x|-\sqrt{\ell}R\right)_+},
\end{aligned}
\end{equation*}
where we used the Gaussian tail of $\ovl Y^0_t$ in the last line, and
$c=c(d,\lambda,\ell, T, R, \Gamma, \{\sigma_k\}_{k\in K})$ is a positive constant. Adjusting the value of $c$ yields our thesis.
\end{proof}

\begin{lemma}\label{double-sum}
For any finite $T, d\ge 1$, we have that
\begin{align*}
\lim_{|z|\to 0}\limsup_{N\to\infty}\E\int_0^Tdt\frac{1}{N^2}\sum_{i\neq j\in\cN(t)}\left|v^{\ep,z}(t,x_i^N(t),x_j^N(t))\right|=0.
\end{align*}
\end{lemma}
\begin{proof}
Due to \eqref{lifespan}, it is enough to prove the thesis with the auxiliary system replacing the true system, with the former having infinite lifespan, i.e. to prove that 
\begin{align*}
\lim_{|z|\to 0}\limsup_{N\to\infty}\E\int_0^Tdt\frac{1}{N^2}\sum_{i\neq j=1}^N\left|v^{\ep,z}(t,x_i^f(t),x_j^f(t))\right|=0.
\end{align*}
(The same applies to all the subsequent lemmas in this section.)
Consider any fixed pair $(i,j)$, with $1\le i\neq j\le N$, and we go on to bound
\begin{align*}
\E\int_0^T\left|v^{\ep,z}(t,x_i^f(t),x_j^f(t))\right|dt.
\end{align*}
By Proposition \ref{bdd-density}, $(x_i^f(t),x_j^f(t))$ has a joint density $\bp_t(x,y)$ in $\R^{2d}$ satisfying the bound \eqref{joint-HK} with $\ell=2$, implying that 
\begin{align}  \label{dct}
\E\int_0^T\left|v^{\ep,z}(t,x_i^f(t),x_j^f(t))\right|dt&=\iint_{\R^{2d}}\int_0^T \left|v^{\ep,z}(t,x,y)\right|\bp_t(x,y)dtdxdy   \nonumber \\
&\le C\iint_{\R^{2d}}\sup_{t\in[0,T]} \left|v^{\ep,z}(t,x,y)\right|e^{-\sqrt{|x|^2+|y|^2}/C} dxdy\,,     
\end{align}
and thus the problem is reduced to prove that 
\[
\lim_{|z|\to 0}\limsup_{\ep\to0}\iint_{\R^{2d}}\sup_{t\in[0,T]} \left|v^{\ep,z}(t,x,y)\right|e^{-\sqrt{|x|^2+|y|^2}/C} dxdy=0.
\]
By Fatou's lemma, we have that 
\[
\limsup_{\ep\to0}\iint_{\R^{2d}}\sup_{t\in[0,T]} \left|v^{\ep,z}(t,x,y)\right|e^{-\sqrt{|x|^2+|y|^2}/C} dxdy\le \iint_{\R^{2d}}\limsup_{\ep\to0}\sup_{t\in[0,T]} \left|v^{\ep,z}(t,x,y)\right|e^{-\sqrt{|x|^2+|y|^2}/C} dxdy
\]
and thus it is sufficient to prove that 
\[
\lim_{|z|\to 0}\iint_{\R^{2d}}\limsup_{\ep\to0}\sup_{t\in[0,T]} \left|v^{\ep,z}(t,x,y)\right|e^{-\sqrt{|x|^2+|y|^2}/C} dxdy=0.
\]
By a change of variables, for any $t\in[0,T]$,
\begin{equation}
\begin{aligned}
v^{\ep,z}(T-t,x,y)&=r^{\ep,z}(T-t,x,y)-r^{\ep,0}(T-t,x,y)\\
&=\int_{[0,t]\times\R^{2d}}q_{t-s}(x,y; x',y')\left[\theta^\ep(x'-y'+z)-\theta^\ep(x'-y')\right]dsdx'dy'\\
&=\int_{[0,t]\times\R^{2d}}\left[q_{t-s}(x,y;x',y'+z)-q_{t-s}(x,y;x',y')\right]\theta^\ep(x'-y')dsdx'dy'
\end{aligned}
\end{equation}
where $q_t(x,y;x',y')$ is the $(2d)$-dimensional heat kernel (i.e. fundamental solution) associated with operator $\cA_\bx$ introduced in the proof of Proposition \ref{bdd-density}. Hence, as $\ep\to0$, 
\begin{align*}
v^{\ep,z}(T-t,x,y)\to \int_{[0,t]\times\R^{d}}\left[q_{t-s}(x,y;x',x'+z)-q_{t-s}(x,y;x',x')\right]dsdx'.
\end{align*}
Thus, the problem is reduced to prove that 
\begin{align}\label{dct-2}
\lim_{|z|\to0}\iint_{\R^{2d}}\sup_{t\in[0,T]}\left|\int_{[0,t]\times\R^{d}}\left[q_{t-s}(x,y;x',x'+z)-q_{t-s}(x,y;x',x')\right]dsdx'\right|e^{-\sqrt{|x|^2+|y|^2}/C} dxdy=0.
\end{align}
Since the integrand converges to zero pointwise (see below), to apply Vitali convergence theorem, it suffices to check the uniform integrability in the parameter $|z|\le 1$, with respect to the finite measure $e^{-\sqrt{|x|^2+|y|^2}/C}dxdy$, of
\[
A_{x,y,z}:=\sup_{t\in[0,T]}\int_{\R^{d}}\int_0^tq_{t-s}(x,y;x',x'+z)dsdx'\,;
\]
namely, to prove that for some $\delta>1$, we have that 
\[
\sup_{|z|\le 1}\iint_{\R^{2d}}|A_{x,y,z}|^\delta e^{-\sqrt{|x|^2+|y|^2}/C}dxdy<\infty.
\]
To this end, in the proof of Proposition \ref{bounds}, we already mentioned that for $d\ge2$,
\begin{align}\label{q-dom}
\int_0^tq_{t-s}(x,y;x',x'+z)ds\le C\left(|x-x'|^2+|y-z-x'|^2\right)^{\frac{2-2d}{2}}e^{-\frac{|x-x'|^2+|y-z-x'|^2}{C}}+C1_{\{|x-x'|^2+|y-z-x'|^2\le 1\}}
\end{align}
(and the $d=1$ case is similar, involving the logarithm).
Let us integrate this expression in the variable $x'$ over $\R^d$. Note that, for any $x'\in\R^d$, it holds that
\[
|x-x'|\vee |y-z-x'|\ge \frac{|x-(y-z)|}{2}.
\]
Thus, for any $x'$,
\[
|x-x'|^2+|y-z-x'|^2\ge \frac{|x-y+z|^2}{4}.
\]
We separate the integral in $x'$ into one in the ball $x'\in \B\left(\frac{x+(y-z)}{2},|x-y+z|\right)$ and the other outside the ball, and then it is easy to show that the first part is upper bounded by $C|x-y+z|^{2-d}$, and the second part by $C|x-y+z|^{2-d}1_{\{d\ge 3\}}+C|\log|x-y+z||1_{\{d=2\}}$. In the case $d=1$, a constant bound can be shown to hold.

Now we need to check, for some $\delta>1$, the uniform finiteness in $|z|\le 1$ of
\begin{align*}
\iint_{\R^{2d}}|A_{x,y,z}|^\delta e^{-\sqrt{|x|^2+|y|^2}/C}dxdy\le
\begin{cases}
C\iint_{\R^{2d}}|x-y+z|^{(2-d)\delta}e^{-\sqrt{|x|^2+|y|^2}/C}dxdy, \quad d\ge 3\\
C\iint_{\R^{2d}}|\log|x-y+z||^\delta e^{-\sqrt{|x|^2+|y|^2}/C}dxdy, \quad d=2\\
C\iint_{\R^{2d}}e^{-\sqrt{|x|^2+|y|^2}/C}dxdy, \quad d=1.
\end{cases}
\end{align*}
Taking $d\ge3$ for instance (the other cases being similar), 
\begin{align*}
&\iint_{\R^{2d}}|x-y+z|^{(2-d)\delta}e^{-\sqrt{|x|^2+|y|^2}/C}dxdy
\overset{u=\frac{x-y}{\sqrt{2}},\; v=\frac{x-y}{\sqrt{2}}}{=}
\iint_{\R^{2d}}|\sqrt{2}u+z|^{(2-d)\delta}e^{-\sqrt{|u|^2+|v|^2}/C}dudv\\
&\le \iint|\sqrt{2}u+z|^{(2-d)\delta}e^{-(|u|+|v|)/(2C)}dudv\le C\int|\sqrt{2}u+z|^{(2-d)\delta}e^{-|u|/(2C)}du\\
&\overset{|z|\le 1}{\le } C\int_{|u|\le 2}|\sqrt{2}u+z|^{(2-d)\delta}du+C\int_{|u|>2}e^{-|u|/(2C)}du \\
&\overset{|z|\le 1}{\le } C\int_{|w|\le 4}|w|^{(2-d)\delta}dw+C\le C',
\end{align*}
provided we take $\delta\le d/(d-2)$, where the constant $C'$ is independent of $|z|\le 1$. 

Note also that $z\mapsto A_{x,y,z}$ is continuous at $z=0$ by dominated convergence, since LHS of \eqref{q-dom} is continuous in $z$, a classical fact for Green functions, and RHS of \eqref{q-dom} is dominated by $C|x-y|^{2-2d}e^{-|x-x'|^2/C}+C1_{\{|x-x'|\le 1\}}$, for all $|z|\le |x-y|/4$ say (and $x\neq y$), which is integrable in $x'\in\R^d$.

With the integrand of \eqref{dct-2} converging pointwise to $0$ as $|z|\to0$ (i.e. for a.e. fixed $(x,y)$), and uniformly integrable with respect to the finite measure $e^{-\sqrt{|x|^2+|y|^2}/C}dxdy$, we have by Vitali convergence theorem that 
\[
\lim_{|z|\to0}\limsup_{\ep\to0}\iint_{\R^{2d}}\sup_{t\in[0,T]} \left|v^{\ep,z}(T-t,x,y)\right|e^{-\sqrt{|x|^2+|y|^2}/C} dxdy =0.
\]
Since the cardinality in the double sum in the thesis is (at most) $N^2$, this completes the proof. 
\end{proof}

\begin{lemma}\label{double-sum-grad}
For any finite $T, d\ge 1$, we have that
\begin{align*}
\lim_{|z|\to 0}\limsup_{N\to\infty}\E\int_0^Tdt\frac{1}{N^2}\sum_{i\neq j\in\cN(t)}\left|\nabla_xv^{\ep,z}(t,x_i^N(t),x_j^N(t))\right|=0.
\end{align*}
The same statement also holds for $\nabla_y$ instead of $\nabla_x$.
\end{lemma}

\begin{proof}
The proof is analogous to Lemma \ref{double-sum}, using the estimate for $|\nabla_x r^{\ep,z}(t,x,y)|$  \eqref{unif-grad-simp} instead of \eqref{unif-bd-simp}. First we switch to the free system, and it suffices to prove that 
\[
\lim_{|z|\to 0}\limsup_{N\to\infty}\E\int_0^Tdt\frac{1}{N^2}\sum_{i\neq j=1}^N\left|\nabla_xv^{\ep,z}(t,x_i^f(t),x_j^f(t))\right|=0.
\]
Considering any fixed pair of particles $(i,j)$, $1\le i\neq j\le N$, we have that
\begin{align}\label{dct-grad}
\E\int_0^T\left|\nabla_x v^{\ep,z}\left(t,x_i^f(t),x_j^f(t)\right)\right|dt
&=\iint_{\R^{2d}}\int_0^T \left|\nabla_xv^{\ep,z}(t,x,y)\right|\bp_t(x,y) dtdxdy\nonumber\\
&\le \iint_{\R^{2d}}\sup_{t\in[0,T]} \left|\nabla_xv^{\ep,z}(t,x,y)\right|e^{-\sqrt{|x|^2+|y|^2}/C}dxdy
\end{align}
by \eqref{joint-HK}.
Further, by a change of variables, for any $t\in[0,T]$,
\begin{align*}
\nabla_xv^{\ep,z}(T-t,x,y)&=\nabla_xr^{\ep,z}(T-t,x,y)-\nabla_xr^{\ep,0}(T-t,x,y)\\
&=\int_{[0,t]\times\R^{2d}}\nabla_xq_{t-s}(x,y; x',y')\left[\theta^\ep(x'-y'+z)-\theta^\ep(x'-y')\right]dsdx'dy'\\
&=\int_{[0,t]\times\R^{2d}}\left[\nabla_xq_{t-s}(x,y;x',y'+z)-\nabla_xq_{t-s}(x,y;x',y')\right]\theta^\ep(x'-y')dsdx'dy'.
\end{align*}
Hence, as $\ep\to0$, 
\begin{align*}
\nabla_xv^{\ep,z}(T-t,x,y)\to\int_{[0,t]\times\R^d}\left[\nabla_xq_{t-s}(x,y;x',x'+z)-\nabla_xq_{t-s}(x,y;x',x')\right]e^{-\sqrt{|x|^2+|y|^2}/C}dsdx'\,,
\end{align*}
and by Fatou's lemma it is sufficient to prove that 
\begin{align*}
\lim_{|z|\to0}\iint_{\R^{2d}}\sup_{t\in[0,T]}\left|\int_{[0,t]\times\R^d}\left[\nabla_xq_{t-s}(x,y;x',x'+z)-\nabla_xq_{t-s}(x,y;x',x')\right]dx'ds\right|e^{-\sqrt{|x|^2+|y|^2}/C}dxdy=0.
\end{align*}
To check that the integrand above is uniformly integrable in $|z|\le 1$, with respect to the finite measure $e^{-\sqrt{|x|^2+|y|^2}/C}dxdy$, we note that for $d\ge 1$ (already mentioned in the proof of Proposition \ref{bounds})
\[
\int_0^t|\nabla_xq_{t-s}(x,y;x',x'+z)|ds\le C\left(|x-x'|^2+|y-z-x'|^2\right)^{\frac{1-2d}{2}}e^{-\frac{|x-x'|^2+|y-z-x'|^2}{C}}+C1_{\{|x-x'|^2+|y-z-x'|^2\le 1\}}.
\]
By the same reasoning as in the proof of Lemma \ref{double-sum}, as we integrate the preceding expression in $x'$ over $\R^d$, we get an upper bound $C|x-y+z|^{1-d}1_{\{d\ge 2\}}+C|\log|x-y+z||1_{\{d=1\}}$. Then, for $d\ge 2$ we check for some $\delta>1$, the uniform finiteness in the parameter $|z|\le 1$ of
\[
\iint_{\R^{2d}}|x-y+z|^{(1-d)\delta}e^{-\sqrt{|x|^2+|y|^2}/C}dxdy,
\]
which is true provided we take $\delta<d/(d-1)$. By Vitali convergence theorem, we have that 
\[
\lim_{|z|\to0}\limsup_{\ep\to0}\iint_{\R^{2d}}\sup_{t\in[0,T]} \left|\nabla_xv^{\ep,z}(T-t,x,y)\right|e^{-\sqrt{|x|^2+|y|^2}/C} dxdy =0
\]
which leads to our thesis, as in Lemma \ref{double-sum}.
\end{proof}

\begin{remark}
By Lemmas \ref{double-sum} and \ref{double-sum-grad}, the terms $H_2,H_4, H_5,H_8, H_9$ are negligible in the sense of \eqref{negligible}, once we bound test functions $\phi,\psi$ and their derivatives, as well as the $C_b^\infty$ functions $Q^{\alpha\beta}$ and its partial derivatives above by constants. 
\end{remark}

\begin{lemma}\label{triple-sum}
For any finite $T, d\ge 1$, we have that
\begin{align*}
\lim_{|z|\to 0}\limsup_{N\to\infty}\E\int_0^Tdt\frac{1}{N^3}\sum_{i\neq  k\in\cN(t)}\theta^\ep(x_i^N(t)-x_k^N(t))\sum_{j\in\cN(t),j\neq i,k}\left|v^{\ep,z}(t,x_i^N(t),x_j^N(t))\right|=0.
\end{align*}
\end{lemma}

\begin{proof}
First switch to the free system. Consider any triple of particles $(i,j,k)$ with indices all distinct. To bound
\begin{align*}
\E\int_0^T\theta^\ep(x_i^f(t)-x_k^f(t))\left|v^{\ep,z}(t,x_i^f(t),x_j^f(t))\right|dt
\end{align*}
by Proposition \ref{bdd-density}, $(x_i^f(t),x_j^f(t),x_k^f(t))$ has a joint density $\bp_t(x_1,x_2,x_3)$ in $\R^{3d}$ satisfying the bound \eqref{joint-HK} with $\ell=3$. Thus, we can write
\begin{align*}
&\E\int_0^T\theta^\ep(x_i^f(t)-x_k^f(t))\left|v^{\ep,z}(t,x_i^f(t),x_j^f(t))\right|dt\\
&=\iiint_{\R^{3d}}\int_0^T\theta^\ep(x_1-x_3)\left|v^{\ep,z}(t,x_1,x_2)\right|\bp_t(x_1,x_2,x_3)dtdxdydz\\
&\le C\iiint_{\R^{3d}}\int_0^T\theta^\ep(x_1-x_3)\left|v^{\ep,z}(t,x_1,x_2)\right|e^{-\sqrt{\sum_{\ell=1}^3|x_\ell|^2}/C}dtdx_1dx_2dx_3
\end{align*}
Integrating in $x_3$ first and using $\int\theta^\ep(x_1-x_3)dx_3=1$, we bound the above integral above by
\begin{align*}
C\iint_{\R^{2d}}\int_0^T \left|v^{\ep,z}(t,x_1,x_2)\right|e^{-\sqrt{|x_1|^2+|x_2|^2}/C}dtdx_1dx_2
\end{align*}
This is the same integral apprearing in the proof of Lemma \ref{double-sum}, which has been shown to tend to zero as $\ep\to0$ followed by $|z|\to0$. Since the cardinality in the triple sum of the thesis is (at most) $N^3$, this completes the proof.
\end{proof}

\begin{remark}
By Lemma \ref{triple-sum}, the term $H_6$ \eqref{H6} is negligible in the sense of \eqref{negligible} (by considering it term by term), after bounding the test functions by constants.
\end{remark}

\begin{lemma}\label{lem:H6}
For any finite $T, d\ge 1$, we have that
\begin{align*}
\limsup_{N\to\infty}\E\int_0^Tdt\frac{1}{N^3}\sum_{i\neq  j\in\cN(t)}\theta^\ep(x_i^N(t)-x_j^N(t))\left|v^{\ep,z}(t,x_i^N(t),x_j^N(t))\right|=0.
\end{align*}
\end{lemma}

\begin{proof}
First switch to the free system. Consider any pair of particles $(i,j)$ with $1\le i\neq j\le N$. By \eqref{unif-bd-simp}, we have that 
\begin{align*}
&\E\frac{1}{N}\int_0^T\theta^\ep(x_i^f(t)-x_j^f(t))\left|v^{\ep,z}(t,x_i^f(t),x_j^f(t))\right|dt\le \kappa(\ep)\; \E\int_0^T\theta^\ep(x_i^f(t)-x_j^f(t))dt,
\end{align*}
where 
\begin{align}\label{error}
\kappa(\ep):=C
\begin{cases}
\ep^{2-d}N^{-1}, \quad d\ge 3\\
|\log\ep|N^{-1}, \quad d=2\\
N^{-1}, \quad d=1,
\end{cases}
\end{align} 
which vanishes as $N\to\infty$ by \eqref{local}.
By Proposition \ref{bdd-density} with $\ell=2$, we can write
\begin{align*}
\E\int_0^T\theta^\ep(x_i^f(t)-x_j^f(t))dt&=\int_0^T\iint_{\R^{2d}}\theta^\ep(x-y)\bp_t(x,y)dxdydt\\
&\le CT\iint_{\R^{2d}}\theta^\ep(x-y)e^{-\sqrt{|x|^2+|y|^2}/C}dxdy\\
&\le C\int_{\R^d}e^{-|x|/C}dx\le C'
\end{align*}
using $\int\theta^\ep(x-y)dy=1$. Since the double sum has cardinality at most $N^2$, the claim is proved.
\end{proof}

\begin{remark}
By Lemma \ref{lem:H6}, the term $H_7$ \eqref{H7} is negligible in the sense of \eqref{negligible}, after bounding the test functions by constants.
\end{remark}

Now we turn to the terms $B_1,B_2,B_3$ related to the martingales.

\begin{lemma}\label{quadruple}
For any finite $T, d\ge 1$, we have that
\begin{align*}
\lim_{|z|\to 0}\limsup_{N\to\infty}\E\int_0^Tdt\frac{1}{N^4}\left(\sum_{i\neq j\in\cN(t)}\left|v^{\ep,z}(t,x_i^N(t),x_j^N(t))\right|\right)^2=0.
\end{align*}
\end{lemma}

\begin{proof}
First switch to the free system. Expanding the square, we have
\begin{align*}
&\frac{1}{N^4}\left(\sum_{i\neq j=1}^N\left|v^{\ep,z}(t,x_i^f(t),x_j^f(t))\right|\right)^2\\
&= \frac{1}{N^4}\sum_{i\neq j\neq k\neq \ell =1}^N\left|v^{\ep,z}(t,x_i^f(t),x_j^f(t))v^{\ep,z}(t,x_k^f(t),x_\ell^f(t))\right|\\
&+\frac{1}{N^4}\sum_{i\neq j\neq \ell =1}^N\left|v^{\ep,z}(t,x_i^f(t),x_j^f(t))v^{\ep,z}(t,x_i^f(t),x_\ell^f(t))\right|\\
&+\frac{1}{N^4}\sum_{i\neq j\neq k =1}^N\left|v^{\ep,z}(t,x_i^f(t),x_j^f(t))v^{\ep,z}(t,x_k^f(t),x_i^f(t))\right|\\
&+\frac{1}{N^4}\sum_{j\neq i\neq \ell =1}^N\left|v^{\ep,z}(t,x_i^f(t),x_j^f(t))v^{\ep,z}(t,x_j^f(t),x_\ell^f(t))\right|\\
&+\frac{1}{N^4}\sum_{j\neq i\neq k =1}^N\left|v^{\ep,z}(t,x_i^f(t),x_j^f(t))v^{\ep,z}(t,x_k^f(t),x_j^f(t))\right|\\
&+\frac{1}{N^4}\sum_{i\neq j =1}^N\left|v^{\ep,z}(t,x_i^f(t),x_j^f(t))\right|^2\\
&+\frac{1}{N^4}\sum_{i\neq j=1}^N\left|v^{\ep,z}(t,x_i^f(t),x_j^f(t))v^{\ep,z}(t,x_j^f(t),x_i^f(t))\right|
\end{align*}
The last six terms are negligible, since the cardinality in their sums are at most $N^3$ (due to repeated indices), and by the bound \eqref{unif-bd-simp}, each individual quadratic term, e.g.
\begin{align}\label{negli-bd}
\frac{1}{N}\left|v^{\ep,z}(t,x_i^f(t),x_j^f(t))v^{\ep,z}(t,x_i^f(t),x_\ell^f(t))\right|\le \kappa(\ep)
\left|v^{\ep,z}(t,x_i^f(t),x_j^f(t))\right|
\end{align}
where $\kappa(\ep)$ defined in \eqref{error}, and hence the second sum above
\begin{align*}
&\frac{1}{N^4}\sum_{i\neq j\neq \ell =1}^N\left|v^{\ep,z}(t,x_i^f(t),x_j^f(t))v^{\ep,z}(t,x_i^f(t),x_\ell^f(t))\right|\\
&\le \frac{\kappa(\ep)}{N^3}\sum_{i\neq j\neq \ell =1}^N\left|v^{\ep,z}(t,x_i^f(t),x_j^f(t))\right|\le \frac{\kappa(\ep)}{N^2}\sum_{i\neq j =1}^N\left|v^{\ep,z}(t,x_i^f(t),x_j^f(t))\right|.
\end{align*}
Taking expectation and integrating in time, it is negligible by the statement of Lemma \ref{double-sum}. The other sums can be handled similarly.

Now we deal with the first (principle) term, where the cardinality of the sum is $O(N^4)$ and the indices are all distinct, thus it suffices to consider any fixed quadruple of particles $(i,j,k,\ell)$.
\begin{align*}
\E\int_0^T  \left|v^{\ep,z}(t,x_i^f(t),x_j^f(t))v^{\ep,z}(t,x_k^f(t),x_\ell^f(t))\right| dt.
\end{align*}
By Proposition \ref{bdd-density}, $(x_i^f(t),x_j^f(t),x_k^f(t),x_\ell^f(t))$ has a joint density $\bp_t(x_1,x_2,x_3,x_4)$ in $\R^{4d}$ satisfying the bound \eqref{joint-HK} with $\ell=4$. Thus, we can write
\begin{align*}
&\E\int_0^T \left|v^{\ep,z}(t,x_i^f(t),x_j^f(t))v^{\ep,z}(t,x_k^f(t),x_\ell^f(t))\right|  dt\\
&=\int_{\R^{4d}}\int_0^T  \left|v^{\ep,z}(t,x_1,x_2)v^{\ep,z}(t,x_3,x_4)\right|\bp_t(x_1,x_2,x_3,x_4)dtdx_1dx_2dx_3dx_4\\
&\le C \int_{\R^{4d}}\int_0^T  \left|v^{\ep,z}(t,x_1,x_2)v^{\ep,z}(t,x_3,x_4)\right| e^{-\sqrt{\sum_{i=1}^4|x_i|^2}/C}dtdx_1dx_2dx_3dx_4\\
&\le   C \int_0^T\left(\iint_{\R^{2d}} \left|v^{\ep,z}(t,x_1,x_2)\right| e^{-\frac{1}{\sqrt{2}C}\sqrt{|x_1|^2+|x_2|^2}}dx_1dx_2\right)^2dt.
\end{align*} 
As shown in the proof of Lemma \ref{double-sum},
\begin{align*}
\lim_{|z|\to0}\limsup_{\ep\to0}\iint_{\R^{2d}} \sup_{t\in[0,T]}\left|v^{\ep,z}(t,x_1,x_2)\right|e^{-\frac{1}{\sqrt{2}C}\sqrt{|x_1|^2+|x_2|^2}}dx_1dx_2=0,
\end{align*}
and hence, by the preceding inequality, we also have that
\begin{align*}
\lim_{|z|\to0}\limsup_{\ep\to0}\E\int_0^T  \left|v^{\ep,z}(t,x_i^f(t),x_j^f(t))v^{\ep,z}(t,x_k^f(t),x_\ell^f(t))\right| dt=0.
\end{align*}
Since the cardinality in the quadruple sum is (at most) $N^4$, this completes the proof.
\end{proof}

\begin{lemma}\label{quadruple-grad}
For any finite $T, d\ge 1$, we have that
\begin{align*}
\lim_{|z|\to 0}\limsup_{N\to\infty}\E\int_0^Tdt\frac{1}{N^4}\left(\sum_{i\neq j \in\cN(t)}\left|\nabla v^{\ep,z}(t,x_i^N(t),x_j^N(t))\right|\right)^2=0.
\end{align*}
\end{lemma}

\begin{proof}
First switch to the free system.
The proof is analogous to that of Lemmas \ref{quadruple}, just using the estimate \eqref{unif-grad-simp} instead of \eqref{unif-bd-simp}, and Lemma \ref{double-sum-grad} instead of Lemma \ref{double-sum}. Just note that when performing the analogous step to \eqref{negli-bd} to control the negligible terms, by  \eqref{unif-grad-simp}
\begin{align*}
\frac{1}{N}\left|\nabla v^{\ep,z}(t,x_i^f(t),x_j^f(t))\right|\left|\nabla v^{\ep,z}(t,x_i^f(t),x_\ell^f(t))\right|\le \wt\kappa(\ep)
\left|\nabla v^{\ep,z}(t,x_i^f(t),x_j^f(t))\right|
\end{align*}
where now 
\begin{align*}
\wt \kappa(\ep) := C
\begin{cases}
\ep^{1-d}N^{-1}, \quad d\ge 2\\
|\log\ep|N^{-1}, \quad d=1
\end{cases}
\end{align*}
is uniformly bounded by \eqref{local}.
\end{proof}

\begin{remark}
By Lemmas \ref{quadruple} and \ref{quadruple-grad}, the term $B_2$ \eqref{B2} is negligible in the sense of \eqref{negligible}, after bounding the test functions and $\sup_{k\in K}\|\sigma_k\|_\infty$ by constants. They also show that $B_1$ \eqref{B1} is negligible, since
\begin{align*}
\sum_{i\in\cN(t)}\left(\sum_{j\in\cN(t),j\ne i}\left|v^{\ep,z}(t,x_i^N(t),x_j^N(t))\right|\right)^2&\le 
\left(\sum_{i\neq j\in\cN(t)}\left|v^{\ep,z}(t,x_i^N(t),x_j^N(t))\right|\right)^2\\
\sum_{i\in\cN(t)}\left(\sum_{j\in\cN(t),j\ne i}\left|\nabla v^{\ep,z}(t,x_i^N(t),x_j^N(t))\right|\right)^2&\le 
\left(\sum_{i\neq j\in\cN(t)}\left|\nabla v^{\ep,z}(t,x_i^N(t),x_j^N(t))\right|\right)^2
\end{align*}
\end{remark}

\begin{lemma}\label{final-lem}
For any finite $T, d\ge 1$, we have that
\begin{align*}
\lim_{|z|\to 0}\limsup_{N\to\infty}\E\int_0^Tdt\frac{1}{N^5}\sum_{i\neq  k\in\cN(t)}\theta^\ep(x_i^N(t)-x_k^N(t))\left[\sum_{j\in\cN(t),j\neq i,k}v^{\ep,z}(t,x_i^N(t),x_j^N(t))\right]^2=0.
\end{align*}
\end{lemma}

\begin{proof}
By the elementary inequality $(\sum_{i=1}^La_i)^2\le L\sum_{i=1}^La_i^2$ for any $L\in\N$, and \eqref{unif-bd-simp}, we have that 
\begin{align*}
&\frac{1}{N^5}\sum_{i\neq  k\in\cN(t)}\theta^\ep(x_i^N(t)-x_k^N(t))\left[\sum_{j\in\cN(t),j\neq i,k}v^{\ep,z}(t,x_i^N(t),x_j^N(t))\right]^2\\
&\le \frac{1}{N^4}\sum_{i\neq  k\in\cN(t)}\theta^\ep(x_i^N(t)-x_k^N(t))\sum_{j\in\cN(t),j\neq i,k}\left|v^{\ep,z}(t,x_i^N(t),x_j^N(t))\right|^2\\
&\le \frac{\kappa(\ep)}{N^3}\sum_{i\neq  k\in\cN(t)}\theta^\ep(x_i^N(t)-x_k^N(t))\sum_{j\in\cN(t),j\neq i,k}\left|v^{\ep,z}(t,x_i^N(t),x_j^N(t))\right|
\end{align*}
where $\kappa(\ep)\to 0$ is as in \eqref{error}. 
The conclusion then follows from the statement of Lemma \ref{triple-sum}.
\end{proof}

\begin{remark}
By Lemma \ref{final-lem}, the term $B_3$ \eqref{B3} is negligible in the sense of \eqref{negligible}. Indeed, there are five terms inside the square, which we first use the elementary inequality $(\sum_{i=1}^5a_i)^2\le 5\sum_{i=1}^5a_i^2$, then we handle term by term, of which the first four are of the form in the lemma (after bounding the test functions by constants), and the last one by Lemma \ref{lem:H6} together with the fact $N^{-2}|v^{\ep,z}(\cdot)|^2\le \kappa(\ep)^2$.
\end{remark}

To conclude, we have thus far shown that Proposition \ref{ppn:tanaka} holds, by the discussion around \eqref{negligible}.

\section{Pathwise uniqueness of the SPDE \eqref{spde} and regularity of its solutions}\label{unique}
Consider the system%
\begin{align*}
du_{m}\left(  t,x\right)   &  =\left(  \mathcal{L}u_{m}\left(  t,x\right)
+F_{m}\left(  u\left(  t,x\right)  \right)  \right)  dt+\sum_{k\in K}%
\sigma_{k}\cdot\nabla u_{m}\left(  t,x\right)  dW_{t}^{k}\\
m &  =1,...,M,\qquad u=\left(  u_{1},...,u_{M}\right)
\end{align*}
where%
\begin{align}
F_{m}\left(  u\left(  t,x\right)  \right)  &=\sum_{n=1}^{m-1}u_{n}\left(
t,x\right)  u_{m-n}\left(  t,x\right)  -2u_{m}\left(  t,x\right)  \sum
_{n=1}^{M}u_{n}\left(  t,x\right)  \nonumber\\
\mathcal{L}u_{m} &  =\frac{\lambda^2}{2}\Delta u_{m}+\frac{1}{2}\operatorname{div}%
\left(  Q\left( x,x\right)  \nabla u_{m}\right)  \label{elliptic}\\
Q\left(  x,y\right)   &  =\sum_{k\in K}\sigma_{k}\left(  x\right)
\otimes\sigma_{k}\left(  y\right)\nonumber
\end{align}
with initial condition $\left(  r_1p_{1},...,r_Mp_{M}\right)  $, where $\sum_{m=1}^M r_m=1$,
satisfying%
\[
0\leq p_{m}\leq C,\qquad\int p_{m}\left(  x\right)  dx\leq1
\]
for every $m=1,...,M$. Notice that, also%
\[
\int p_{m}^{2}\left(  x\right)  dx\leq C\int p_{m}\left(  x\right)  dx\leq C,
\]
property often used below also for $u_{m}\left(  t,x\right)  $.

In the equations above \eqref{elliptic}, $\mathcal{L}$  is the resulting elliptic operator after the
reformulation of Stratonovich in It\^{o} form. Assume $\sigma_{k}\in C_{b}%
^{\infty}(\R^d; \R^d)$, $\operatorname{div}\sigma_{k}=0$.

\begin{definition}\label{def:spde-w}
\label{def weak sol}Given a filtered probability space $(\Omega,\cF,\{\cG_t\}_{t\ge0},\P)$ and Brownian motions $\{W_t^k\}_{k\in K}$,
by very weak solution we mean a progressively measurable process $u\left(
t,x\right)  $ such that, for some constant $U>0$,
\[
\mathbb{P}\left(  0\leq u_{m}\left(  t,x\right)  \leq U\text{ for all }\left(
t,x\right)  \right)  =1\text{ for all }m=1,...,M
\]%
\[
\mathbb{P}\left(  \int u_{m}\left(  t,x\right)  dx\leq1\text{ for all
}t\right)  =1\text{ for all }m=1,...,M
\]%
\begin{align*}
\left\langle u_{m}\left(  t\right)  ,\phi\right\rangle  &  =\left\langle
r_mp_{m},\phi\right\rangle +\int_{0}^{t}\left\langle u_{m}\left(
s\right)  ,\mathcal{L}^{\ast}\phi\right\rangle ds+\int_{0}^{t}\left\langle
F_{m}\left(  u\left(  s\right)  \right)  ,\phi\right\rangle ds\\
&  +\sum_{k\in K}\int_0^t\left\langle u_{m}\left(  s\right)  ,\sigma_{k}\cdot
\nabla\phi\right\rangle dW_{s}^{k}%
\end{align*}
for all $\phi\in C_{c}^{\infty}$, $\P$-a.s. If in addition they satisfy%
\[
\max_{m=1,...,M}\mathbb{E}\int_{0}^{T}\int\left\vert \nabla u_{m}\left(
t,x\right)  \right\vert ^{2}dxdt<\infty
\]
then they are called weak solutions.
\end{definition}

As already remarked for $p_{m}$, from the assumptions it follows that $\int
u_{m}^{2}\left(  t,x\right)  dx\leq U$ a.s., for every $m=1,...,M$.

\begin{lemma}
\label{lemma regularization}Very weak solutions are also weak solutions.
\end{lemma}

\begin{corollary}
\label{Coroll uniqueness}Weak solutions are pathwise unique.
\end{corollary}

\begin{proposition}\label{Proposition regularity nonlinear}
Let $n>d/4$ be an integer and let
$\left(  u_{m}\right)  _{m=1,...,M}$ be the unique solution given by Corollary
\ref{Coroll uniqueness}. If the initial conditions $u_{m}\left(  0\right)  $
belong to $W^{2n,2}\left(  \mathbb{R}^{d}\right)  $, $m=1,...,M$, then
$\left(  u_{m}\right)  _{m=1,...,M}$ has the following regularity:%
\[
\mathbb{E}\left[  \sup_{t\in\left[  0,T\right]  }\left\Vert u_{m}\left(
t,\cdot\right)  \right\Vert _{W^{2n,2}}^{2}\right]  +\mathbb{E}\int_{0}%
^{T}\left\Vert u_{m}\left(  t,\cdot\right)  \right\Vert _{W^{2n+1,2}}%
^{2}dt<\infty
\]
for every $m=1,...,M$. In particular, if $u_{m}\left(  0\right)  \in
C^{\infty}\left(  \mathbb{R}^{d}\right)  $, $m=1,...,M$, with square
integrable derivatives of all orders, then $\P$-a.s. one has $u_{m}\left(
t\right)  \in C^{\infty}\left(  \mathbb{R}^{d}\right)  $ for all $t\in\left[
0,T\right]  $ and $m=1,...,M$.
\end{proposition}
The proof of Proposition \ref{Proposition regularity nonlinear} is postponed to Section \ref{sec:free}, since it shares some technical ingredients with the proofs presented in that section.


\subsection{Proof of Lemma \ref{lemma regularization}}

\subsubsection{Preparation}

One can prove (cf. \cite{Fla 95 book}, \cite{Flandoli LNM}) that there
exists $\eta<1$ such that
\[
\sum_{k\in K}\left\Vert \sigma_{k}\cdot\nabla f\right\Vert _{L^{2}}^{2}%
\leq-2\eta\left\langle \mathcal{L}f,f\right\rangle
\]
for all $f\in C_{c}^{\infty}$. We want to exploit this property by means of an
energy inequality. For this purpose we need to apply It\^{o} formula but
$u_{m}$ does not have the necessary regularity, in particular the term
$\left\langle \mathcal{L}u_{m}\left(  t\right)  ,u_{m}\left(  t\right)
\right\rangle $ is not well defined (in a sense its good definition is our
thesis). Thus we take a smooth symmetric density $\rho$, with compact support
in the unitary ball $B\left(  0,1\right)  $, we define $\rho_{\epsilon}\left(
x\right)  =\epsilon^{-d}\rho\left(  \epsilon^{-1}x\right)  $ and set
\[
u_{m}^{\epsilon}\left(  t,x\right)  =\left(  \rho_{\epsilon}\ast u_{m}\left(
t\right)  \right)  \left(  x\right)  =\int\rho_{\epsilon}\left(  x-y\right)
u_{m}\left(  t,y\right)  dy.
\]
The process $u_{m}^{\epsilon}\left(  t,x\right)  $ is smooth in $x$; since
$\int u_{m}^{2}\left(  t,x\right)  dx\leq U$ a.s., from the smoothness \ and
compact support of $\rho$ plus Young's inequality for convolutions, it follows
in particular $u_{m}^{\epsilon}\left(  t\right)  \in W^{2,2}$ a.s. (but the
$W^{2,2}$-norm depends on $\epsilon$), hence the term $\left\langle
\mathcal{L}u_{m}^{\epsilon}\left(  t\right)  ,u_{m}^{\epsilon}\left(
t\right)  \right\rangle $ is well defined.

From the weak formulation, using a test function of the form $\rho_{\epsilon
}\ast\phi$ and the arbitrariety of $\phi$, we easily get%
\begin{align*}
u_{m}^{\epsilon}\left(  t\right)    & =u_{m}^{\epsilon}\left(  0\right)
+\int_{0}^{t}\left(  \mathcal{L}u_{m}^{\epsilon}\left(  s\right)
+\rho_{\epsilon}\ast F_{m}\left(  u\left(  s\right)  \right)  \right)  ds\\
& +\sum_{k\in K}\int_{0}^{t}\rho_{\epsilon}\ast\left(  \sigma_{k}\cdot\nabla
u_{m}\left(  s\right)  \right)  dW_{s}^{k}%
\end{align*}
where all processes can be interpreted, for instance, as $L^{2}$-valued
processes and where the term $\rho_{\epsilon}\ast\left(  \sigma_{k}\cdot\nabla
u_{m}\left(  t\right)  \right)  $ is a short notation for
\[
\left[  \rho_{\epsilon}\ast\left(  \sigma_{k}\cdot\nabla u_{m}\left(
t\right)  \right)  \right]  \left(  x\right)  :=\int\nabla_{x}\rho_{\epsilon
}\left(  x-y\right)  \cdot\sigma_{k}\left(  y\right)  u_{m}\left(  t,y\right)
dy.
\]
Now $u_{m}^{\epsilon}\left(  t\right)  $ is regular enough to apply It\^{o}
formula (cf. \cite{KrylovRozovskii}, \cite{Pardoux}, \cite{PrevRoeckner}):%
\begin{align*}
d\left\Vert u_{m}^{\epsilon}\left(  t\right)  \right\Vert _{L^{2}}^{2} &
=2\left\langle \mathcal{L}u_{m}^{\epsilon}\left(  t\right)  ,u_{m}^{\epsilon
}\left(  t\right)  \right\rangle dt+2\left\langle \rho_{\epsilon}\ast
F_{m}\left(  u\left(  t\right)  \right)  ,u_{m}^{\epsilon}\left(  t\right)
\right\rangle dt\\
&  +2\sum_{k\in K}\left\langle \rho_{\epsilon}\ast\left(  \sigma_{k}%
\cdot\nabla u_{m}\left(  t\right)  \right)  ,u_{m}^{\epsilon}\left(  t\right)
\right\rangle dW_{t}^{k}+\sum_{k\in K}\left\Vert \rho_{\epsilon}\ast\left(
\sigma_{k}\cdot\nabla u_{m}\left(  t\right)  \right)  \right\Vert _{L^{2}}%
^{2}dt.
\end{align*}
Then we introduce the commutators. Let us write%
\[
\rho_{\epsilon}\ast\left(  \sigma_{k}\cdot\nabla u_{m}\left(  t\right)
\right)  =\sigma_{k}\cdot\nabla u_{m}^{\epsilon}\left(  t\right)
+R_{m,k}^{\epsilon}\left(  t\right)
\]
where $R_{m,k}^{\epsilon}\left(  t\right)  $ is defined by the identity. We
have%
\begin{align*}
d\left\Vert u_{m}^{\epsilon}\left(  t\right)  \right\Vert _{L^{2}}^{2} &
=2\left\langle \mathcal{L}u_{m}^{\epsilon}\left(  t\right)  ,u_{m}^{\epsilon
}\left(  t\right)  \right\rangle dt+2\left\langle \rho_{\epsilon}\ast
F_{m}\left(  u\left(  t\right)  \right)  ,u_{m}^{\epsilon}\left(  t\right)
\right\rangle dt\\
&  +2\sum_{k\in K}\left\langle R_{m,k}^{\epsilon}\left(  t\right)
,u_{m}^{\epsilon}\left(  t\right)  \right\rangle dW_{t}^{k}+\sum_{k\in
K}\left\Vert \sigma_{k}\cdot\nabla u_{m}^{\epsilon}\left(  t\right)
+R_{m,k}^{\epsilon}\left(  t\right)  \right\Vert _{L^{2}}^{2}dt
\end{align*}
where we have used the fact that, being $\operatorname{div}\sigma_{k}=0$,
\[
\left\langle \sigma_{k}\cdot\nabla u_{m}^{\epsilon}\left(  t\right)
,u_{m}^{\epsilon}\left(  t\right)  \right\rangle =0.
\]
For every $\delta>0$ we have $2ab\leq\delta a^{2}+\delta^{-1}b^{2}$, hence%
\begin{align*}
\sum_{k\in K}\left\Vert \sigma_{k}\cdot\nabla u_{m}^{\epsilon}\left(
t\right)  +R_{m,k}^{\epsilon}\left(  t\right)  \right\Vert _{L^{2}}^{2} &
\leq\left(  1+\delta\right)  \sum_{k\in K}\left\Vert \sigma_{k}\cdot\nabla
u_{m}^{\epsilon}\left(  t\right)  \right\Vert _{L^{2}}^{2}\\
&  +\left(  1+\delta^{-1}\right)  \sum_{k\in K}\left\Vert R_{m,k}^{\epsilon
}\left(  t\right)  \right\Vert _{L^{2}}^{2}%
\end{align*}%
\[
\leq-\left(  1+\delta\right)  2\eta\left\langle \mathcal{L}u_{m}^{\epsilon
}\left(  t\right)  ,u_{m}^{\epsilon}\left(  t\right)  \right\rangle +\left(
1+\delta^{-1}\right)  \sum_{k\in K}\left\Vert R_{m,k}^{\epsilon}\left(
t\right)  \right\Vert _{L^{2}}^{2}.
\]
Choose $\delta>0$ such that $\left(  1+\delta\right)  2\eta=2-\zeta$ for some
$\zeta>0$. We get%
\begin{align*}
d\left\Vert u_{m}^{\epsilon}\left(  t\right)  \right\Vert _{L^{2}}^{2} &
\leq\zeta\left\langle \mathcal{L}u_{m}^{\epsilon}\left(  t\right)
,u_{m}^{\epsilon}\left(  t\right)  \right\rangle dt+2\left\langle
\rho_{\epsilon}\ast F_{m}\left(  u\left(  t\right)  \right)  ,u_{m}^{\epsilon
}\left(  t\right)  \right\rangle dt\\
&  +2\sum_{k\in K}\left\langle R_{m,k}^{\epsilon}\left(  t\right)
,u_{m}^{\epsilon}\left(  t\right)  \right\rangle dW_{t}^{k}+\left(
1+\delta^{-1}\right)  \sum_{k\in K}\left\Vert R_{m,k}^{\epsilon}\left(
t\right)  \right\Vert _{L^{2}}^{2}dt.
\end{align*}

One has
\[
2\left\langle \rho_{\epsilon}\ast F_{m}\left(  u\left(  t\right)  \right)
,u_{m}^{\epsilon}\left(  t\right)  \right\rangle \leq\left\Vert \rho
_{\epsilon}\ast F_{m}\left(  u\left(  t\right)  \right)  \right\Vert _{L^{2}%
}^{2}+\left\Vert u_{m}^{\epsilon}\left(  t\right)  \right\Vert _{L^{2}}^{2}%
\]
and
\begin{align*}
\left\Vert \rho_{\epsilon}\ast F_{m}\left(  u\left(  t\right)  \right)
\right\Vert _{L^{2}}^{2}  &  \leq\left\Vert F_{m}\left(  u\left(  t\right)
\right)  \right\Vert _{L^{2}}^{2}\leq\left\Vert F_{m}\left(  u\left(
t\right)  \right)  \right\Vert _{\infty}\left\Vert F_{m}\left(  u\left(
t\right)  \right)  \right\Vert _{L^{1}}\\
&  \leq C\left\Vert F_{m}\left(  u\left(  t\right)  \right)  \right\Vert
_{L^{1}}%
\end{align*}
for some deterministic constant $C$, since $0\leq u_{m}\left(  t,x\right)
\leq U$ a.s. for every $m$; moreover each term of $F_{m}\left(  u\left(
t\right)  \right)  $ has $L^{1}$-norm bounded by $U$: indeed a.s.%
\[
\int u_{k}\left(  t,x\right)  u_{h}\left(  t,x\right)  dx\leq U.
\]
We have found%
\begin{align*}
&  d\left\Vert u_{m}^{\epsilon}\left(  t\right)  \right\Vert _{L^{2}}%
^{2}+\zeta\int\left\vert \nabla u_{m}^{\epsilon}\left(  t,x\right)
\right\vert ^{2}dxdt\\
&  \leq Cdt+\left\Vert u_{m}^{\epsilon}\left(  t\right)  \right\Vert _{L^{2}%
}^{2}dt+2\sum_{k\in K}\left\langle R_{m,k}^{\epsilon}\left(  t\right)
,u_{m}^{\epsilon}\left(  t\right)  \right\rangle dW_{t}^{k}+\left(
1+\delta^{-1}\right)  \sum_{k\in K}\left\Vert R_{m,k}^{\epsilon}\left(
t\right)  \right\Vert _{L^{2}}^{2}dt.
\end{align*}
for some constant $C>0$.

\subsubsection{Commutator estimate}

Let us prove:

\begin{lemma}
\label{lemma commutator}%
\[
\sum_{k\in K}\left\Vert R_{m,k}^{\epsilon}\left(  t\right)  \right\Vert _{L^{2}}^{2}\leq
C_{Q}\int\left\vert u_{m}\left(  t,x\right)  \right\vert ^{2}dx.
\]
for a suitable constant $C_{Q}>0$.
\end{lemma}

\begin{proof}
Collecting the definitions, we have%
\begin{align*}
R_{m,k}^{\epsilon}\left(  t\right)   &  =\int\nabla_{x}\rho_{\epsilon}\left(
x-y\right)  \cdot\sigma_{k}\left(  y\right)  u_{m}\left(  t,y\right)
dy-\sigma_{k}\left(  x\right)  \cdot\nabla\int\rho_{\epsilon}\left(
x-y\right)  u_{m}\left(  t,y\right)  dy\\
&  =\int\nabla_{x}\rho_{\epsilon}\left(  x-y\right)  \cdot\left(  \sigma
_{k}\left(  y\right)  -\sigma_{k}\left(  x\right)  \right)  u_{m}\left(
t,y\right)  dy.
\end{align*}
One has%
\begin{align*}
\sum_{k\in K}\left\Vert R_{m,k}^{\epsilon}\left(  t\right)  \right\Vert
_{L^{2}}^{2}  &  =\sum_{k\in K}\iiint\nabla_{x}\rho_{\epsilon}\left(
x-y\right)  \cdot\left(  \sigma_{k}\left(  y\right)  -\sigma_{k}\left(
x\right)  \right)  u_{m}\left(  t,y\right)\\
&  \cdot\nabla_{x}\rho_{\epsilon}\left(  x-y^{\prime}\right)  \cdot\left(
\sigma_{k}\left(  y^{\prime}\right)  -\sigma_{k}\left(  x\right)  \right)
u_{m}\left(  t,y^{\prime}\right)  dydy^{\prime}dx
\end{align*}%
\begin{align*}
&  \leq U^{2}\iiint\left\Vert \sum_{k\in K}\left(  \sigma_{k}\left(
y\right)  -\sigma_{k}\left(  x\right)  \right)  \otimes\left(  \sigma
_{k}\left(  y^{\prime}\right)  -\sigma_{k}\left(  x\right)  \right)
\right\Vert \\
&  \cdot \left\vert \nabla_{x}\rho_{\epsilon}\left(  x-y\right)  \right\vert
\left\vert \nabla_{x}\rho_{\epsilon}\left(  x-y^{\prime}\right)  \right\vert
\left\vert u_{m}\left(  t,y\right)  \right\vert \left\vert u_{m}\left(
t,y^{\prime}\right)  \right\vert dydy^{\prime}dx
\end{align*}
where $\left\Vert \cdot\right\Vert $ here denotes the Euclidean matrix norm.
We show below that
\begin{equation}
\left\Vert \sum_{k\in K}\left(  \sigma_{k}\left(  y\right)  -\sigma_{k}\left(
x\right)  \right)  \otimes\left(  \sigma_{k}\left(  y^{\prime}\right)
-\sigma_{k}\left(  x\right)  \right)  \right\Vert \leq C_{Q}\left\vert
x-y\right\vert \left\vert x-y^{\prime}\right\vert \label{property of Q}%
\end{equation}
for some constant $C_{Q}$. Recall the support property of $\rho$; it implies
that $\rho_{\epsilon}$ has support in the ball $B\left(  0,\epsilon\right)  $,
hence the previous expression is bounded by%
\begin{align*}
&  \leq C_{Q}\epsilon^{-2}\int dx\int_{B\left(  x,\epsilon\right)  }%
dy\int_{B\left(  x,\epsilon\right)  }dy^{\prime}\epsilon^{-2d}\left\vert
\left(  \nabla\rho\right)  \left(  \frac{x-y}{\epsilon}\right)  \right\vert
\left\vert \left(  \nabla\rho\right)  \left(  \frac{x-y^{\prime}}{\epsilon
}\right)  \right\vert \\
&  \cdot\left\vert x-y\right\vert \left\vert x-y^{\prime}\right\vert
\left\vert u_{m}\left(  t,y\right)  \right\vert \left\vert u_{m}\left(
t,y^{\prime}\right)  \right\vert
\end{align*}%
\begin{align*}
&  \leq C_{Q}\iiint dxdydy^{\prime}\epsilon^{-2d}\left\vert \left(
\nabla\rho\right)  \left(  \frac{x-y}{\epsilon}\right)  \right\vert \left\vert
\left(  \nabla\rho\right)  \left(  \frac{x-y^{\prime}}{\epsilon}\right)
\right\vert \left\vert u_{m}\left(  t,y\right)  \right\vert \left\vert
u_{m}\left(  t,y^{\prime}\right)  \right\vert \\
&  =C_{Q}\int\left(  \int\epsilon^{-d}\left\vert \left(  \nabla\rho\right)
\left(  \frac{x-y}{\epsilon}\right)  \right\vert \left\vert u_{m}\left(
t,y\right)  \right\vert dy\right)  ^{2}dx.
\end{align*}
By Young's inequality for convolutions, this is bounded by
\[
\leq C_{Q}\int\left\vert u_{m}\left(  t,x\right)  \right\vert ^{2}dx.
\]

It remains to prove (\ref{property of Q}). It is equivalent to%
\[
\left\Vert Q\left(  y,y^{\prime}\right)  -Q\left(  y,x\right)  -Q\left(
x,y^{\prime}\right)  +Q\left(  x,x\right)  \right\Vert \leq C_{Q}\left\vert
x-y\right\vert \left\vert x-y^{\prime}\right\vert .
\]
It is sufficient to prove a similar estimate componentwise, for the
matrix-value function $Q$. Now%
\begin{align*}
&  Q_{ij}\left(  y,y^{\prime}\right)  -Q_{ij}\left(  y,x\right) \\
&  =\int_{0}^{1}\nabla_2 Q_{ij}\left( y, \alpha\left(  y-y^{\prime}\right)
+\left(  1-\alpha\right)  \left(  y-x\right)  \right)  d\alpha\cdot\left(
x-y^{\prime}\right)
\end{align*}%
\[
Q_{ij}\left(  x,y^{\prime}\right)  -Q_{ij}\left(  x,x\right)  =\int_{0}%
^{1}\nabla_2 Q_{ij}\left( x, \alpha\left(  x-y^{\prime}\right)  \right)
d\alpha\cdot\left(  x-y^{\prime}\right)
\]%
\begin{align*}
&  \partial_{h}Q_{ij}\left(  y,\alpha\left(  y-y^{\prime}\right)  +\left(
1-\alpha\right)  \left(  y-x\right)  \right)  -\partial_{h}Q_{ij}\left(
x,\alpha\left(  x-y^{\prime}\right)  \right) \\
&  =\int_{0}^{1}\nabla_2\partial_{h}Q_{ij}\left( y, \beta\left(  \alpha\left(
y-y^{\prime}\right)  +\left(  1-\alpha\right)  \left(  y-x\right)  \right)
+\left(  1-\beta\right)  \left(  \alpha\left(  x-y^{\prime}\right)  \right)
\right)  d\beta\cdot\left(  y-x\right)  \\
&+\int_0^1\nabla_1\partial_h Q_{ij}(\beta y+(1-\beta)x, \alpha(x-y'))d\beta\cdot(y-x).
\end{align*}
Collecting these identities, we get the required bound.
\end{proof}

\subsubsection{Conclusion}

From Lemma \ref{lemma commutator} and the a priori bounds on $u_{m}$ we get
\[
\sum_{k\in K}\left\Vert R_{m,k}^{\epsilon}\left(  t\right)  \right\Vert
_{L^{2}}^{2}\leq C_{Q}U\int\left\vert u_{m}\left(  t,x\right)  \right\vert
dx\leq C_{Q}U.
\]
Hence%

\[
d\left\Vert u_{m}^{\epsilon}\left(  t\right)  \right\Vert _{L^{2}}^{2}%
+\zeta\int\left\vert \nabla u_{m}^{\epsilon}\left(  t,x\right)  \right\vert
^{2}dxdt\leq C^{\prime}dt+2\sum_{k\in K}\left\langle R_{m,k}^{\epsilon}\left(
t\right)  ,u_{m}^{\epsilon}\left(  t\right)  \right\rangle dW_{t}^{k}%
\]
for a new constant $C^{\prime}>0$. It follows%
\begin{align*}
\zeta\mathbb{E}\int_{0}^{T}\int\left\vert \nabla u_{m}^{\epsilon}\left(
t,x\right)  \right\vert ^{2}dxdt  &  \leq C^{\prime}T+2\left(\sum_{k\in K}%
\mathbb{E}\int_{0}^{T}\left\langle R_{m,k}^{\epsilon}\left(  t\right)
,u_{m}^{\epsilon}\left(  t\right)  \right\rangle ^{2}dt\right)^{1/2}\\
&  \leq C^{\prime}T+2\left(\sum_{k\in K}\mathbb{E}\int_{0}^{T}\left\Vert
R_{m,k}^{\epsilon}\left(  t\right)  \right\Vert _{L^{2}}^{2}\left\Vert
u_{m}^{\epsilon}\left(  t\right)  \right\Vert _{L^{2}}^{2}dt\right)^{1/2}.
\end{align*}
As above, $\left\Vert u_{m}^{\epsilon}\left(  t\right)  \right\Vert _{L^{2}%
}^{2}\leq U$, and $\sum_{k\in K}\left\Vert R_{m,k}^{\epsilon}\left(  t\right)
\right\Vert _{L^{2}}^{2}\leq C_{Q}U$, hence%
\[
\zeta\mathbb{E}\int_{0}^{T}\int\left\vert \nabla u_{m}^{\epsilon}\left(
t,x\right)  \right\vert ^{2}dxdt\leq C^{\prime}T+2C_{Q}^{1/2}UT^{1/2}.
\]
The proof of the lemma is complete.

\subsection{Proof of Corollary \ref{Coroll uniqueness}}

Let $u=\left(  u_{1},...,u_{M}\right)  $, $u^{\prime}=\left(  u_{1}^{\prime
},...,u_{M}^{\prime}\right)  $ be two weak solutions. Set $v\left(  t\right)
=u\left(  t\right)  -u^{\prime}\left(  t\right)  $, $v_{m}\left(  t\right)
=u_{m}\left(  t\right)  -u_{m}^{\prime}\left(  t\right)  $. The regularity of
$v$ is sufficient to apply It\^{o} formula (cf. \cite{KrylovRozovskii},
\cite{Pardoux}, \cite{PrevRoeckner}):%
\begin{align*}
d\left\Vert v_{m}\left(  t\right)  \right\Vert _{L^{2}}^{2} &  =2\left\langle
\mathcal{L}v_{m}\left(  t\right)  ,v_{m}\left(  t\right)  \right\rangle
dt+2\left\langle F_{m}\left(  u\left(  t\right)  \right)  -F_{m}\left(
u^{\prime}\left(  t\right)  \right)  ,v_{m}\left(  t\right)  \right\rangle
dt\\
&  +2\sum_{k\in K}\left\langle \sigma_{k}\cdot\nabla v_{m}\left(  t\right)
,v_{m}\left(  t\right)  \right\rangle dW_{t}^{k}+\sum_{k\in K}\left\Vert
\sigma_{k}\cdot\nabla v_{m}\left(  t\right)  \right\Vert _{L^{2}}^{2}dt.
\end{align*}
Since $\left\langle \sigma_{k}\cdot\nabla v_{m}\left(  t\right)  ,v_{m}\left(
t\right)  \right\rangle =0$ and $\sum_{k\in K}\left\Vert \sigma_{k}\cdot\nabla
v_{m}\left(  t\right)  \right\Vert _{L^{2}}^{2}$ is bounded above by
$-2\left\langle \mathcal{L}v_{m}\left(  t\right)  ,v_{m}\left(  t\right)
\right\rangle $, we get%
\[
d\left\Vert v_{m}\left(  t\right)  \right\Vert _{L^{2}}^{2}\le 2\left\langle
F_{m}\left(  u\left(  t\right)  \right)  -F_{m}\left(  u^{\prime}\left(
t\right)  \right)  ,v_{m}\left(  t\right)  \right\rangle dt.
\]
Now%
\[
F_{m}\left(  u\right)  -F_{m}\left(  u^{\prime}\right)  =\sum_{n=1}%
^{m-1}\left(  u_{n}u_{m-n}-u_{n}^{\prime}u_{m-n}^{\prime}\right)  -2\left(
u_{m}\sum_{n=1}^{M}u_{n}-u_{m}^{\prime}\sum_{n=1}^{M}u_{n}^{\prime}\right)
\]
and each term of the form $u_{h}u_{k}-u_{h}^{\prime}u_{k}^{\prime}$ can be
estimated as%
\begin{align*}
\left\vert u_{h}u_{k}-u_{h}^{\prime}u_{k}^{\prime}\right\vert  &
\leq\left\vert u_{h}u_{k}-u_{h}u_{k}^{\prime}\right\vert +\left\vert
u_{h}u_{k}^{\prime}-u_{h}^{\prime}u_{k}^{\prime}\right\vert \\
&  \leq U\left\vert v_{k}\right\vert +U\left\vert v_{h}\right\vert
\end{align*}
hence%
\[
\left\vert F_{m}\left(  u\right)  -F_{m}\left(  u^{\prime}\right)  \right\vert
\leq C\sum_{n=1}^{M}\left\vert v_{n}\right\vert
\]
where $C$ depends on $U$ and $M$. It follows%
\begin{align*}
\frac{d}{dt}\left\Vert v_{m}\left(  t\right)  \right\Vert _{L^{2}}^{2} &
\leq2C\sum_{n=1}^{M}\int\left\vert v_{n}\left(  t,x\right)  \right\vert
\left\vert v_{m}\left(  t,x\right)  \right\vert dx\\
&  \le 2C\sum_{n=1}^{M}\left\Vert v_{n}\left(  t\right)  \right\Vert _{L^{2}}%
^{2}+2CM\left\Vert v_{m}\left(  t\right)  \right\Vert _{L^{2}}^{2}.
\end{align*}
Summing over $m$ and applying Gronwall lemma, we deduce $\left\Vert
v_{m}\left(  t\right)  \right\Vert _{L^{2}}^{2}=0$ for every $m$ and $t$,
which is pathwise uniqueness.

\section{Relative compactness of the empirical measure}
\label{tightness}
In this section, we show the tightness of the sequence of laws of $\{\mu ^{N,m}\}_{m\le M}$ taking values in $\cD_T(\cM_{+,1})^M$. 

\subsection{A general compactness criterion}

Let $\mathcal{M}_{+,1}\left(  \mathbb{R}^{d}\right)  $ be the set of positive
Borel measures on $\mathbb{R}^{d}$ with mass $\leq1$. Recall that, given
$R>0$, the set $\mathcal{K}_{R}\subset\mathcal{M}_{+,1}\left(  \mathbb{R}%
^{d}\right)  $ defined as%
\begin{align*}
\mathcal{K}_{R}=\Big\{  \mu\in\mathcal{M}_{+,1}\left(  \mathbb{R}^{d}\right)
:\int_{\mathbb{R}^{d}}\left\vert x\right\vert \mu\left(  dx\right)  \leq
R\Big\}
\end{align*}
is relatively compact in $\mathcal{M}_{+,1}\left(  \mathbb{R}^{d}\right)  $
endowed with the topology of weak convergence of measures.

The weak convergence on $\mathcal{M}_{+,1}\left(  \mathbb{R}^{d}\right)  $ can
be metrized in the following way. For every compact set $K\subset
\mathbb{R}^{d}$, the space $C\left(  K\right)  $ is separable; let $\left(
f_{n}^{K}\right)  _{n\in\mathbb{N}}$ be a dense sequence in $C\left(
K\right)  $ and define the function $\delta_{K}:\mathcal{M}_{+,1}\left(
K\right)  ^{2}\rightarrow\lbrack0,\infty)$ as%
\begin{align*}
\delta_{K}\left(  \mu,\nu\right)  =\sum_{n=1}^{\infty}2^{-n}\left(  \left\vert
\left\langle \mu,f_{n}^{K}\right\rangle -\left\langle \nu,f_{n}^{K}%
\right\rangle \right\vert \wedge1\right)  .
\end{align*}
This is a metric on $\mathcal{M}_{+,1}\left(  K\right)  $ and the metric space
$\left(  \mathcal{M}_{+,1}\left(  K\right)  ,d_{K}\right)  $ is complete and
separable; convergence in this metric is weak convergence of measures. Taking
a sequence of compact sets $K_{m}$ with $\cup_{m}K_{m}=\mathbb{R}^{d}$ and
proceeding in a similar way we may define a metric on $\mathcal{M}%
_{+,1}\left(  \mathbb{R}^{d}\right)  $. Rearranging the double procedure in a
single one, we may claim that there exists \ a sequence $\left(  f_{n}\right)
_{n\in\mathbb{N}}$, dense in $C\left(  K\right)  $ for every compact set
$K\subset\mathbb{R}^{d}$, such that
\begin{align*}
\delta\left(  \mu,\nu\right)  =\sum_{n=1}^{\infty}2^{-n}\left(  \left\vert
\left\langle \mu,f_{n}\right\rangle -\left\langle \nu,f_{n}\right\rangle
\right\vert \wedge1\right)
\end{align*}
is a metric on $\mathcal{M}_{+,1}\left(  \mathbb{R}^{d}\right)  $ and the
metric space $\left(  \mathcal{M}_{+,1}\left(  \mathbb{R}^{d}\right)
,\delta\right)  $ is complete and separable; convergence in this metric is
weak convergence of measures. Finally, we may take the sequence $\left(
f_{n}\right)  _{n\in\mathbb{N}}$, in $C_{c}^{\infty}\left(  \mathbb{R}%
^{d}\right)  $, by revising the previous construction from the beginning and
using the density of $C_{c}^{\infty}\left(  \mathbb{R}^{d}\right)  $ in
$C_{c}\left(  \mathbb{R}^{d}\right)  $.

Given $T>0$, consider the space of c\`adl\`ag functions
\begin{align*}
\mu_{\cdot}:\left[  0,T\right]  \rightarrow\mathcal{M}_{+,1}\left(
\mathbb{R}^{d}\right)
\end{align*}
where $\left(  \mathcal{M}_{+,1}\left(  \mathbb{R}^{d}\right)  ,\delta\right)
$ is considered as a metric space (hence continuity and limits of $t\mapsto
\mu_{t}$ are understood in this metric). Denote it by $\mathcal{D}\left(
\left[  0,T\right]  ;\mathcal{M}_{+,1}\left(  \mathbb{R}^{d}\right)  \right)
$ or more shortly as $\mathcal{D}_{T}\left(  \mathcal{M}_{+,1}\right)  $ and
endow it by the Skorohod topology, not recalled here, but denoted by $d$
below (cf. \cite[Ch. 3, Eq. (5.2)]{EK}).

Criteria of compacteness in $\left(  \mathcal{D}_{T}\left(  \mathcal{M}%
_{+,1}\right)  ,d\right)  $ are usually expressed by means of a modified
modulus of continuity, to account of jumps. When the jumps are very small (as
in our case), the classical modulus of continuity is sufficient, defined as
($\mu_{\cdot}\in\mathcal{D}_{T}\left(  \mathcal{M}_{+,1}\right)  $)%
\begin{align*}
\omega_{\gamma}\left(  \mu_{\cdot}\right)  =\sup_{\substack{s,t\in\left[
0,T\right]  \\\left\vert t-s\right\vert \leq\gamma}}\delta\left(  \mu_{s}%
,\mu_{t}\right)  .
\end{align*}

A sufficient condition for a (deterministic) sequence $\left\{  \mu_{\cdot}^{n}\right\}
\subset\mathcal{D}_{T}\left(  \mathcal{M}_{+,1}\right)  $ to be relatively
compact is:
\begin{proposition}\label{ppn:seq-cpt}
If (i). the family%
\begin{align*}
\left\{  \mu_{t}^{n};n\in\mathbb{N},t\in\left[  0,T\right]  \right\}
\subset\mathcal{M}_{+,1}\left(  \mathbb{R}^{d}\right)
\end{align*}
is relatively compact (see above a sufficient condition); and (ii).
\begin{align*}
\lim_{\gamma\rightarrow0}\sup_{n}\omega_{\gamma}\left(  \mu_{\cdot}%
^{n}\right)  =0,
\end{align*}
then $\left\{  \mu_{\cdot}^{n}\right\}  $ is relatively compact in
$\mathcal{D}_{T}\left(  \mathcal{M}_{+,1}\right)  $.
\end{proposition}

Let now $\left\{  \mu_{\cdot}^{n}\right\}  $ be a family of {\it{random}} elements of
$\mathcal{D}_{T}\left(  \mathcal{M}_{+,1}\right)  $ and let $\left\{
\mathbb{P}^{n}\right\}  $ be their laws. This sequence of laws is weakly relatively
compact if it is tight, namely if given $\epsilon>0$ there exists a compact
set $\mathcal{K}_{\epsilon}\subset\mathcal{D}_{T}\left(  \mathcal{M}%
_{+,1}\right)  $ such that
\begin{align*}
\mathbb{P}^{n}\left(  \mathcal{K}_{\epsilon}\right)  =\mathbb{P}\left(
\mu_{\cdot}^{n}\in\mathcal{K}_{\epsilon}\right)  \geq1-\epsilon.
\end{align*}
Proposition \ref{ppn:seq-cpt} yields the following practical sufficient condition:

\begin{proposition}\label{suff-rel-cpt}
If (i). there exists $C_{1}>0$ such that
\begin{align*}
\mathbb{E}\int_{\mathbb{R}^{d}}\left\vert x\right\vert \mu_{t}^{n}\left(
dx\right)  \leq C_{1}%
\end{align*}
for all $n$ and $t\in[0,T]$; 
(ii). and there exists $\beta>0$ such that, for every $\phi\in C_{c}^{\infty
}\left(  \mathbb{R}^{d}\right)  $, there exists $C_{\phi}>0$ such that
\begin{align*}
\mathbb{E}\left[  \left\vert \left\langle \mu_{t}^{n},\phi\right\rangle
-\left\langle \mu_{s}^{n},\phi\right\rangle \right\vert \right]  \leq C_{\phi
}\left\vert t-s\right\vert ^{\beta}%
\end{align*}
for all $t,s\in\left[  0,T\right]  $, 
then the sequence $\left\{  \mathbb{P}^{n}\right\}  $ is relatively compact.
\end{proposition}

\subsection{Application to our case}

Let $\{\mu_{\cdot}^{N,m}\}_{m\le M}$ be the empirical measures defined in \eqref{empirical}. Let us check the conditions of Proposition \ref{suff-rel-cpt} above, for each $\mu_{\cdot}^{N,m}$, $1\le m\le M$. Let us start with (ii).
We have, analogously to \eqref{key-identity}, for $\phi\in C_{c}^{\infty}\left(  \mathbb{R}^{d}\right)  $, $0\le s<t\le T$,

\begin{equation*}
\begin{aligned}
\left\langle \phi(x), \mu^{N,m}_t(dx)\right\rangle & = \left\langle \phi(x), \mu^{N,m}_s(dx)\right\rangle
+\int_s^tdr\frac{\lambda^2}{2N}\sum_{i\in\cN(r)}\Delta \phi(x_i^N(r))1_{\{m_i(r)=m\}}\\
&+\int_s^tdr\frac{1}{2N}\sum_{i\in\cN(r)}\operatorname{div}\left(Q(x_i^N(r),x_i^N(r))\nabla \phi(x_i^N(r))\right)1_{\{m_i(r)=m\}}\\
&+\int_s^tdr \frac{1}{N^2}\sum_{i\neq j\in\cN(r)}\theta^\ep(x_i^N(r)-x_j^N(r))\Big[\frac{m_i(r)}{m_i(r)+m_j(r)}\phi(x_i^N(r))1_{\{m_i(r)+m_j(r)=m\}}\\
&+\frac{m_j(r)}{m_i(r)+m_j(r)}\phi(x_j^N(r))1_{\{m_i(r)+m_j(r)=m\}}
 -\phi(x_i^N(r))1_{\{m_i(r)=m\}}-\phi(x_j^N(r))1_{\{m_j(r)=m\}}\Big]\\
& +\left(M^{1,D,\phi}_t-M^{1,D,\phi}_s\right) + \left(M^{2,D,\phi}_t-M^{2,D,\phi}_s\right) + \left(M_t^{J,\phi} - M_s^{J,\phi}\right).
\end{aligned}
\end{equation*}
Since $N(r)\le N$, we have that
\begin{equation*}
\begin{aligned}
\mathbb{E} \left\vert \left\langle \mu_{t}^{N,m},\phi\right\rangle
-\left\langle \mu_{s}^{N,m},\phi\right\rangle \right\vert   &
\leq C(\lambda)\left\vert t-s\right\vert \left\Vert \phi\right\Vert
_{C^2}\|Q\|_{C^1}\\
&+4\|\phi\|_\infty \E \int_s^tdr \frac{1}{N^2}\sum_{i\neq j\in\cN(r)}\theta^\ep(x_i^N(r)-x_j^N(r)) \\
& +\mathbb{E} \left\vert M^{1,D,\phi}_{t}-M^{1,D,\phi}_{s}\right\vert  +\mathbb{E} \left\vert M^{2,D,\phi}_{t}-M^{2,D,\phi}_{s}\right\vert  +\mathbb{E} \left\vert M^{J,\phi}_{t}-M^{J,\phi}_{s}\right\vert 
\end{aligned}
\end{equation*}
Similarly to \eqref{mg-van-1}, \eqref{mg-van-2}, we have that
\begin{align*}
\left[  \mathbb{E}\left\vert M^{1,D,\phi}_{t}-M^{1,D,\phi}_{s}\right\vert \right]  ^{2}  &
\leq\mathbb{E}\left[  \left\vert M^{1,D,\phi}_{t}-M^{1,D,\phi}_{s}\right\vert ^{2}\right]  \leq C(\lambda) (t-s)N^{-1}\|\phi\|_{C^1}^2
\end{align*}
\begin{align*}
\left[\mathbb{E}  \left\vert M^{2,D,\phi}_{t}-M^{2,D,\phi}_{s}\right\vert \right]  ^{2}  &
\leq\mathbb{E}\left[  \left\vert M^{2,D,\phi}_{t}-M^{2,D,\phi}_{s}\right\vert ^{2}\right]  \leq (t-s)\sum_{k\in K}\|\phi\|^2_{C^1}\|\sigma_k\|^2_\infty
\end{align*}
To deal with 
\[
\E \int_s^tdr \frac{1}{N^2}\sum_{i\neq j\in\cN(r)}\theta^\ep(x_i^N(r)-x_j^N(r)), 
\]
we consider each fixed pair $(i,j)$ of particles, and assume they both are active in $[s,t]$. By Proposition \ref{bdd-density}, $(x_i^N(t),x_j^N(t))$ has a joint density $\bp_t(x_1,x_2)$ satisfying the bound \eqref{joint-HK} with $\ell=2$. Thus, we have
\begin{align*}
&\E \int_s^t\theta^\ep(x_i^N(r)-x_j^N(r))dr =\int_s^t\theta^\ep(x_1-x_2)\bp_r(x_1,x_2)dx_1dx_2\\
&\le C(t-s)\iint_{\R^{2d}}\theta^\ep(x_1-x_2)e^{-\sqrt{|x_1|^2+|x_2|^2}/C}dx_1dx_2 \le C'(t-s).
\end{align*}
Finally, 
\begin{align*}
\left[ \mathbb{E} \left\vert M^{J,\phi}_{t}-M^{J,\phi}_{s}\right\vert \right]  ^{2} 
\leq\mathbb{E}\left[  \left\vert M^{J,\phi}_{t}-M^{J,\phi}_{s}\right\vert ^{2}\right]  &\leq 16\|\phi\|^2_\infty\E \int_s^tdr\frac{1}{N^3}\sum_{i\neq j\in\cN(r)}\theta^\ep\left(x_i^N(r)-x_j^N(r)\right) \\
&\le C'(t-s)N^{-1}\|\phi\|^2_\infty 
\end{align*}
by the previous estimate. Summarizing,%
\begin{align*}
\mathbb{E} \left\vert \left\langle \mu_{t}^{N,q},\phi\right\rangle
-\left\langle \mu_{s}^{N,q},\phi\right\rangle \right\vert  \leq
C\left(\phi,\{\sigma_k\}_{k\in K}, \lambda\right)\left\vert t-s\right\vert ^{1/2}.
\end{align*}
Concerning (i), for any $t\in[0,T]$,
\begin{align*}
\mathbb{E} \int_{\mathbb{R}^{d}}\left\vert x\right\vert \mu_{t}^{N,m}\left(
dx\right)  \leq\frac{1}{N}\mathbb{E} \sum_{i\in\cN(t)}\left\vert
x_{i}\left(  t\right)  \right\vert  \leq \mathbb{E}
\left\vert x_{1}\left(  t\right)  \right\vert
\end{align*}
by the exchangeability among the particles, and the RHS is finite, by  \eqref{joint-HK} with $\ell=1$.

\begin{remark}\label{joint-tight}
As a consequence of the tightness of $\big\{\mu^{N,m}_t:t\in[0,T]\big\}_{m\le M}$, $N\in\N$, the sequence of $\cD_T(\cM_{+,1})^M\times C([0,T];\R)^{|K|}$-valued random variables
\[
\big\{\big\{\mu^{N,m}_t:t\in[0,T]\big\}_{m\le M},\big\{W_t^k:t\in[0,T]\big\}_{k\in K}\big\}, \quad N\in \N
\] 
is also tight, as needed in Section \ref{sec:emp-mea}, where $C([0,T];\R)$ is endowed with the uniform topology. Note that the Brownian motions are independent of $N$.
\end{remark}

\section{Existence and boundedness of limit density}\label{exist-density}
In this section, we show that 
\begin{proposition}\label{ppn:density}
Any subsequential limit in law of $\{\mu_t^{N,m}(dx), t\in[0,T]\}_{m\le M}$ in $\cD_T(\cM_{+,1})^M$ is concentrated on absolutely continuous paths, and its density with respect to Lebesgue measure is uniformly bounded by the deterministic constant $\Gamma$ in Condition \ref{cond:init}. 
\end{proposition}

Denote on $(\Omega,\cF,\P)$ the empirical measure of all active particles (regardless of mass) in the true system
\[
\mu_t^N(dx):=\sum_{m=1}^M\mu_t^{N,m}(dx)=\frac{1}{N}\sum_{i\in\cN(t)}\delta_{x_i^N(t)}(dx), \quad  t\ge 0.
\]
Let us recall the auxiliary free system of particles $\{x_i^f(t)\}_{i=1}^\infty$ introduced in Section \ref{sec:bounds}. For each $N\in\N$, we denote their empirical measure 
\[
\mu_t^{N,f}(dx):= \frac{1}{N}\sum_{i=1}^{N}\delta_{x^f_i(t)}(dx), \quad t\ge 0.
\]
In addition to \eqref{lifespan}, we also have the following set inclusion holding, $\P$-a.s. for all $t\ge0$
\begin{align*}
\left\{x_i^N(t): i\in\cN(t)\right\}\subset\left\{x_i^f(t): i=1,...,N\right\}\subset\R^d
\end{align*}
and the domination between the empirical measures: $\P$-a.s. for any $N,m$ and $t$, and Borel set $B\subset\R^d$,
\begin{align}\label{domin}
\mu_t^{N,m}(B)\le \mu_t^N(B)\le \mu_t^{N,f}(B).
\end{align}
It is also clear that $\mu_0^{N,f}(dx)\to\sum_{m=1}^Mr_mp_m(x)dx$ (weakly), as $N\to\infty$, in $\cM_{+,1}(\R^d)$ in probability.

\medskip
We now argue that in order to show Proposition \ref{ppn:density}, it is sufficient to prove that 
\begin{proposition}\label{claim:free-sys}
For every finite $T$, the sequence of empirical measures of the auxiliary free system, $\{\mu_t^{N,f}(dx): t\in[0,T]\}$, converges in law as $N\to\infty$ to a limit $\{\ovl{\mu}^f_t(dx): t\in[0,T]\}$, in the space $C([0,T];\cM_{+,1}(\R^d))$. The latter random measure is absolutely continuous with respect to Lebesgue measure, with a density $\ovl {u}^f(t,x)$ uniformly bounded by the deterministic constant $\Gamma$. 
\end{proposition}

We know by Section \ref{tightness} that the sequence of laws $\{\mathscr L^N\}_N$ of $\big\{\{\mu ^{N,m}\}_{m\le M},\mu ^{N,f}\big\}$ taking values in $\cD_T(\cM_{+,1})^{M+1}$ form a tight sequence hence is weakly relatively compact. Fix any weak subsequential limit $\ovl{\mathscr L}$ of $\mathscr L^{N_j}$. By Skorohod's representation theorem, on an auxiliary probability space $(\wh\Omega,\wh\cF,\bP)$, there exist random variables $\big\{\{\wh\mu^{N_j,m} \}_{m\le M},  \wh\mu^{N_j, f}\big\}$, $j\ge 1$ and $\big\{\{\ovl \mu ^m\}_{m\le M}, \ovl {\mu} ^f\big\}$, having the laws $\mathscr L^{N_j}$, $j\ge1$, $\ovl{\mathscr L}$, respectively, such that $\bP$-a.s.
\[
\big\{\{\mu^{N_j,m} \}_{m\le M}, \mu^{N_j, f} \big\}\to \big\{\{\ovl \mu ^m\}_{m\le M}, \ovl {\mu} ^f\big\}, \quad j\to\infty. 
\]
In particular, $\ovl {\mu} ^f$ satisfies the properties in Proposition \ref{claim:free-sys}.

By \eqref{domin} and the representation, on $(\wh\Omega,\wh\cF,\bP)$ we have  $\bP$-a.s. for every $t, m$ and $\phi\in C_c^\infty(\R^d)$ with $\phi\ge 0$,
\[
\left\langle \wh\mu_t^{N_j,f}(dx)-\wh\mu_t^{N_j,m}(dx), \phi \right\rangle \ge 0.
\]
As $j\to\infty$, the above nonnegative sequence converges $\bP$-a.s. to 
\[
\left\langle \ovl\mu_t^{f}(dx)-\ovl \mu_t^{m}(dx), \phi \right\rangle \ge 0.
\]
This implies that $\bP$-a.s. for every open set $A\subset \R^d$, $t\in[0,T]$ and $m\le M$,
\[
\ovl \mu_t^m(A)\le \ovl {\mu}_t^f(A)=\int_A \ovl u^f(t,x)dx.
\]
Recall that any Borel probability measure on a Polish space is regular; namely the measure of a Borel set is the infimum of the measures of open sets larger than the given Borel set. Fix any Borel set $B\subset\R^d$ with null Lebesgue measure, and any open set $A\supset B$. By absolute continuity of $\ovl {\mu}_t^f$ with respect to Lebesgue, we have that $\ovl {\mu}_t^f(B)=0$, and further
\begin{align}\label{Borel}
0\le \ovl \mu_t^m(B)\le \ovl\mu_t^m(A)\le \ovl {\mu}_t^f(A), \quad \forall t\in[0,T].
\end{align}
Taking infimum over all open sets $A\supset B$ in \eqref{Borel}, we have that 
\[
0\le \ovl \mu_t^m(B)\le \inf_{\text{open } A\supset B} \ovl {\mu}_t^f(A)=\ovl {\mu}_t^f(B)=0, \quad \forall t\in[0,T].
\]
That is, $\{\ovl \mu_t^m(dx)\}_{m\le M}$ is also absolutely continuous with respect to Lebesgue measure, $\bP$-a.s. Let us denote its density by $\{u_m(t,x)\}_{m\le M}$, then we have that 
\[
\int_B \left[\ovl u^f(t,x)-u_m(t,x)\right] dx\ge 0
\]
for every Borel set $B\subset\R^d$, which implies $\ovl u^f(t,x)-u_m(t,x)\ge 0$, Lebesgue-a.e., whereby $\|u_m(t,\cdot)\|_\infty\le \Gamma$, $\bP$-a.s.

Thus we are left to show Proposition \ref{claim:free-sys}. As a preliminary, note that we can repeat our derivation in Section \ref{sec:emp-mea} for the empirical measure $\mu^{N,f}$ of the free system, and due to its linearity (there are no coagulation terms), it is easy to show that under the Skorohod representation, any of its subsequential limit $\ovl \mu^f$ must be a measure-valued solution of the following SPDE (formally written)
\begin{align*}
\begin{cases}
\partial\mu_t(x)&=\frac{\lambda^2}{2}\Delta \mu_t(x)dt+\sum_{k\in K}\sigma_k(x)\cdot \nabla \mu_t(x)\,\circ\, dW_t^k,\quad (t,x)\in[0,T]\times\R^d, \\\\
\mu_0(x)&=\sum_{m=1}^Mr_mp_m(x)dx.
\end{cases}
\end{align*}
(see in particular  \eqref{key-on-new}, \eqref{linear-terms}, \eqref{env-mg-conv}). The Proposition \ref{claim:free-sys} is proved in the next Section \ref{sec:free}.

\section{The SPDE of the free system}\label{sec:free}

\subsection{Existence, uniqueness and regularity results}

Arguing in the "free" case as in the more difficult case with interaction, we
can prove that the family of laws $Q_{N}$ of the empirical measures are tight
and that every limit point may be seen, on a suitable probability space, as a
random time-dependent probability measure $\mu_{t}$ (in the case with
interaction there were only finite measures, but here the empirical measures
have mass equal to one), adapted to the noise $W$, satisfying%
\begin{equation}\label{equation for measures}
\begin{aligned}
\int\phi\left(  t,x\right)  \mu_{t}\left(  dx\right)   &  =\int\phi\left(
0,x\right)  u_{0}\left(  x\right)  dx+\sum_{k\in K}\int_{0}^{t}\left(
\mathbb{\int}\nabla\phi\left(  s,x\right)  \cdot\sigma_{k}\left(  x\right)
\mu_{s}\left(  dx\right)  \right)  dW_{s}^{k}\\
&  +\int_{0}^{t}\frac{1}{2}\sum_{ij}\left(  \mathbb{\int}\partial_{i}%
\partial_{j}\phi\left(  s,x\right)  \left(  \lambda^2\delta_{ij}+Q_{ij}\left(
x,x\right)  \right)  \mu_{s}\left(  dx\right)  \right)  ds \\
&  +\int_{0}^{t}\int\partial_{s}\phi\left(  s,x\right)  \mu_{s}\left(
dx\right)  ds 
\end{aligned}
\end{equation}
for every $\phi\in C_{c}^{1,2}\left(  \left[  0,T\right]  \times\mathbb{R}%
^{d}\right)  $.

The first result we want to prove is the uniqueness of measure-valued
solutions to this equation. Although potentially several techniques may be
used, we present here one based on SPDE theory in negative order Sobolev
spaces, which may be of interest on its own. The approach presented here is
inspired by \cite{Fla 95 book}. For the sake of simplicity we develop the theory
in the Hilbert scale $W^{\alpha,2}\left(  \mathbb{R}^{d}\right)  $ but,
following \cite{Krylov} and \cite{Agresti}, one can work in a Banach scale
$W^{\alpha,p}\left(  \mathbb{R}^{d}\right)  $ with $p>2$ which has the
advantage to reduce the necessary degree of differentiability to have the
Sobolev embedding $W^{\alpha,p}\left(  \mathbb{R}^{d}\right)  \subset
C_{b}\left(  \mathbb{R}^{d}\right)  $ (see below) and thus allows one to ask
less differentiability of the coefficients $\sigma_{k}$; since an optimal result
is still not clear, we do not stress this level of generality in this work.

In dimension $d$ one has the Sobolev embedding $W^{\frac{d}{2}+\delta
,2}\left(  \mathbb{R}^{d}\right)  \subset C_{b}\left(  \mathbb{R}^{d}\right)
$ for every $\delta>0$;\ for simplicity of exposition we take $\delta=\frac
{1}{2}$. Therefore the set of probability measures $\Pr\left(  \mathbb{R}%
^{d}\right)  $ is contained in the dual of $W^{\frac{d+1}{2},2}\left(
\mathbb{R}^{d}\right)  $, the negative order Sobolev space $W^{-\frac{d+1}%
{2},2}\left(  \mathbb{R}^{d}\right)  $:%
\[
\Pr\left(  \mathbb{R}^{d}\right)  \subset W^{-\frac{d+1}{2},2}\left(
\mathbb{R}^{d}\right)  \text{.}%
\]
The set-theoretical inclusion is a continuous embedding of metric spaces when
$\Pr\left(  \mathbb{R}^{d}\right)  $ is endowed of the distance%
\[
d\left(  \mu,\nu\right)  =\sup_{\left\Vert \phi\right\Vert _{\infty}\leq
1}\left\vert \left\langle \mu,\phi\right\rangle -\left\langle \nu
,\phi\right\rangle \right\vert
\]
where $\left\langle \mu,\phi\right\rangle $ denotes, as usual, $\int%
_{\mathbb{R}^{d}}\phi\left(  x\right)  \mu\left(  dx\right)  $. Moreover, up
to a constant $C>0$, one has%
\[
\left\Vert \mu\right\Vert _{W^{-\frac{d+1}{2},2}}\leq C
\]
for all $\mu\in\Pr\left(  \mathbb{R}^{d}\right)  $.

A measure-valued solution in the sense above is also a weak solution of class
$W^{-\frac{d+1}{2},2}$. Choosing $\phi\left(  t,x\right)  $ related to
$e^{t\mathcal{L}^{\ast}}\psi$, with suitable choice of times, one can rewrite
the previous identity as
\[
\left\langle \psi,\mu_{t}\right\rangle =\left\langle e^{t\mathcal{L}^{\ast}%
}\psi,\mu_{0}\right\rangle -\sum_{k\in K}\int_{0}^{t}\left\langle \sigma
_{k}\cdot\nabla e^{\left(  t-s\right)  \mathcal{L}^{\ast}}\psi,\mu
_{s}\right\rangle dW_{s}^{k}%
\]
for every $\psi\in C_{c}^{\infty}\left(  \mathbb{R}^{d}\right)  $, where we
write $e^{t\mathcal{L}^{\ast}}$ so that the duality structure is more clear,
but it is equal to $e^{t\mathcal{L}}$, $\cL$ as in \eqref{elliptic}. We only remark that, thanks to the
$C_{b}^{\infty}$ regularity of $\sigma_{k}$, the semigroup $e^{t\mathcal{L}}$
maps $W^{k,2}\left(  \mathbb{R}^{d}\right)  $ in itself for every $k$, so that
$\sigma_{k}\cdot\nabla e^{\left(  t-s\right)  \mathcal{L}^{\ast}}\psi$ is of
class $W^{\frac{d+1}{2},2}$ and thus the duality is well defined under the
stochastic integral (similarly for the initial condition term).

We want to move a step further, namely interpret the previous identity as an
equation of the form%
\begin{equation}
\mu_{t}=e^{t\mathcal{L}}\mu_{0}+\sum_{k\in K}\int_{0}^{t}e^{\left(
t-s\right)  \mathcal{L}}\sigma_{k}\cdot\nabla\mu_{s}dW_{s}^{k}\label{H -32}%
\end{equation}
in a suitable Hilbert space. For technical reasons which will be clear below,
given $d$, we choose the minimal positive integer $N$ such that $2N+1\geq
\frac{d+1}{2}$. Then we consider progressively measurable processes $\mu_{t}$
in $W^{-2N,2}\left(  \mathbb{R}^{d}\right)  $ with the following regularity:%
\begin{equation}
\sup_{t\in\left[  0,T\right]  }\mathbb{E}\left[  \left\Vert \mu_{t}\right\Vert
_{W^{-2N,2}}^{2}\right]  +\mathbb{E}\int_{0}^{T}\left\Vert \mu_{t}\right\Vert
_{W^{-2N+1,2}}^{2}dt<\infty.\label{regularity negative}%
\end{equation}
Probability measure solutions of the equation above belong to this space.
Under these condition, $\mathbb{E}\int_{0}^{T}\left\Vert \nabla\mu
_{t}\right\Vert _{W^{-2N,2}}^{2}dt<\infty$ ($\nabla$ interpreted in the sense
of distributions), the multiplication with $\sigma_{k}\in C_{b}^{\infty
}\left(  \mathbb{R}^{d}\right)  $ remains of the same class; and the operators
$e^{\left(  t-s\right)  \mathcal{L}}$ are equibounded (on finite time
intervals) in $W^{-2N,2}\left(  \mathbb{R}^{d}\right)  $ by duality, hence
$\int_{0}^{t}e^{\left(  t-s\right)  \mathcal{L}}\sigma_{k}\cdot\nabla\mu
_{s}dW_{s}^{k}$ is a stochastic process in $W^{-2N,2}\left(  \mathbb{R}%
^{d}\right)  $. We interpret (\ref{H -32}) thus as an identity in
$W^{-2N,2}\left(  \mathbb{R}^{d}\right)  $, for solution of class
(\ref{regularity negative}).

Call $a_{ij}\left(  x\right)  :=\frac{1}{2}\left(  \lambda^2 +Q_{ij}\left(  x,x\right)
\right)  $. It is easy to check (\cite{Flandoli LNM}) that there exists $\eta_{0}\in\left(
0,1\right)  $ such that%
\[
\frac{1}{2}\sum_{k\in K}\left(  \sigma_{k}\left(  x\right)  \cdot\xi\right)
^{2}\leq\eta_{0}\sum_{i,j}a_{ij}\left(  x\right)  \xi_{i}\xi_{j}%
\]
for all $\xi\in\mathbb{R}^{d}$. This implies%
\[
\sum_{k\in K}\left\Vert \sigma_{k}\cdot\nabla g\right\Vert _{L^{2}}^{2}%
\leq-2\eta_{0}\left\langle \mathcal{L}g,g\right\rangle
\]
and thus%
\begin{equation}
\sum_{k\in K}\int_{s}^{T}\left\Vert \sigma_{k}\cdot\nabla e^{\left(
t-s\right)  \mathcal{L}}f\right\Vert _{L^{2}}^{2}dt\leq\eta_{0}\left\Vert
f\right\Vert _{L^{2}}^{2}\label{basic in L2}%
\end{equation}
for all $f\in L^{2}\left(  \mathbb{R}^{d}\right)  $ because
\begin{align*}
\sum_{k\in K}\int_{s}^{T}\left\Vert \sigma_{k}\cdot\nabla e^{\left(
t-s\right)  \mathcal{L}}f\right\Vert _{L^{2}}^{2}dt  & \leq-2\eta_{0}\int%
_{s}^{T}\left\langle \mathcal{L}e^{\left(  t-s\right)  \mathcal{L}%
}f,e^{\left(  t-s\right)  \mathcal{L}}f\right\rangle dt\\
& =-\eta_{0}\int_{s}^{T}\frac{d}{dt}\left\Vert e^{\left(  t-s\right)
\mathcal{L}}f\right\Vert _{L^{2}}^{2}dt\leq\eta_{0}\left\Vert f\right\Vert
_{L^{2}}^{2}.
\end{align*}
For every $n\in\mathbb{Z}$, let us endow $W^{2n,2}\left(  \mathbb{R}%
^{d}\right)  $ by the norm
\[
\left\Vert f\right\Vert _{W^{2n,2}}:=\left\Vert \left(  1-\mathcal{L}\right)
^{n}f\right\Vert _{L^{2}}.
\]

\begin{lemma}
\label{lemma super ellipticity}Let $\eta\in\left(  \eta_{0},1\right)  $ be
given. For every integer number $n\in\mathbb{Z}$ there exists a constant
$C_{n,K}>0$ such that
\[
\sum_{k\in K}\int_{s}^{T}\left\Vert \sigma_{k}\cdot\nabla e^{\left(
t-s\right)  \mathcal{L}}f\right\Vert _{W^{2n,2}}^{2}dt\leq\left(  \eta
+C_{n,K}\left(  T-s\right)  \right)  \left\Vert f\right\Vert _{W^{2n,2}}^{2}%
\]
for all $f\in W^{2n,2}\left(  \mathbb{R}^{d}\right)  $.
\end{lemma}

\begin{proof}
The case $n=0$ is already proved above. We give the proof for $n=1$ and $n=-1
$, the general case being similar.

For every $k\in K$ there is a second order differential operator
$D_{k}^{\left(  2\right)  }$, with coefficients of class $C_{b}^{\infty
}\left(  \mathbb{R}^{d}\right)  $, such that
\begin{equation}
\sigma_{k}\cdot\nabla\left(  1-\mathcal{L}\right)  f=\left(  1-\mathcal{L}%
\right)  \sigma_{k}\cdot\nabla f+D_{k}^{\left(  2\right)  }%
f\label{identity operators}%
\end{equation}
for all $f\in C_{c}^{\infty}\left(  \mathbb{R}^{d}\right)  $. Indeed, the
terms with third order derivatives in $\sigma_{k}\cdot\nabla\left(
1-\mathcal{L}\right)  f$ and $\left(  1-\mathcal{L}\right)  \sigma_{k}%
\cdot\nabla f$ coincide. Therefore (case $n=1$)
\begin{align*}
\sum_{k\in K}\int_{s}^{T}\left\Vert \sigma_{k}\cdot\nabla e^{\left(
t-s\right)  \mathcal{L}}f\right\Vert _{W^{2,2}}^{2}dt  & =\sum_{k\in K}%
\int_{s}^{T}\left\Vert \left(  1-\mathcal{L}\right)  \sigma_{k}\cdot\nabla
e^{\left(  t-s\right)  \mathcal{L}}f\right\Vert _{L^{2}}^{2}dt\\
& =\sum_{k\in K}\int_{s}^{T}\left\Vert \left[  \sigma_{k}\cdot\nabla\left(
1-\mathcal{L}\right)  -D_{k}^{\left(  2\right)  }\right]  e^{\left(
t-s\right)  \mathcal{L}}f\right\Vert _{L^{2}}^{2}dt.
\end{align*}
For $\epsilon>0$ we use the inequality $\left(  a+b\right)  ^{2}\leq\left(
1+\epsilon\right)  a^{2}+\left(  1+\frac{1}{\epsilon}\right)  b^{2}$ to get%
\[
\leq\left(  1+\epsilon\right)  \sum_{k\in K}\int_{s}^{T}\left\Vert \sigma
_{k}\cdot\nabla e^{\left(  t-s\right)  \mathcal{L}}\left(  1-\mathcal{L}%
\right)  f\right\Vert _{L^{2}}^{2}dt+\left(  1+\frac{1}{\epsilon}\right)
\sum_{k\in K}\int_{s}^{T}\left\Vert D_{k}^{\left(  2\right)  }e^{\left(
t-s\right)  \mathcal{L}}f\right\Vert _{L^{2}}^{2}dt.
\]
The first term is handled by (\ref{basic in L2}), the second one by a trivial
bound, to get%
\[
\leq\left(  1+\epsilon\right)  \eta_{0}\left\Vert \left(  1-\mathcal{L}%
\right)  f\right\Vert _{L^{2}}^{2}+C_{\epsilon,K}\left(  T-s\right)
\left\Vert f\right\Vert _{W^{2,2}}^{2}.
\]
If $\epsilon$ satisfies $\left(  1+\epsilon\right)  \eta_{0}=\eta$, this is
the required bound.

For $n=-1$, we have to prove
\[
\sum_{k\in K}\int_{s}^{T}\left\Vert \left(  1-\mathcal{L}\right)  ^{-1}%
\sigma_{k}\cdot\nabla e^{\left(  t-s\right)  \mathcal{L}}f\right\Vert _{L^{2}%
}^{2}dt\leq\left(  \eta+C_{-1,K}\left(  T-s\right)  \right)  \left\Vert
\left(  1-\mathcal{L}\right)  ^{-1}f\right\Vert _{L^{2}}^{2}.
\]
We set $g=\left(  1-\mathcal{L}\right)  ^{-1}f\in L^{2}\left(  \mathbb{R}%
^{d}\right)  $, so that we have to prove%
\[
\sum_{k\in K}\int_{s}^{T}\left\Vert \left(  1-\mathcal{L}\right)  ^{-1}%
\sigma_{k}\cdot\nabla\left(  1-\mathcal{L}\right)  e^{\left(  t-s\right)
\mathcal{L}}g\right\Vert _{L^{2}}^{2}dt\leq\left(  \eta+C_{-1,K}\left(
T-s\right)  \right)  \left\Vert g\right\Vert _{L^{2}}^{2}.
\]
By (\ref{identity operators}), the left-hand-side is bounded by
\[
\leq\left(  1+\epsilon\right)  \sum_{k\in K}\int_{s}^{T}\left\Vert \sigma
_{k}\cdot\nabla e^{\left(  t-s\right)  \mathcal{L}}g\right\Vert _{L^{2}}%
^{2}dt+\left(  1+\frac{1}{\epsilon}\right)  \sum_{k\in K}\int_{s}%
^{T}\left\Vert \left(  1-\mathcal{L}\right)  ^{-1}D_{k}^{\left(  2\right)
}e^{\left(  t-s\right)  \mathcal{L}}g\right\Vert _{L^{2}}^{2}dt
\]
and the claimed result is proved as above.
\end{proof}

\begin{proposition}
\label{proposition negative order}Equation (\ref{H -32}) has a unique solution
in the class (\ref{regularity negative}). In particular, the equation for
probability measures (\ref{equation for measures}) has a unique solution.
\end{proposition}

\begin{proof}
Let $\mu_{t}^{\left(  i\right)  }$, $i=1,2$ be two solutions and $\mu_{t}%
=\mu_{t}^{\left(  1\right)  }-\mu_{t}^{\left(  2\right)  }$. Then%
\[
\mu_{t}=\sum_{k\in K}\int_{0}^{t}e^{\left(  t-s\right)  \mathcal{L}}\sigma
_{k}\cdot\nabla\mu_{s}dW_{s}^{k}.
\]
~We introduce the auxiliary processes in $W^{-2N,2}\left(  \mathbb{R}%
^{d}\right)  $
\[
v_{k}\left(  t\right)  =\sigma_{k}\cdot\nabla\mu_{t}%
\]
and their equations%
\[
v_{k}\left(  t\right)  =\sum_{h\in K}\int_{0}^{t}\sigma_{k}\cdot\nabla
e^{\left(  t-s\right)  \mathcal{L}}v_{h}\left(  s\right)  dW_{s}^{h}.
\]
We have%
\[
\sum_{k\in K}\mathbb{E}\int_{0}^{T}\left\Vert v_{k}\left(  t\right)
\right\Vert _{W^{-2N,2}}^{2}dt=\sum_{h\in K}\mathbb{E}\int_{0}^{T}\sum_{k\in
K}\int_{s}^{T}\left\Vert \sigma_{k}\cdot\nabla e^{\left(  t-s\right)
\mathcal{L}}v_{h}\left(  s\right)  \right\Vert _{W^{-2N,2}}^{2}dtds
\]
and thus, using Lemma \ref{lemma super ellipticity},
\[
\leq\left(  \eta+C_{K}T\right)  \sum_{h\in K}\int_{0}^{T}\mathbb{E}\left[
\left\Vert v_{h}\left(  s\right)  \right\Vert _{W^{-2N,2}}^{2}\right]  ds.
\]
This implies $v_{k}=0$ for all $k$ if $T$ is small enough, hence $\mu_{t}=0$
since%
\[
\mu_{t}=\sum_{k\in K}\int_{0}^{t}e^{\left(  t-s\right)  \mathcal{L}}%
v_{k}\left(  s\right)  dW_{s}^{k}.
\]
The argument can be repeated on intervals of constant length, proving uniqueness.
\end{proof}

We have proved pathwise uniqueness for equation (\ref{equation for measures}).
This implies convergence in probability of the empirical measures, by an
argument of \cite{Gyon Krylov} that we omit.

Now we prove the regularity of $\mu_{t}$. Consider the equation (for
functions, now)%
\begin{align*}
du\left(  t,x\right)    & =\frac{1}{2}\left(  \lambda^2\Delta+\operatorname{div}\left(
Q\left(x,x\right)  \nabla\right)  \right)  u\left(  t,x\right)  dt+\sum_{k\in
K}\sigma_{k}\left(  x\right)  \cdot\nabla u\left(  t,x\right)  dW_{t}^{k}\\
u|_{t=0}  & =u_{0}%
\end{align*}
interpreted in the mild form%
\begin{equation}
u\left(  t\right)  =e^{t\mathcal{L}}u_{0}+\sum_{k\in K}\int_{0}^{t}e^{\left(
t-s\right)  \mathcal{L}}\sigma_{k}\cdot\nabla u\left(  s\right)  dW_{s}%
^{k}.\label{mild regular}%
\end{equation}

\begin{definition}
Given $u_{0}\in W^{2n,2}\left(  \mathbb{R}^{d}\right)  $, we say that $u$ is a
mild solution in $W^{2n,2}\left(  \mathbb{R}^{d}\right)  $ of equation
(\ref{mild regular}) if it is progressively measurable in $W^{2n,2}\left(
\mathbb{R}^{d}\right)  $, satisfies%
\[
\sup_{t\in\left[  0,T\right]  }\mathbb{E}\left[  \left\Vert u\left(
t,\cdot\right)  \right\Vert _{W^{2n,2}}^{2}\right]  +\mathbb{E}\int_{0}%
^{T}\left\Vert u\left(  t,\cdot\right)  \right\Vert _{W^{2n+1,2}}^{2}dt<\infty
\]
and identity (\ref{mild regular}) holds true.
\end{definition}

\begin{proposition}
Given $u_{0}\in W^{2n,2}\left(  \mathbb{R}^{d}\right)  $, there exists a
unique mild solution in $W^{2n,2}\left(  \mathbb{R}^{d}\right)  $ of equation
(\ref{mild regular}).
\end{proposition}

\begin{proof}
As above for the uniqueness proof, we consider the auxiliary equations%
\[
v_{k}\left(  t\right)  =\sigma_{k}\cdot\nabla e^{t\mathcal{L}}u_{0}+\sum_{h\in
K}\int_{0}^{t}\sigma_{k}\cdot\nabla e^{\left(  t-s\right)  \mathcal{L}}%
v_{h}\left(  s\right)  dW_{s}^{h}%
\]
and progressively measurable solutions $v=\left(  v_{k}\right)  _{k\in K}$
such that
\[
\left\Vert v\right\Vert _{2}^{2}:=\sum_{k\in K}\mathbb{E}\int_{0}%
^{T}\left\Vert v_{k}\left(  t\right)  \right\Vert _{W^{2n,2}}^{2}dt
\]
is finite. Given $u_{0}\in W^{2n,2}\left(  \mathbb{R}^{d}\right)  $, the terms
$\sigma_{k}\cdot\nabla e^{t\mathcal{L}}u_{0}$ form a vector with this
property, by Lemma \ref{lemma super ellipticity}. Then we apply the
contraction mapping principle to the equation for $v$ with the norm
$\left\Vert v\right\Vert _{2}$ above. Being linear, the key estimate is%
\begin{align*}
& \sum_{k\in K}\mathbb{E}\int_{0}^{T}\left\Vert \sum_{h\in K}\int_{0}%
^{t}\sigma_{k}\cdot\nabla e^{\left(  t-s\right)  \mathcal{L}}v_{h}\left(
s\right)  dW_{s}^{h}\right\Vert _{W^{2n,2}}^{2}ds\\
& =\sum_{h\in K}\mathbb{E}\int_{0}^{T}\sum_{k\in K}\int_{s}^{T}\left\Vert
\sigma_{k}\cdot\nabla e^{\left(  t-s\right)  \mathcal{L}}v_{h}\left(
s\right)  \right\Vert _{W^{2n,2}}^{2}dtds\\
& \leq\left(  \eta+C_{K}T\right)  \sum_{h\in K}\int_{0}^{T}\mathbb{E}\left[
\left\Vert v_{h}\left(  s\right)  \right\Vert _{W^{2n,2}}^{2}\right]  ds
\end{align*}
which, for $T$ small enough, allows one to apply the contraction principle.
The argument can be repeated on intervals of constant length, proving
existence and uniqueness. 
\end{proof}

Collecting the previous results we have:

\begin{theorem}\label{thm:reg-free}
Given $u_{0}\in W^{2n,2}\left(  \mathbb{R}^{d}\right)  $ for some $n\geq1$,
$u_{0}$ non negative with $\int u_{0}\left(  x\right)  dx=1$, define $\mu
_{0}\left(  dx\right)  =u_{0}\left(  x\right)  dx$. Then equation
(\ref{equation for measures}) has a unique solution $\mu_{t}\left(  dx\right)
$ according to Proposition \ref{proposition negative order}. It has a density
$u\left(  t,x\right)  $,
\[
\mu_{t}\left(  dx\right)  =u\left(  t,x\right)  dx
\]
which is a mild solution in $W^{2n,2}\left(  \mathbb{R}^{d}\right)  $ of
equation (\ref{mild regular}). In particular, if $n>d/4$, then $\left\Vert
u\left(  t,\cdot\right)  \right\Vert _{\infty}$ is finite a.s.
\end{theorem}

\begin{proof}
The function $u_{0}\in W^{2n,2}\left(  \mathbb{R}^{d}\right)  $ is also of
class $u_{0}\in W^{-2N,2}\left(  \mathbb{R}^{d}\right)  $, and the
corresponding mild solution $u$ of class $W^{2n,2}\left(  \mathbb{R}%
^{d}\right)  $ is also a distributional solution in the sense of Proposition
\ref{proposition negative order}, hence it is the unique solution of that
negative order Sobolev class. A probability-measure solution $\mu_{t}\left(
dx\right)  $ exists as a subsequential limit of the empirical measures, with
initial condition $\mu_{0}\left(  dx\right)  =u_{0}\left(  x\right)  dx$,
hence by the uniqueness statement of Proposition
\ref{proposition negative order}, it coincides with $u$, in the sense of
distributions, which implies $\mu_{t}\left(  dx\right)  =u\left(  t,x\right)
dx$.
\end{proof}

The previous argument proves only that $\left\Vert u\left(  t,\cdot\right)
\right\Vert _{\infty}$ is finite. In the next section we prove that it is
uniformly bounded by a deterministic constant.

\subsection{Uniform upper bound}

\begin{lemma}\label{lemma uniform upper bound}
Let $u_{0}\in W^{2n,2}\left(  \mathbb{R}^{d}\right)  $ with $n>d/4$ and let
$C_{0}>0$ be a constant such that
\[
\left\Vert u_{0}\right\Vert _{\infty}\leq C_{0}.
\]
Then
\[
\left\Vert u\left(  t,\cdot\right)  \right\Vert _{\infty}\leq C_{0}%
\]
a.s. in all parameters.
\end{lemma}

\begin{proof}
Let us prove only that
\[
u\left(  t,x\right)  \leq C_{0}.
\]
The proof that $u\left(  t,x\right)  \geq-C_{0}$ is similar. The first remark
is that the process identically equal to $C_{0}$ is a solution, in the sense
that it satisfies the weak formulation, with initial condition $C_{0}$. By
linearity, the process
\[
v\left(  t,x\right)  :=u\left(  t,x\right)  -C_{0}%
\]
also satisfies the weak formulation, with initial condition
\[
v_{0}\left(  x\right)  :=u_{0}\left(  x\right)  -C_{0}\leq0.
\]
Our thesis is that also%
\[
v\left(  t,x\right)  \leq0.
\]

Given $\delta>0$, let $\beta_{\delta}:\mathbb{R}\rightarrow\lbrack0,\infty)$
be a $C^{2}$-convex function such that $\beta_{\delta}\left(  r\right)  =0$
for $r<0$, $\beta_{\delta}\left(  r\right)  =r-\delta$ for $r>2\delta$. The
family $\beta_{\delta}$ converges uniformly to the function $\beta\left(
r\right)  =r1_{[0,\infty)}\left(  r\right)  $. From It\^{o} formula%
\[
d\beta_{\delta}\left(  v\left(  t,x\right)  \right)  =\beta_{\delta}^{\prime
}\left(  v\left(  t,x\right)  \right)  dv\left(  t,x\right)  +\frac{1}{2}%
\beta_{\delta}^{\prime\prime}\left(  v\left(  t,x\right)  \right)  d\left[
v\left(  \cdot,x\right)  \right]  _{t}.
\]
From a number of intermediate computations that we omit (better understood
formally at the level of Stratonovich calculus), we get%
\begin{align*}
d\beta_{\delta}\left(  v\left(  t,x\right)  \right)   & =\frac{\lambda^2}{2}%
\beta_{\delta}^{\prime}\left(  v\left(  t,x\right)  \right)  \Delta v\left(
t,x\right)  dt+\sum_{k\in K}\sigma_{k}\cdot\nabla\beta_{\delta}\left(
v\left(  t,x\right)  \right)  dW_{t}^{k}\\
& +\frac{1}{2}\operatorname{div}\left(  Q\left(  x,x\right)  \nabla\beta
_{\delta}\left(  v\left(  t,x\right)  \right)  \right)  dt
\end{align*}
interpreted of course in integral%
\begin{align*}
\beta_{\delta}\left(  v\left(  t,x\right)  \right)   & =\int_{0}^{t}\frac
{\lambda^2}{2}\beta_{\delta}^{\prime}\left(  v\left(  s,x\right)  \right)  \Delta
v\left(  s,x\right)  ds\\
& +\sum_{k\in K}\int_{0}^{t}\sigma_{k}\left(  x\right)  \cdot\nabla
\beta_{\delta}\left(  v\left(  s,x\right)  \right)  dW_{s}^{k}\\
& +\frac{1}{2}\int_{0}^{t}\operatorname{div}\left(  Q\left( x,x\right)
\nabla\beta_{\delta}\left(  v\left(  s,x\right)  \right)  \right)  ds
\end{align*}
(but pointwise in $x$, thanks to the regularity of $v$). We have neglected the
term $\beta_{\delta}\left(  v_{0}\left(  x\right)  \right)  $ because it is
zero, by definition of the objects. Now we want to integrate in $x$ this
identity. The first and second derivatives of $v$, being equal to those of $u
$, are integrable; hence all terms on the right-hand-side are integrable. The
function $v$ itself could not be integrable, becuase the constant $C_{0}$ is
not, hence $\beta_{\delta}\left(  v\left(  t,x\right)  \right)  $ a priori is
not integrable. However, it is integrable as a consequence of the identity. We
have also used stochastic Fubini theorem to deal with the stochastic term.

Taking into account that
\[
\int_{\mathbb{R}^{d}}\sigma_{k}\left(  x\right)  \cdot\nabla\beta_{\delta
}\left(  v\left(  s,x\right)  \right)  dx=0
\]
because $\operatorname{div}\sigma_{k}=0$ and
\[
\int_{\mathbb{R}^{d}}\operatorname{div}\left(  Q\left( x,x\right)  \nabla
\beta_{\delta}\left(  v\left(  s,x\right)  \right)  \right)  dx=0
\]
(in both cases we use Gauss-Green formula), we get
\[
\int_{\mathbb{R}^{d}}\beta_{\delta}\left(  v\left(  t,x\right)  \right)
dx=\frac{\lambda^2}{2}\int_{0}^{t}\int_{\mathbb{R}^{d}}\beta_{\delta}^{\prime}\left(
v\left(  s,x\right)  \right)  \Delta v\left(  s,x\right)  dxds.
\]
But%
\[
\int_{\mathbb{R}^{d}}\beta_{\delta}^{\prime}\left(  v\left(  s,x\right)
\right)  \Delta v\left(  s,x\right)  dx=-\int_{\mathbb{R}^{d}}\beta_{\delta
}^{\prime\prime}\left(  v\left(  s,x\right)  \right)  \left\vert \nabla
v\left(  s,x\right)  \right\vert ^{2}dx\leq0
\]
because, by convexity, $\beta_{\delta}^{\prime\prime}\left(  r\right)  \geq0$.
Therefore%
\[
\int_{\mathbb{R}^{d}}\beta_{\delta}\left(  v\left(  t,x\right)  \right)
dx\leq0.
\]
Since the function $\beta_{\delta}$ is non-negative, we deduce
\[
\beta_{\delta}\left(  v\left(  t,x\right)  \right)  =0
\]
a.s. and thus, taking the limit as $\delta\rightarrow0$, $v\left(  t,x\right)
\leq0$ a.s., completing the proof.
\end{proof}

\subsection{Proof of Proposition \ref{Proposition regularity nonlinear}}

The proof consists in two main steps\ plus a few remarks. First we show that a
modified version of the system for $\left(  u_{m}\right)  _{m=1,...,M}$ has a
global solution in the regularity class specified by Proposition
\ref{Proposition regularity nonlinear}; this is Step 1 below. Then, in Step 2,
we connect this regular solution with the one provided by Corollary
\ref{Coroll uniqueness}: we show they are the same, hence the solution of
Corollary \ref{Coroll uniqueness} has the regularity stated by Proposition
\ref{Proposition regularity nonlinear}.

\textbf{Step 1}. Consider the auxiliary system%
\begin{align*}
u_{m}\left(  t\right)   &  =e^{t\mathcal{L}}u_{m}\left(  0\right)  +\int%
_{0}^{t}e^{\left(  t-s\right)  \mathcal{L}}\widetilde{F}_{m}\left(  u\left(
s\right)  \right)  ds+\sum_{k\in K}\int_{0}^{t}e^{\left(  t-s\right)
\mathcal{L}}v_{m,k}\left(  s\right)  dW_{s}^{k}\\
v_{m,k}\left(  t\right)   &  =\sigma_{k}\cdot\nabla e^{t\mathcal{L}}%
u_{m}\left(  0\right)  +\int_{0}^{t}\sigma_{k}\cdot\nabla e^{\left(
t-s\right)  \mathcal{L}}\widetilde{F}_{m}\left(  u\left(  s\right)  \right)
ds+\sum_{h\in K}\int_{0}^{t}\sigma_{k}\cdot\nabla e^{\left(  t-s\right)
\mathcal{L}}v_{m,h}\left(  s\right)  dW_{s}^{h}\\
u  &  =\left(  u_{1},...,u_{M}\right)
\end{align*}
with $m=1,...,M$, $k\in K$ and $\widetilde{F}_{m}$ defined as follows. Recall
that $0\leq u_{m}\left(  0\right)  \leq U$. Taken a smooth compact support
function $\chi:\mathbb{R}\rightarrow\mathbb{R}$ such that
\[
\chi\left(  a\right)  =a\text{ for }\left\vert a\right\vert \leq U+1,
\]
we define $\widetilde{F}_{m}$ as%
\[
\widetilde{F}_{m}\left(  u\left(  t,x\right)  \right)  =\sum_{n=1}^{m-1}%
\chi\left(  u_{n}\left(  t,x\right)  \right)  \chi\left(  u_{m-n}\left(
t,x\right)  \right)  -2\chi\left(  u_{m}\left(  t,x\right)  \right)
\sum_{n=1}^{M}\chi\left(  u_{n}\left(  t,x\right)  \right)  .
\]

Given $T>0$, consider the space $\mathcal{X}_{T}$ of progressively measurable
processes $u_{m},v_{m,k}$ in $W^{2n,2}\left(  \mathbb{R}^{d}\right)  $ such
that
\begin{align*}
\left\Vert \left(  u_{m},v_{m,k}\right)  _{m=1,...,M,k\in K}\right\Vert
_{\mathcal{X}_{T}}^{2}  &  :=\sup_{t\in\left[  0,T\right]  }\sum_{m=1}%
^{M}\mathbb{E}\left[  \left\Vert u_{m}\left(  t,\cdot\right)  \right\Vert
_{W^{2n,2}}^{2}\right] \\
&  +\sum_{m=1}^{M}\sum_{k\in K}\mathbb{E}\int_{0}^{T}\left\Vert v_{m,k}\left(
t,\cdot\right)  \right\Vert _{W^{2n+1,2}}^{2}dt
\end{align*}
is finite; the space $\mathcal{X}_{T}$ with the norm $\left\Vert
\cdot\right\Vert _{\mathcal{X}_{T}}$ is a Banach space. In it, let us define
the map $\Gamma_{T}$ as follows. We write an element $\left(  u_{m}%
,v_{m,k}\right)  _{m=1,...,M,k\in K}$ of $\mathcal{X}_{T}$ in the form
$\left(  u,v\right)  $, $u=\left(  u_{m}\right)  _{m=1,...,M}$, $v=\left(
v_{m,k}\right)  _{m=1,...,M,k\in K}$, and similarly we write $\Gamma
_{T}\left(  u,v\right)  $ in the form $\left(  \Gamma_{T}^{\left(  1\right)
}\left(  u,v\right)  ,\Gamma_{T}^{\left(  2\right)  }\left(  u,v\right)
\right)  $, with components $\Gamma_{T}^{\left(  1\right)  }\left(
u,v\right)  _{m}$, $\Gamma_{T}^{\left(  2\right)  }\left(  u,v\right)  _{m,k}%
$, $m=1,...,M,k\in K$, given by%
\[
\Gamma_{T}^{\left(  1\right)  }\left(  u,v\right)  _{m}\left(  t\right)
:=e^{t\mathcal{L}}u_{m}\left(  0\right)  +\int_{0}^{t}e^{\left(  t-s\right)
\mathcal{L}}\widetilde{F}_{m}\left(  u\left(  s\right)  \right)  ds+\sum_{k\in
K}\int_{0}^{t}e^{\left(  t-s\right)  \mathcal{L}}v_{m,k}\left(  s\right)
dW_{s}^{k}%
\]%
\begin{align*}
\Gamma_{T}^{\left(  2\right)  }\left(  u,v\right)  _{m,k}\left(  t\right)   &
:=\sigma_{k}\cdot\nabla e^{t\mathcal{L}}u_{m}\left(  0\right)  +\int_{0}%
^{t}\sigma_{k}\cdot\nabla e^{\left(  t-s\right)  \mathcal{L}}\widetilde{F}%
_{m}\left(  u\left(  s\right)  \right)  ds\\
&  +\sum_{h\in K}\int_{0}^{t}\sigma_{k}\cdot\nabla e^{\left(  t-s\right)
\mathcal{L}}v_{m,h}\left(  s\right)  dW_{s}^{h}.
\end{align*}
By Lemma \ref{lemma super ellipticity}, $\Gamma_{T}$ maps $\mathcal{X}_{T}$
into itself; the proof is similar to the computation done below to prove the
contraction property and thus it is not duplicated. We only remark a property
of $\widetilde{F}_{m}$ for the purpose of proving that $\Gamma_{T}$ maps
$\mathcal{X}_{T}$ into itself: since $\chi$ is smooth and compact support,
hence having all derivatives of every order bounded, $\chi\left(  u_{m}\left(
t,x\right)  \right)  \chi\left(  u_{n}\left(  t,x\right)  \right)  $ is of
class
\begin{equation}
C\left(  \left[  0,T\right]  ;L^{2}\left(  \Omega;W^{2n,2}\left(
\mathbb{R}^{d}\right)  \right)  \right)  \cap L^{2}\left(  \left[  0,T\right]
\times\Omega;W^{2n+1,2}\left(  \mathbb{R}^{d}\right)  \right) \label{class}%
\end{equation}
for every $u_{n},u_{m}$ of the same class.

Given two input functions $\left(  u,v\right)  $, $\left(  u^{\prime
},v^{\prime}\right)  $, with the same initial values $u_{m}\left(  0\right)
$, we have%
\begin{align*}
\Gamma_{T}^{\left(  1\right)  }\left(  u,v\right)  _{m}\left(  t\right)
-\Gamma_{T}^{\left(  1\right)  }\left(  u^{\prime},v^{\prime}\right)
_{m}\left(  t\right)   &  =\int_{0}^{t}e^{\left(  t-s\right)  \mathcal{L}%
}\left(  \widetilde{F}_{m}\left(  u\left(  s\right)  \right)  -\widetilde{F}%
_{m}\left(  u^{\prime}\left(  s\right)  \right)  \right)  ds\\
&  +\sum_{k\in K}\int_{0}^{t}e^{\left(  t-s\right)  \mathcal{L}}\left(
v_{m,k}\left(  s\right)  -v_{m,k}^{\prime}\left(  s\right)  \right)
dW_{s}^{k}%
\end{align*}%
\begin{align*}
\Gamma_{T}^{\left(  2\right)  }\left(  u,v\right)  _{m,k}\left(  t\right)
-\Gamma_{T}^{\left(  2\right)  }\left(  u^{\prime},v^{\prime}\right)
_{m,k}\left(  t\right)   &  =\int_{0}^{t}\sigma_{k}\cdot\nabla e^{\left(
t-s\right)  \mathcal{L}}\left(  \widetilde{F}_{m}\left(  u\left(  s\right)
\right)  -\widetilde{F}_{m}\left(  u^{\prime}\left(  s\right)  \right)
\right)  ds\\
&  +\sum_{h\in K}\int_{0}^{t}\sigma_{k}\cdot\nabla e^{\left(  t-s\right)
\mathcal{L}}\left(  v_{m,h}\left(  s\right)  -v_{m,h}^{\prime}\left(
s\right)  \right)  dW_{s}^{h}.
\end{align*}
Let us develop the estimates for the second line, the first one being similar
and a little bit easier. We have, from $\left(  a+b\right)  ^{2}\leq\left(
1+\frac{1}{\epsilon}\right)  a^{2}+\left(  1+\epsilon\right)  b^{2}$,
\begin{align*}
&  \sum_{m=1}^{M}\sum_{k\in K}\mathbb{E}\int_{0}^{T}\left\Vert \Gamma
_{T}^{\left(  2\right)  }\left(  u,v\right)  _{m,k}\left(  t\right)
-\Gamma_{T}^{\left(  2\right)  }\left(  u^{\prime},v^{\prime}\right)
_{m,k}\left(  t\right)  \right\Vert _{W^{2n+1,2}}^{2}dt\\
&  \leq\left(  1+\frac{1}{\epsilon}\right)  T\sum_{m=1}^{M}\sum_{k\in
K}\mathbb{E}\int_{0}^{T}\int_{s}^{T}\left\Vert \sigma_{k}\cdot\nabla
e^{\left(  t-s\right)  \mathcal{L}}\left(  \widetilde{F}_{m}\left(  u\left(
s\right)  \right)  -\widetilde{F}_{m}\left(  u^{\prime}\left(  s\right)
\right)  \right)  \right\Vert _{W^{2n+1,2}}^{2}dtds\\
&  +\left(  1+\epsilon\right)  \sum_{m=1}^{M}\sum_{k\in K}\mathbb{E}\int%
_{0}^{T}\sum_{h\in K}\int_{s}^{T}\left\Vert \sigma_{k}\cdot\nabla e^{\left(
t-s\right)  \mathcal{L}}\left(  v_{m,h}\left(  s\right)  -v_{m,h}^{\prime
}\left(  s\right)  \right)  \right\Vert _{W^{2n+1,2}}^{2}dtds
\end{align*}
where we have used also Fubini-Tonelli theorem. We apply Lemma
\ref{lemma super ellipticity} to both terms and get%
\begin{align*}
&  \leq\left(  1+\frac{1}{\epsilon}\right)  T\left(  \eta+C_{K}T\right)
\sum_{m=1}^{M}\mathbb{E}\int_{0}^{T}\left\Vert \widetilde{F}_{m}\left(
u\left(  s\right)  \right)  -\widetilde{F}_{m}\left(  u^{\prime}\left(
s\right)  \right)  \right\Vert _{W^{2n,2}}^{2}ds\\
&  +\left(  1+\epsilon\right)  \left(  \eta+C_{K}T\right)  \sum_{m=1}%
^{M}\mathbb{E}\int_{0}^{T}\sum_{h\in K}\left\Vert v_{m,h}\left(  s\right)
-v_{m,h}^{\prime}\left(  s\right)  \right\Vert _{W^{2n,2}}^{2}ds.
\end{align*}
Using as above the fact that all derivatives of every order of $\chi$ are
bounded, we get%
\begin{align*}
&  \leq C_{\chi}\left(  1+\frac{1}{\epsilon}\right)  T\left(  \eta
+C_{K}T\right)  \sum_{m=1}^{M}\mathbb{E}\int_{0}^{T}\left\Vert u_{m}\left(
s\right)  -u_{m}^{\prime}\left(  s\right)  \right\Vert _{W^{2n,2}}^{2}ds\\
&  +\left(  1+\epsilon\right)  \left(  \eta+C_{K}T\right)  \sum_{m=1}%
^{M}\mathbb{E}\int_{0}^{T}\sum_{h\in K}\left\Vert v_{m,h}\left(  s\right)
-v_{m,h}^{\prime}\left(  s\right)  \right\Vert _{W^{2n,2}}^{2}ds.
\end{align*}
Since $\eta<1$, if $T$ is small enough we deduce that $\Gamma_{T}$ is a
contraction in $\mathcal{X}_{T}$. Therefore it has a unique fixed point. Since
the size of $T$ to get this result is not related to the size of the initial
condition, the procedure can be repeated on intervals of constant length. We
deduce that there is a unique solution in $\mathcal{X}_{T}$ with $T$
arbitrarily large and a priori chosen.

Let $\left(  u_{m},v_{m,k}\right)  _{m=1,...,M,k\in K}$ be the unique solution
of the system above. From the first $M$ identities of the system (those for
$u_{m}$, $m=1,...,M$), we see that $\sigma_{k}\cdot\nabla u_{m}$ is well
defined in
\[
L^{2}\left(  \left[  0,T\right]  \times\Omega;W^{2n,2}\left(  \mathbb{R}%
^{d}\right)  \right)
\]
and we have%
\begin{align*}
\sigma_{k}\cdot\nabla u_{m}\left(  t\right)   &  =\sigma_{k}\cdot\nabla
e^{t\mathcal{L}}u_{m}\left(  0\right)  +\int_{0}^{t}\sigma_{k}\cdot\nabla
e^{\left(  t-s\right)  \mathcal{L}}\widetilde{F}_{m}\left(  u\left(  s\right)
\right)  ds\\
&  +\sum_{h\in K}\int_{0}^{t}\sigma_{k}\cdot\nabla e^{\left(  t-s\right)
\mathcal{L}}v_{m,h}\left(  s\right)  dW_{s}^{h}.
\end{align*}
But this is equal to $v_{m,k}\left(  t\right)  $, by the second group of
equations of the system. Hence we may replace $v_{m,h}\left(  s\right)  $ by
$\sigma_{h}\cdot\nabla u_{m}\left(  s\right)  $ in the first group of
equations and get the identity%
\begin{equation}
u_{m}\left(  t\right)  =e^{t\mathcal{L}}u_{m}\left(  0\right)  +\int_{0}%
^{t}e^{\left(  t-s\right)  \mathcal{L}}\widetilde{F}_{m}\left(  u\left(
s\right)  \right)  ds+\sum_{k\in K}\int_{0}^{t}e^{\left(  t-s\right)
\mathcal{L}}\sigma_{k}\cdot\nabla u_{m}\left(  s\right)  dW_{s}^{k}%
.\label{modified mild}%
\end{equation}
Therefore $\left(  u_{m}\right)  _{m=1,...,M}$ is a solution of this mild
system, of class (\ref{class}).

Now let us use Theorem 6.10 of \cite{DPZ} (the semigroup $e^{t\mathcal{L}}
$ is of contraction type, being also a Markov semigroup) to deduce that
$u_{m}$ has continuous paths in $W^{2n,2}\left(  \mathbb{R}^{d}\right)  $,
precisely
\[
u_{m}\in L^{2}\left(  \Omega;C\left(  \left[  0,T\right]  ;W^{2n,2}\left(
\mathbb{R}^{d}\right)  \right)  \right)  .
\]
This completes the proof that there exists a solution, unique, with the
regularity specified in Proposition \ref{Proposition regularity nonlinear};
however it is the solution of a modified system, with $\widetilde{F}_{m}$ in
place of $F_{m}$.

\textbf{Step 2}. Let $\left(  u_{m}^{0}\right)  _{m=1,...,M}$ (we use a new
notation to avoid confution) the solution of the original system introduced in
Definition \ref{def:spde-w} and proved to be unique by Corollary
\ref{Coroll uniqueness}. Since each $u_{m}^{0}\left(  t\right)  $ take values
in $\left[  0,U\right]  $, we have $\chi\left(  u_{m}^{0}\left(  t,x\right)
\right)  =u_{m}^{0}\left(  t,x\right)  $ and then $\left(  u_{m}^{0}\right)
_{m=1,...,M}$ is also a solution, in the space%
\begin{equation}
C\left(  \left[  0,T\right]  ;L^{2}\left(  \Omega;L^{2}\left(  \mathbb{R}%
^{d}\right)  \right)  \right)  \cap L^{2}\left(  \left[  0,T\right]
\times\Omega;W^{1,2}\left(  \mathbb{R}^{d}\right)  \right)
\label{basic regularity}%
\end{equation}
of equation (\ref{modified mild}) (originally it is a solution in the weak
sense, but the passage from the weak formulation to the mild formulation, in
the regularity class (\ref{basic regularity}), is standard, see for instance
Proposition 6.3 of \cite{DPZ}). Let $\left(  u_{m}\right)  _{m=1,...,M}$ be
the smoother solution given by Step 1 above. It satisfies the same equation
(\ref{modified mild}), and it has the regularity (\ref{basic regularity}).
Hence $\left(  u_{m}^{0}\right)  _{m=1,...,M}=\left(  u_{m}\right)
_{m=1,...,M}$ (fact that completes the proof of Proposition
\ref{Proposition regularity nonlinear}) if we prove that equation
(\ref{modified mild}) has a unique solution in the class
(\ref{basic regularity}). But the proof of this fact is identical to the one
of Step 1 above, based on the inequality of Lemma
\ref{lemma super ellipticity} which holds (first of all) for $n=0$.

Authors address: Scuola Normale Superiore di Pisa. Piazza Dei Cavalieri 7. Pisa PI 56126. Italia. 

E-mails: \{franco.flandoli, ruojun.huang\}@ sns.it.
\end{document}